\documentclass[11pt]{amsart}
\pdfoutput=1
\usepackage[margin=1in]{geometry}
\usepackage[english]{babel}
\usepackage[utf8]{inputenc}
\usepackage{csquotes}
\usepackage[T1]{fontenc} \usepackage{ae} \usepackage{aecompl}
\usepackage{etoolbox}
\usepackage{enumerate}
\usepackage{enumitem}
\usepackage{tabularx}
\usepackage{graphicx}
\usepackage{amsmath}
\usepackage{amsfonts}
\usepackage{amssymb}
\usepackage{amsthm}
\usepackage{thmtools}
\usepackage{mathtools}
\usepackage{float}
\usepackage{xparse}	
\usepackage[all]{xy}
\usepackage{pdflscape}

\usepackage{hyperref}
\hypersetup{
    breaklinks = true,
    colorlinks = true,
    citecolor = black,
    linkcolor = black,
    urlcolor = black,
}

\usepackage[style=alphabetic, maxbibnames=4]{biblatex}	
\usepackage{xurl}
\newtheorem{thm}{Theorem}[section]
\newtheorem{lema}{Lemma}[section]

\newtheorem{coro}{Corollary}[section]
\newtheorem{conj}{Conjecture}
\declaretheorem[style=definition,name=Definition,numberwithin=section]{Def}
\declaretheorem[style=definition,name=Example,numberwithin=section]{ex}
\declaretheorem[style=definition,name=Remark,numberwithin=section]{rmk}

\newcommand{\C}{\mathbb{C}}
\newcommand{\Q}{\mathbb{Q}}
\newcommand{\Z}{\mathbb{Z}}
\newcommand{\N}{\mathbb{N}}

\newcommand{\Var}{\mathrm{Var}}
\newcommand{\Hom}[2]{\operatorname{Hom}(#1,#2)}
\NewDocumentCommand{\Gr}{O{k} O{n}}{\operatorname{Gr}(#1,#2)}
\newcommand{\Res}{\operatorname{Res}}
\newcommand{\Ind}{\operatorname{Ind}}
\newcommand{\Per}{\operatorname{Per}}

\newcommand{\PGL}{\operatorname{PGL}}
\newcommand{\GL}{\operatorname{GL}}
\newcommand{\SL}{\operatorname{SL}}

\renewcommand{\S}{\mathbb{S}}
\newcommand{\tr}{\operatorname{Tr}}
\newcommand{\A}{\mathbb{A}}
\newcommand{\ov}[1]{\overline{#1}}

\renewcommand{\P}{\mathbb{P}}

\newcommand{\mcB}{\mathcal{B}}
\newcommand{\mcC}{\mathcal{C}}
\newcommand{\mcD}{\mathcal{D}}
\newcommand{\mcF}{\mathcal{F}}
\newcommand{\mcI}{\mathcal{I}}
\newcommand{\mcJ}{\mathcal{J}}
\newcommand{\mcL}{\mathcal{L}}
\newcommand{\mcM}{\mathcal{M}}
\newcommand{\mcN}{\mathcal{N}}
\newcommand{\mcO}{\mathcal{O}}
\newcommand{\mcP}{\mathcal{P}}
\newcommand{\mcR}{\mathcal{R}}
\newcommand{\mcT}{\mathcal{T}}

\title{On the motive of quotients of induced actions}
\author{Lucas de Amorin}
\email{ldamorin@dm.uba.ar}
\address{Departamento de Matemática-IMAS, Facultad de Ciencias Exactas y Naturales, Universidad de Buenos Aires, Argentina}

\begin{document}

\begin{abstract}
    We explore computational tools that allow us to compute the class on the Grothendieck ring of varieties of finite cyclic quotients in some interesting examples. As main application, we determine the motive of low rank representation varieties associated with torus knots and general linear groups using an equivariant analogue of the strategy for special linear groups due to A. González-Prieto and V. Muñoz. 
\end{abstract}

\maketitle

\setcounter{tocdepth}{1}
\tableofcontents

\section{Introduction}

A usual strategy to compute a motive\footnote{By motive we always mean class on the Grothendieck ring of varieties.} is to use cell decompositions and locally trivial fibrations. When dealing with quotients, it becomes harder to apply; usually one needs things to be equivariant. More precisely, a problem that arises is as follows. Take an abelian finite group $\Gamma$ and two $\Gamma$-varieties $X$ and $Y$ and assume one knows the motives of $X/\Gamma$ and $Y/\Gamma$. The objective is to compute the motive of $(X\times Y)/\Gamma$. However, the previous information is not enough in general. In cohomology, one knows that $H^\cdot(X/\Gamma)$ and $H^\cdot(Y/\Gamma)$ are the invariant subspaces of $H^\cdot(X)$ and $H^\cdot(Y)$ under the induced $\Gamma$-action. On the other hand, by Künneth's formula, 
\[H^\cdot((X\times Y)/\Gamma)=\bigoplus_\chi H^\cdot(X)_\chi \otimes H^\cdot(Y)_{\chi^{-1}}\] 
where $\chi$ varies along all irreducible characters of $\Gamma$ and $H^\cdot(X)_\chi$ and $H^\cdot(Y)_\chi$ are the isotypic pieces of $H^\cdot(X)$ and $H^\cdot(Y)$ associated with $\chi$. Hence, one can determine $H^\cdot((X\times Y)/\Gamma)$ from the isotypic decompositions of $H^\cdot(X)$ and $H^\cdot(Y)$. Now at level of motives there is no obvious candidate for what should be the isotypic pieces $[X]_\chi$ and $[Y]_\chi$. 

In his PhD thesis \cite{JesseVogel}, J. Vogel proposes a novel definition. His idea is to look for some class $[X]^\Gamma$ which will be a linear combination of products of motives and representations of $\Gamma$ in such a way one has that, for any subgroup $H$, $[X/H]$ can be recover by adding up all those motives that are together representations where $H$ acts trivially. It is not true that in general one has a Künneth formula
\[[X\times Y]^\Gamma=[X]^\Gamma\times[Y]^\Gamma\]
but it holds if $X$ or $Y$ are built up with linear actions on affine spaces. Hence, one can compute $[(X\times Y)/\Gamma]$ from $[X]^\Gamma$ and $[Y]^\Gamma$ in these cases. Now the question is: how to compute $[X]^\Gamma$? 

Before going on, we should say two things. One is that $[X]^\Gamma$ has less information than the equivariant motive of $X$ but more than just the plain one. Second, J.Vogel's definition needs some non-natural, non-canonical choices, except for cyclic groups where no choice is needed. For that reason we will stick to cyclic groups. But it is very reasonable to expect that similar results should hold for other families like symmetric groups where natural choices can be made.

Let us return to the question of how to compute $[X]^\Gamma$ and assume from now on that $\Gamma$ is cyclic. We do not know a general recipe, but there are three cases we can dealt with. As we will see they cover interesting examples. This first case is actually due to J. Vogel. If $X$ is a linear action on an affine space, $[X]^\Gamma$ is $[X]$ times the trivial representation. The second one is an induced action. Given a bigger cyclic group $\Gamma'\supset \Gamma$, one has a natural action of $\Gamma'$ on $\Ind_\Gamma^{\Gamma'}(X)=(X\times \Gamma')/\Gamma$. As expected, one has 
\[[\Ind_\Gamma^{\Gamma'}(X)]^{\Gamma'}=\Ind_\Gamma^{\Gamma'}([X]^\Gamma)\]
in other words, induction and motivic isotypic decomposition commute. Let us clarify that the second induction means that one changes in $[X]^\Gamma$ each representations of $\Gamma$ by its induced $\Gamma'$ one.

The more interesting case is the third one. As before, let $\Gamma'\supset \Gamma$ be a bigger cyclic group and call $K=\Gamma'/\Gamma$. Then one has a permutation of factors action of $\Gamma'$ on $X^K$ where elements of $\Gamma$ acts by acting in every coordinate and the ones of $\Gamma'\setminus\Gamma$ by shifting coordinates (for a precise definition see section \ref{permutaion-of-factors}). Denote it $\Per_\Gamma^{\Gamma'}(X)$. Our main tool to compute $[\Per_\Gamma^{\Gamma'}(X)]$, or $[\Per_\Gamma^{\Gamma'}(X)]^{\Gamma'}$, from $[X]$, respectively $[X]^\Gamma$, is the following. Note that a priori it is not even clear that the equivariant motive of $\Per_\Gamma^{\Gamma'}(X)$ only depends on the one of $X$. 

\begin{thm}\label{thm:summarize-per-action}
    For any two cyclic groups $\Gamma\subset\Gamma'$ there exists a unique function $\Per_\Gamma^{\Gamma'}:K_0^\Gamma(\Var_\C)\to K_0^{\Gamma'}(\Var_\C)$ such that:
    \begin{enumerate}
        \item $\Per_\Gamma^{\Gamma'}([X])$ is the class of the permutation of factors action on $X^{\Gamma'/\Gamma}$ for any $\Gamma$-variety $X$.
        \item $(\xi_1\xi_2)^K=\xi_1^K\xi_2^K$ for any $\xi_1,\xi_2$. 
        \item For any $\xi_1,\xi_2\in K_0^\Gamma(\Var_\C)$, the following binomial formula holds  
        \[\Per_\Gamma^{\Gamma'}(\xi_1+\xi_2)=\sum_{d|N}\frac{d}{N}\Ind_{\Gamma_d}^{\Gamma'}\left(\sum_{d|d'|N}\left(\Per_\Gamma^{\Gamma_d}\left(\xi_1\right)^{\frac{d'}{d}}+\Per_\Gamma^{\Gamma_d}\left(\xi_2\right)^{\frac{d'}{d}}\right)^{\frac{N}{d'}}\mu\left(\frac{d'}{d}\right)\right)\]
        where $N=|\Gamma'/\Gamma|$, $\Gamma_d$ is the unique index $d$ subgroup of $\Gamma'$, and $\mu$ is the Möbius function.
        \item If $\xi\in A_\Gamma$, $\Per_\Gamma^{\Gamma'}(\xi)\in A_{\Gamma'}$.
    \end{enumerate}
\end{thm}

In the previous result, $A_\Gamma$ denotes the subring of $K_0^\Gamma(\Var_\C)$ generated by classes of the form $[\Ind_H^\Gamma(\A^n)]$ where $H\subset \Gamma$ is a subgroup and $H$ acts linearly on $\A^n$. We additionally determine $\Per_\Gamma^{\Gamma'}([\Ind_H^\Gamma(\A^n)])$ in section \ref{permutaion-of-factors}. Therefore, one can compute $\Per_\Gamma^{\Gamma'}(\xi)$ for any $\xi\in A_\Gamma$.

The main motivation for this article comes from an example where these tools are useful. Namely, character varieties. Its definition is simple; the character variety $\mcM(\Gamma,G)$ is a moduli space of morphisms of a group $\Gamma$ into some reductive group $G$, up to conjugation. It parametrizes representations of $\Gamma$ with structure group $G$. 

Perhaps one of the most studied cases arises in the so-called non-abelian Hodge theory (\cite{SimpsonI}, \cite{SimpsonII}). In this case, $\Gamma$ is the fundamental group of some compact curve $C$ and $\mcM(\Gamma,G)$ is diffeomorphic to the moduli space of Higgs bundles. The cohomology of this space has been object of intense research. Starting with the work of Hitchin \cite{Hitchin} and going up to recent work by many mathematicians (for example \cite{HT-TMS}, \cite{GWZ}, \cite{MS-PW}, \cite{HMMS-PW}), very deep mathematics have been uncover (and some still remains to be uncover). In another direction, there have been applications to three dimensional topology. The pioneer work of Culler and Shallen \cite{CS} shown how simple geometric properties of character varieties, where $\Gamma$ is the fundamental group of a three manifold and $G$ is $\SL_2(F)$ for some field $F$, imply remarkable results as the Smith conjecture. 

In view of this a lot of effort has been put in computing cohomological invariants of character varieties (for example see  \cite{DWW}, \cite{Mellit2} for its Betti numbers, \cite{HRV}, \cite{HLRV1}, \cite{HLRV2}, \cite{MM} for its $E$-polynomials, and \cite{GP}, \cite{TQFT} for its motive). One may expect this to be easier if one has a simple presentation of $\Gamma$. In this direction it has been study what happens for free groups (\cite{FS21}, \cite{free1}, \cite{free2}, \cite{free3}, \cite{free4}) and for torus knots (\cite{su2}, \cite{LMN}, \cite{sl3}, \cite{PM}), i.e. groups with two generators $x,y$ and a unique relation,  $x^n=y^m$ for some coprime integers $n,m$. In the later case, a promising strategy have been develop by A. González-Prieto and V. Muñoz \cite{PM} for $\SL_r$. It has been completed so far for $r\leq 4$. In the present paper we extend their method to $\GL_r$ and make the computations for $r\leq 4$. 

We make three remarks. First, previously known cases for $\GL_r$ are $r=2$ \cite{LMN} and $r=3$ \cite{sl3}. Second, the difficulties for $\SL_r$ with $r\geq 4$ are about to be overcome by A. Calleja (private communication). We expect the same to happen for $\GL_r$. Third, \cite{PM} is based on a core fibration which is actually not well-defined. We suitable correct that. It is worthy to mention that their computations hold verbatim anyway.

Our results are as follows. Given coprime positive integers $n$ and $m$, denote with $\mcR^{irr}(G)$ and $\mcM^{irr}(G)$ the representation and character varieties of the fundamental group of the $(n,m)$ torus knot into $G$ respectively. We work over the complex numbers. 

\begin{thm}\label{coprime-case}
    Let $n,m,r$ be positive integers with $n$ and $m$ coprime. If $r$ is coprime with $n$ and $m$, 
    \[[\mcR^{irr}(\GL_r)]=\frac{q-1}{r}[\mcR^{irr}(\SL_r)]\]
    in $K_0(\Var_\C)$ where $q$ is the class of the affine line.
\end{thm}

We expect that the previous result holds for the character varieties after localization by $q$ and $q^i-1$, $i\geq 1$. Moreover, we prove it for $r\leq 3$ by showing that $[\mcR^{irr}_r(\GL)] = [\PGL_r(\C)]\cdot [\mcM^{irr}_r(\GL)]$. We expect this equality to hold for arbitrary $r$. On the other hand, when $r$ is not coprime with $n$ and $m$, the relation is not as simple. In this case, the computation heavily relies on these new tools. In the following result, note that the class of $\mcM^{irr}(\SL_4)$ is known \cite[Theorem 1.1]{PM}. 

\begin{thm}
    Let $n,m$ be coprime positive integers with $m$ odd. Denote with $q$ the class of the affine line in $K_0(\Var_\C)$. Then
    \begin{align*}
        &\frac{[\mcR^{irr}(\GL_4)]}{[\PGL_4(\C)]}- \frac{q-1}{4}[\mcM^{irr}(\SL_4)]=\delta_{2|n}\left(-
        \frac{(n-2)(m-1)}{16}(q - 1)   (q^5 + 2q^4 - q^3 + 3q^2 + 2q - 2)\right.\\
        &\left.+\frac{n-2}{8}\binom{m-1}{2}(q - 1) q  (q^4 - 2q^2 - 2)-\frac{(n-2)(m-1)}{8}
         (q - 1) q (3q^2 + 2)\right.\\
        &\left.+\frac{n-2}{32}\binom{m-1}{3} (q - 1) q^2 (q^7 + 2q^6 + 4q^5 + 5q^4 - 6q^3 - 6q^2 - 6q - 6)\right.\\
        &\left.-\frac{1}{2}\binom{m-1}{2}  (q-1)q  (q^2 + q + 3)+\frac{1}{8}\binom{m-1}{3}q^4 (q^2 + q - 5)
        \right)\\ 
        &+\delta_{4|n}\left( \frac{1}{128}\binom{m-1}{3} (q - 1) q^2 (q^7 + 2q^6 - 3q^4 + 6q^3 + 6q^2 + 6q + 6)\right.\\
        &\left.+\frac{1}{16}\binom{m-1}{2} q (3q^7 + 3q^6 + q^4 + 3q^3 - 4q^2 + 4q - 4) \right.\\
        &\left.+\frac{m-1}{32}(q^5 - 3q^4 + 4q^3 + q^2 - 4q + 2) +\frac{m-1}{16}(q - 1)q(3  q^2  +2q^2 + 2) \right)
    \end{align*}
    in the localization of $K_0(\Var_\C)$ by the multiplicative set generated by $q$ and $q^i-1$ for $i=1,2,3,4$.
\end{thm}

Let us say a few words about the proof. In \cite{PM}, they make a reduction to work with representations for $(\Z/n\Z) \star (\Z/m\Z)$. In our case, a similar reduction is possible. But two new difficulty appears, essentially because the center of $\SL_r$ is finite, while the one of $\GL_r$ is not. This forces us to not only work with $\mcM((\Z/n\Z)\star (\Z/m\Z),\GL_r)$, but equivariantly with respect to a cyclic group action. At the level of $E$-polynomials, Gónzales-Prieto--Muñoz's is simple to adapt because multiplicative properties for equivariant $E$-polynomials are known. However, in general, the $E$-polynomial does not determine the motive. But we show this is indeed true in our case. A good example of this strategy is subsection \ref{subsec:conjcl}. There we first construct a cell decomposition to show a motive is nice enough. However, it is combinatorially too complicated to actually use it to compute it. But at the level of $E$-polynomials, the computation is simple.

\medskip\noindent\it{Acknowledgements. }\rm Several computations have been check with the help of Sage. The author wishes to thank A. Calleja for introducing him to the subject and to A. González-Prieto and M. Mereb for very useful discussions. This research was supported by a PhD fellowship from CONICET.

\section{Preliminaries}

\subsection{Rational representations of cyclic groups}

We refer the reader to \cite{Serre} for a detailed exposition, here we just summarize what we are going to use. For a finite group $\Gamma$ denote with $R_\Q(\Gamma)$ the Grothendieck ring of its category of finite dimensional representations defined over $\Q$. For any divisor $d$ of $|\Gamma|$, we will denote with $\Q^d$ the permutation action induced by the unique quotient of $\Gamma$ of size $d$. Sometimes we will call $T_\Gamma$ the trivial representation $\Q^1$.

\begin{lema}
    Let $\Gamma$ be a finite cyclic group. Then $R_\Q(\Gamma)$ has a $\Z$-basis given by $\Q^d$, where $d$ runs over all positive divisors of $\Gamma$. 
\end{lema}
\begin{proof}
    All irreducible representations of $\Gamma$ defined over $\Q$ are of the form $\Q[x]/\Phi_d(x)$ where $d||\Gamma|$ and $\Phi_d(x)$ is the cyclotomic polynomial associated with $d$. Hence, we only need to note that the representation $\Q[x]/\Phi_d(x)$ appears with multiplicity one at $\Q^d$ and all other simple representations in $\Q^d$ are of the form $\Q[x]/\Phi_{d'}(x)$ for $d'<d$.
\end{proof}

\begin{lema}
    For any $d,d'|n$,
    \[\Q^d\cdot \Q^{d'}=\gcd(d,d')\Q^{\operatorname{lcm}(d,d')}\]
    in $R_\Q(\Gamma)$.
\end{lema}
\begin{proof}
    Indeed, the basis $e_i\otimes e_j$ of $\Q^d\otimes \Q^{d'}$ has its elements permuted with all orbits of size $\operatorname{lcm}(d,d')$.      
\end{proof}

\begin{lema}
    Let $\Gamma$ be a finite cyclic group. There is a division algorithm in $\Q[q]\otimes R_\Q(\Gamma)$; if $f=gh$, there is an algorithm to determine $h$ in function of $f$ and $g$.
\end{lema}
\begin{proof}
    Uses the basis $\Q^d$. Write $h = \sum h_d\otimes \Q^d$, $f=\sum f_d\otimes \Q^d$ and $g=\sum g_d\otimes \Q^d$. Note that when computing $f_d$ by the product $gh$, one only uses $g_{d'}$ and $h_{d'}$ for $d'|d$. More precisely,
    \[f_d\otimes \Q^d = \sum_{d_1,d_2}\gcd(d_1,d_2) g_{d_1}h_{d_2}\otimes \Q^{d}\]
    where $d_1$ and $d_2$ run over all divisors of $d$ such that $\operatorname{lcm}(d_1,d_2)=d$. Hence,
    \[h_d\left(\sum_{d_1|d}d_1g_{d_1}\right)=f_d -\sum_{d_2<d}\sum_{d_1}\gcd(d_1,d_2) g_{d_1}h_{d_2} \]
    where in the left hand side $d_1$ runs over all divisors of $d$ such that $\operatorname{lcm}=(d_1,d_2)d$. The result follows.
\end{proof}

For a subgroup $H\subset\Gamma$, denote with $\Res^\Gamma_H:R_\Q(\Gamma)\to R_\Q(H)$ and $\Ind_H^\Gamma:R_\Q(H)\to R_\Q(\Gamma)$ the restriction and induction of representations. Recall some basic properties:

\begin{thm}\label{thm:sec2.1-1}
    Let $H\subset \Gamma'\subset \Gamma$ be finite groups. Then
    \begin{enumerate}
        \item $\Res_{H}^\Gamma(\Ind_{H}^\Gamma(V))=\frac{|\Gamma|}{|H|}V$ if $\Gamma$ is abelian,
        \item $\Res_{H}^{\Gamma'}(\Res_{\Gamma'}^\Gamma(V))=\Res_H^\Gamma(V)$, and
        \item $\Ind_{\Gamma'}^\Gamma(\Ind_H^{\Gamma'}(V))=\Ind_H^\Gamma(V)$
    \end{enumerate}
    for any representation $V$.
\end{thm}

\begin{thm}\label{thm:sec2.1-2}
    Let $H_1,H_2\subset \Gamma$ be finite groups. Then
    \begin{enumerate}
        \item $\Ind_{H_1}^\Gamma(V_1)\cdot V=\Ind_{H_1}^\Gamma(V_1\cdot \Res_{H_1}^\Gamma(V))$,
        \item $\Ind_{H_1}^\Gamma(V_1)\cdot\Ind_{H_2}^\Gamma(V_2)=\frac{|\Gamma|}{|H_1H_2|}\Ind_{H_1\cap H_2}^\Gamma(\Res_{H_1\cap H_2}^{H_1}(V_1)\cdot \Res_{H_1\cap H_2}^{H_2}(V_2))$, and
        \item $\Res_{H_1}^{\Gamma}(\Ind_{H_2}^{\Gamma}(V_2))=\frac{|\Gamma||H_1\cap H_2|}{|H_2||H_1|}\Ind_{H_1\cap H_2}^{H_1}(\Res_{H_1\cap  H_2}^{H_2}(V))$ if $\Gamma$ is abelian
    \end{enumerate}
    for any representations $V_1$, $V_2$ and $V$ of $H_1$, $H_2$ and $\Gamma$ respectively.
\end{thm}

There is a paring on $R_\Q(\Gamma)$. It is given by
\[\langle V_1,V_2\rangle = \frac{1}{|\Gamma|}\sum_{\gamma\in \Gamma}\tr(\gamma|V_1)\tr(\gamma^{-1}|V_2)\]
for any two representations $V_1,V_2$ of $\Gamma$. 

\begin{thm}(Frobenius reciprocity theorem)
    Let $H\subset \Gamma$ be finite groups. Then
    \[\langle V_1,\Ind_H^\Gamma(V_2) \rangle_{\Gamma} = \langle \Res_H^\Gamma(V_1),V_2 \rangle_H\]
    for any $V_1\in R_\Q(\Gamma)$ and $V_2\in R_\Q(H)$.
\end{thm}

\subsection{Grothendieck ring of varieties}

Let $\Gamma$ be a finite group. A $\Gamma$-variety is a variety $X$ together with an action of $\Gamma$ such that any point has an equivariant affine neighborhood. The Grothendieck ring of $\Gamma$-varieties, $K_0^\Gamma(\Var_\C)$, is defined as the free abelian group generated by isomorphism classes of $\Gamma$-varieties quotient by all relations $[X]=[Z]+[X\setminus Z]$ for any $\Gamma$-variety $X$ and invariant closed subset $Z\subset X$. Its multiplication is given by $[X][Y]=[X\times Y]$ where $\Gamma$ acts diagonally in $X\times Y$. Denote with $\Res^\Gamma_H:\Gamma-\Var_\C\to H-\Var_\C$ and $\Ind_H^\Gamma:H-\Var_\C\to \Gamma-\Var_\C$ the restriction and induction of actions. They induce $\Res^\Gamma_H:K_0^\Gamma(\Var_\C)\to K_0^H(\Var_\C)$ and $\Ind_H^\Gamma:K_0^H(\Var_\C)\to K_0^\Gamma(\Var_\C)$. Let us proof the analogues of Theorems \ref{thm:sec2.1-1} and \ref{thm:sec2.1-2} in this context.

\begin{lema}
    Let $H\subset \Gamma'\subset \Gamma$ be finite groups, $X$ be a $H$-variety, and $Y$ be a $\Gamma$-variety. Then
    \begin{enumerate}
        \item $[\Res_H^{\Gamma}(\Ind_H^{\Gamma}([H]))]=\frac{|\Gamma|}{|H|}[X]$ in $K_0^H(\Var_\C)$ if $\Gamma$ is abelian, 
        \item $\Res_H^{\Gamma'}(\Res_{\Gamma'}^{\Gamma}(Y))=\Res^{\Gamma}_H(Y)$,
        \item $\Ind_{\Gamma'}^{\Gamma}(\Ind_H^{\Gamma'}(X))\simeq \Ind_H^{\Gamma}(X)$ as $\Gamma$-varieties, and
        \item $\Ind_H^{\Gamma}(X)\times Y \simeq \Ind_H^{\Gamma}(X\times\Res_H^{\Gamma}(Y))$ as $\Gamma$-varieties.
    \end{enumerate}
\end{lema}
\begin{proof}
    For the first item note $\Ind_H^{\Gamma}(X)\simeq X\times \Gamma/H$ as $H$-varieties. The second is clear. For the last two, the isomorphisms are $((x,\gamma'),\gamma)\mapsto (x,\gamma'\gamma)$ and $((x,\gamma),y)\mapsto ((x,\gamma^{-1}y),\gamma)$.
\end{proof}

\begin{lema}\label{ind-times-ind}
    Let $H_1,H_2\subset \Gamma$ be finite abelian groups, $X$ be a $H_1$-variety and $Y$ be a $H_2$-variety. Then
    \[[\Ind_{H_1}^\Gamma(X)]\cdot[\Ind_{H_2}^\Gamma(Y)]=\frac{|\Gamma|}{|H_1H_2|}[\Ind_{H_1\cap H_2}^\Gamma(\Res_{H_1\cap H_2}^{H_1}(X)\cdot \Res_{H_1\cap H_2}^{H_2}(Y))]\]
    in $K_0^\Gamma(\Var_\C)$. In particular,
    $[\Ind_{H_1}^{\Gamma}(X)]\cdot [\Ind_{H_1}^{\Gamma}(Y)] =\frac{|\Gamma|}{|H_1|}[\Ind_{H_1}^{\Gamma}(X\times Y)] $
    for any two $H_1$-varieties $X$ and $Y$.
\end{lema}
\begin{proof}
    Note $\Ind_{H_1}^\Gamma(X)\times\Ind_{H_2}^\Gamma(Y)$ has a cell decomposition indexed by $\Gamma/H_1\times \Gamma/H_2$ where each cell is $H_1\cap H_2$-equivariantly isomorphic to $X\times Y$. By the exact sequence
    \[\Gamma/(H_1\cap H_2)\to\Gamma/H_1\times \Gamma/H_2\to\Gamma/H_1H_2\]
    we see that there are $|\Gamma/H_1H_2|$ orbits on the indexed set each one isomorphic to $\Gamma/(H_1\cap H_2)$.
\end{proof}

\begin{lema}\label{res-ind}
    Let $H_1,H_2\subset \Gamma$ be finite abelian groups and $X$ be a $H_2$-variety. Then
    \[[\Res_{H_1}^\Gamma(\Ind_{H_2}^\Gamma(X))]=\frac{|\Gamma||H_1\cap H_2|}{|H_1||H_2|}[\Ind_{H_1\cap H_2}^{H_1}(\Res_{H_1\cap H_2}^{H_2}(X))]\]
    in $K_0^{H_1}(\Var_\C)$.
\end{lema}
\begin{proof}
    By its very definition $\Ind_{H_2}^\Gamma(X)$ has a cell decomposition indexed by $\Gamma/H_2$ where each cell is only fixed by elements of $H_2$. Hence, the result follows by noticing that the action of $H_1/H_1\cap H_2$ on $\Gamma/H_2$ is free with $\frac{|\Gamma||H_1\cap H_2|}{|H_1||H_2|}$ orbits. 
\end{proof}

\subsection{Möbius inversion} 

For a detailed exposition see \cite{Hall}. The Möbius function is defined as 
\[\mu(n):=\left\{\begin{array}{cc}
    1 & \text{if }n\text{ is square-free and has an even number of prime divisors} \\
    -1 & \text{if }n\text{ is square-free and has an odd number of prime divisors} \\
    0 & \text{if }n\text{ is not square-free}
\end{array}\right.\]
for any positive integer $n$.

\begin{thm}[Möbius Inversion Formula]
    Let $f:\N\to \C$, $g:\N\to \C$ be two functions. Then 
    \[g(n)=\sum_{d|n}f\left(\frac{n}{d}\right)\]
    for any $n$, if and only if
    \[f(n)=\sum_{d|n}\mu(d)g\left(\frac{n}{d}\right)\]
    for any $n$.
\end{thm}

As a particular case one has 
\[\sum_{d|n} \mu(d)=0\]
unless $n=1$, in which case is $1$. The following two lemmas will be used latter on. Its notation has been choose to match the one of Theorem \ref{thm:per-act}.

\begin{lema}\label{moebius-lemma-1}
    Let $e|f_1,\ldots,f_j$ be positive integers. Then
    \[\sum_{e_1,\ldots,e_j}\prod_{l=1}^j\mu\left(\frac{f_l}{e_l}\right) = \left\{\begin{array}{cc}
       0  & \text{ if }f_i\neq f_j\text{ for some }i,j \\
       \mu\left(\frac{f_1}{e}\right)  & \text{ if }f_i=f_j\text{ for all }i,j 
    \end{array}\right.\]
    where the sum is over all tuples $e_1,\ldots,e_j$ such that each $e_l$ divides $f_l$ and $\gcd(e_1,\ldots,e_j)=e$.
\end{lema}
\begin{proof}
    This is well-known for $j=1$. Let $f=\gcd(f_1,\ldots,f_j)$. Note that
    \[\sum_{e'|e|f}\sum_{e_1,\ldots,e_l}\prod_{l=1}^j\mu\left(\frac{f_l}{e_l}\right)=\sum_{e'|e_l|f_l}\prod_{l=1}^j\mu\left(\frac{f_l}{e_l}\right)=\prod_{l=1}^j\left(\sum_{e'|e_l|f_l}\mu\left(\frac{f_l}{e_l}\right)\right)=\prod_{l=1}^j \delta_{f_l,e'}\]
    for any $e'|f$. Hence,
    \[\sum_{e_1,\ldots,e_l}\prod_{l=1}^j\mu\left(\frac{f_l}{e_l}\right) =\mu\left(\frac{f}{e}\right)\prod_{i,j} \delta_{f_l,f_i} \]
    by Möbius inversion.
\end{proof}

\begin{lema}\label{moebius-lemma-2}
    For any positive integers $e|d',k$,
    \[\sum_{d} \mu\left(\frac{d'}{d}\right)=\left\{\begin{array}{cc}
        0 &  \text{if }d'\not|k\\
        \mu\left(\frac{d'}{e}\right) & \text{if }d'|k
    \end{array}\right.\]
    where $d$ runs over all divisors of $d'$ with $\gcd(d,k)=e$.
\end{lema}
\begin{proof}
    If $e$ does not divide $k$, the result is clear. Now 
    \[\sum_{e'|e|\gcd(d',k)}\sum_{d} \mu\left(\frac{d'}{d}\right)=\sum_{e'|d|d'} \mu\left(\frac{d'}{d}\right)=\delta_{e',d'}\]
    for any $e'|\gcd(d,k)$. Hence, the result follows by Möbius inversion. Note that when $e'=\gcd(d',k)$, the condition $e=d'$ is equivalent to $d'|k$.
\end{proof}

\section{Motivic isotypic decompositions}

\subsection{Definition and basic properties}

For a detailed exposition and several examples about the following result/definition see \cite[section 3.6]{JesseVogel}. 

\begin{lema}
    Let $\Gamma$ be a finite cyclic group. Then the map
    \begin{align*}
        \Psi: R_\Q(\Gamma)\otimes\Q \to \bigoplus_{H\subset \Gamma}\Q &,& V \mapsto (\langle T_H,\Res^\Gamma_H(V)\rangle)_{H\subset G}
    \end{align*}
    is a bijection, where $T_H$ denotes the trivial $H$-representation.
\end{lema}
\begin{proof}
    For injectivity see \cite[Lemma 3.6.11]{JesseVogel}. For surjectivity, see Example 3.6.16 in loc.cit.
\end{proof}

\begin{thm}[J. Vogel]
    Let $\Gamma$ be a finite cyclic group. Then for any $\Gamma$-variety $X$ there exists a unique class $[X]^\Gamma\in K_0(\Var_\C)\otimes R_\Q(\Gamma)\otimes\Q$ such that, for any subgroup $H\subset \Gamma$,
    \[ \langle T_H, \Res^\Gamma_H([X]^\Gamma)\rangle=[X/H]\]
    in $K_0(\Var_\C)$. There is an additive extension $[-]^\Gamma:K_0^\Gamma(\Var_\C)\to K_0(\Var_\C)\otimes R_\Q(\Gamma)\otimes\Q$.  
\end{thm}
\begin{proof}
    See \cite[Definition 3.6.12]{JesseVogel}. 
\end{proof}

\begin{lema}
    Let $\Gamma'\subset\Gamma$ be finite cyclic groups. Then
    \[[\Res_{\Gamma'}^\Gamma(\xi)]^{\Gamma'}=\Res_{\Gamma'}^\Gamma([\xi]^{\Gamma})\]
    for any $\xi\in K_0^{\Gamma}(\Var_\C)$.
\end{lema}
\begin{proof}
    Let $H\subset \Gamma'$ be a subgroup. Then,
    \begin{align*}
        \langle T_H, \Res^{\Gamma'}_H(\Res_{\Gamma'}^\Gamma([\xi]^{\Gamma}))\rangle&= \langle T_H, \Res_{H}^{\Gamma}([\xi]^{\Gamma})\rangle= [\xi/H]
    \end{align*}     
    by the definition of $[\xi]^\Gamma$.
\end{proof}

\begin{lema}
    Let $\Gamma'\subset\Gamma$ be finite cyclic groups. Then
    \[[\Ind_{\Gamma'}^\Gamma(\xi)]^\Gamma=\Ind_{\Gamma'}^\Gamma([\xi]^{\Gamma'})\]
    for any $\xi\in K_0^{\Gamma'}(\Var_\C)$.
\end{lema}
\begin{proof}
    Let $H\subset \Gamma$ be a subgroup. By the results of section 2,
    \begin{align*}
        \langle T_H, \Res^\Gamma_H(\Ind_{\Gamma'}^\Gamma([\xi]^{\Gamma'}))\rangle&= \frac{|\Gamma||H\cap\Gamma'|}{|\Gamma'||H|}\langle T_H, \Ind^H_{H\cap\Gamma'}(\Res_{H\cap \Gamma'}^{\Gamma'}([\xi]^{\Gamma'}))\rangle\\
        &= \frac{|\Gamma||H\cap\Gamma'|}{|\Gamma'||H|}\langle \Res_{H\cap \Gamma'}^{H}(T_H), \Res_{H\cap \Gamma'}^{\Gamma'}([\xi]^{\Gamma'})\rangle\\
        &=  \frac{|\Gamma||H\cap\Gamma'|}{|\Gamma'||H|} [\xi/(H\cap \Gamma')]
    \end{align*} 
    On the other hand, 
    \[\Res_{H}^\Gamma(\Ind_{\Gamma'}^\Gamma(\xi))=\frac{|\Gamma||H\cap \Gamma'|}{|H||\Gamma'|}\Ind_{H\cap \Gamma'}^{H}(\Res_{H\cap \Gamma'}^{\Gamma'}(\xi))\]
    by Lemma \ref{res-ind}. Hence,
    \begin{align*}
        [\Ind_{\Gamma'}^\Gamma(\xi)/H] &=\frac{|\Gamma||H\cap \Gamma'|}{|H||\Gamma'|}[\Ind_{H\cap \Gamma'}^{H}(\Res_{H\cap \Gamma'}^{\Gamma'}(\xi))/H]\\
        &=\frac{|\Gamma||H\cap \Gamma'|}{|H||\Gamma'|}[\xi/(H\cap \Gamma')].
    \end{align*}
\end{proof}

Let $A_\Gamma$ be the subring of $K_0^\Gamma(\Var_\C)$ generated by classes of linear actions of $\Gamma$ in finite disjoint unions of affine spaces. By linear we mean that if some $\gamma\in \Gamma$ fixes a connected component $\A^n$, it acts in a linear way in this component. Note that this family of generators is closed by products.

\begin{thm}[J.Vogel]
    Let $\Gamma$ be a finite cyclic group. For any $\xi\in A_\Gamma$ and $\xi'\in K^\Gamma_0(\Var_\C)$, we have
    \[[\xi\xi']^\Gamma=[\xi]^\Gamma [\xi']^\Gamma\]
    in $K_0(\Var_\C)\otimes R_\Q(\Gamma)\otimes\Q$.
\end{thm}
\begin{proof}
    Vogel's theorem \cite[theorem 3.6.19]{JesseVogel} item $(i)$ is the statement for linear actions on affine spaces. To deal with linear actions on disjoint unions note the following. Take a linear action of $\Gamma$ in $X=\sqcup_{i=1}^N\A^{n_i}$. We may assume that $\Gamma\cdot \A^{n_1} = X$. In this case, all the $n_i$ will be equal to some $n$. Furthermore, we have that if $H\subset \Gamma$ is the subgroup of those elements which maps $\A^{n_1}$ to $\A^{n_1}$, then $X = \A^n \times_H \Gamma = \Ind_H^\Gamma(\A^n)$. Furthermore, if $Y$ is a $\Gamma$-variety, 
    $X\times Y = \Ind_H^\Gamma(\A^n)\times Y= \Ind_H^\Gamma(\A^n\times \Res_H^\Gamma(Y))$. Therefore,
    \begin{align*}
        [X\times Y]^\Gamma &= \Ind_H^\Gamma([\A^n\times \Res_H^\Gamma(Y)]^H)\\
        &= \Ind_H^\Gamma([\A^n]^H[ \Res_H^\Gamma(Y)]^H)\\
        &= \Ind_H^\Gamma([\A^n]^H\Res_H^\Gamma([ Y]^\Gamma))\\
        &= \Ind_H^\Gamma([\A^n]^H)[ Y]^\Gamma\\
        &= [X]^\Gamma[Y]^\Gamma
    \end{align*} 
    by the previous lemmas.
\end{proof}

\begin{coro}
    Let $\Gamma$ be a finite cyclic group and denote with $q$ the class of the affine line on $K_0(\Var_\C)$. For any $\xi\in A_\Gamma$, $[\xi]^\Gamma\in \Z[q]\otimes R_\Q(\Gamma)$ and $[\xi/\Gamma]\in \Z[q]$. In addition, for any such $\xi$, its equivariant $E$-polynomial $E^\Gamma(\xi;u,v)$ belongs to $R_\Q(\Gamma)[uv]$ and, under the substitution $q=uv$, $[\xi]^\Gamma=E^\Gamma(\xi;q)$ in $K_0(\Var_\C)\otimes R_\Q(\Gamma)\otimes\Q$.
\end{coro}
\begin{proof}
    First, let $\xi$ be the class of the affine line $\A^1$ with a linear action of $\Gamma$. We know then that $H_c^\cdot(\A^1)$ is supported in degree $2$ and its one-dimensional. Moreover, one has $H_c^2(\A^1)=H^{1,1}(\A^1)$ and it is generated by the class of a point. As any two points define the same class, one has $E^\Gamma(\xi;u,v)=uv\otimes T_\Gamma$. Now, $[\xi/H]=[\xi]$ for any $H\subset \Gamma$ (see \cite[thm 3.6.19]{JesseVogel}). So $[\xi]^\Gamma=q\otimes T_\Gamma$ and the result holds in this case.

    Second, assume the result holds for $\xi$ and $\xi'$. Then it clearly holds for $-\xi$ and $\xi+\xi'$ as everything is additive. It also holds for $\xi\xi'$. Indeed, both the $[-]^\Gamma$ and $E^\Gamma$ are multiplicative in $A_\Gamma$.
    
    Finally, we need to check that if the result holds for $\xi$ and $\Gamma\subset\Gamma'$ is another cyclic group, then the result holds for $\Ind_\Gamma^{\Gamma'}(\xi)$. This amounts to $E^{\Gamma'}(\Ind_\Gamma^{\Gamma'}(\xi)) = \Ind_\Gamma^{\Gamma'}(E^\Gamma(\xi))$. Now recall that, for a $\Gamma$-variety $X$, the underlying variety of $\Ind_\Gamma^{\Gamma'}(X)$ is noting but $\bigsqcup_{\Gamma'/\Gamma}X$. Hence, $H_c^\cdot(\Ind_\Gamma^{\Gamma'}(X))=\bigoplus_{\Gamma'/\Gamma}H_c^\cdot(X)=\Ind_\Gamma^{\Gamma'}(H_c^\cdot(X))$.
\end{proof}

\subsection{Example: fixed rank matrices}

Let us look to a family of classes that live in $A_\Gamma$. We denote $\mcP_n$ the set of all subsets of $\{1,\ldots,n\}$.

\begin{lema}\label{Vnmrv}
    Let $m, n$ be positive integer numbers, $r: \mcP_n\setminus\{\emptyset\}\to \N_0$ be a function and $\Gamma\subset (\S^1)^n$ be a finite cyclic subgroup. Fix $v_1,\ldots, v_t\in\C^m$, $0\leq t\leq n$, such that for any $S\in \mcP_t\setminus\emptyset$ the family $\{v_s:s\in S\}$ has rank $r(S)$. Consider
    \[V_{n,m,r,v}:=\left\{\begin{array}{c}
        g\in M_{n,m}(\C):\text{the }i\text{-th row of }g\text{ is }v_i\text{ for }1\leq i\leq t\text{ and for any } S\in \mcP_n\setminus\emptyset\\\text{ the }|S|\times m\text{ submatrix of }g
        \text{ whose set of rows is }S\text{ has rank }r(S)
    \end{array}\right\} \]
    equipped with the action of $\Gamma$ given by multiplying each row by the corresponding root of unity. Then $[V_{n,m,r,v}]\in A_\Gamma$. In addition, 
    \[ [V_{n,m,r,v}]^\Gamma = \prod_{j=t+1}^n \left(q^{d(\mcP_{j-1}\setminus\mcL\mcI_j)}-\sum_{\emptyset\neq \mcF\subset\mcL\mcI_{j}}(-1)^{|\mcF|+1}q^{d(\mcF\cup (\mcP_{j-1}\setminus\mcL\mcI_j))}\right) \otimes T_\Gamma\]
    where $q$ is the class of the affine line,
    \[\mcL\mcI_j=\{S\in \mcP_{j-1}:r(S\cup\{j\})\neq r(S)\} \]
    for any $1\leq j\leq n$, $r(\emptyset)=0$, and
    \[d(\mcF)=(-1)^{|\mcF|-1}\left(  r\left(\bigcup_{S\in \mcF} S\right) + \sum_{\emptyset\neq \mcF'\subsetneq \mcF}(-1)^{|\mcF'|}d(\mcF')\right)\]
    for any $\emptyset\neq\mcF\subset\mcP_n$ and $d(\emptyset)=m$.
\end{lema}
\begin{proof}
    This just mimics the usual proof of $[\GL_r(\C)]=\prod_{i=0}^{r-1}(q^r-q^i)$. We proceed by induction on $n$. For $n=t$ the result is clear. Assume that $n\geq t+1$ and the result holds for $n-1$.

    Consider the map $\pi:V_{n,m,r,v}\to V_{n-1,m,r',v}$ given by forgetting the last row, where $r'$ is obtained by composition with the obvious inclusion $\mcP_{n-1}\to \mcP_n$. On the other hand, for each $S\in \mcP_{n-1}$ we have the map $\phi: V_{n-1,m,r',v}\to \Gr[r(S)][m]$ given by looking at the span of the rows indexed by $S$. The pullback of the tautological bundle of the grassmaniann gives us a locally trivial vector bundle $\pi_S: W_S\to V_{n-1,m,r',v}$. Moreover, it is equivariantly locally trivial if we equipped the tautological bundle with the action of $\Gamma$ given by multiplying by the corresponding coordinates.

    Notice that $[W_S]=[V_{n-1,m,r',v}]q^{r(S)}\in A_\Gamma$ by inductive hypothesis. In addition, there are natural inclusions $W_S\subset V_{n-1,m,r',v}\times \A^m$. The final observation is that 
    \[ V_{n,m,r,v} = \left(\bigcap_{r(S\cup \{n\})= r(S)} W_S \right)\setminus \left(\bigcup_{r(S\cup\{n\})\neq r(S)} W_S \right)\]
    where $S\subset \mcP_{n-1}$. Therefore it suffices to prove that
    \[\bigcap_{S\in \mcF}W_S\to V_{n-1,m,r',v}\]
    is locally trivial for any $\mcF\subset \mcP_{n-1}$. But again this is a pullback of a tautological bundle. The point is that, for any $\emptyset\neq \mcF\subset\mcP_{n-1}$, there is an algebraic map
    \[V_{n-1,m,r',v} \to \Gr[d(\mcF)][m]\]
    given by taking the intersection of all the vector spaces spanned by rows indexed by $S\in \mcF$, where 
    \[(-1)^{|\mcF|-1}d(\mcF) =  r\left(\bigcup_{S\in \mcF} S\right) + \sum_{\emptyset\neq \mcF'\subsetneq \mcF}(-1)^{|\mcF'|}d(\mcF').\]
    
    The formula for $[V_{n,m,r,v}]$ follows by the inclusion-exclusion principle.
\end{proof}

In the previous result one could have $t=0$, i.e. no fixed rows. In this case we drop the reference to $v$ in the notation $V_{n,m,r,v}$. 

\begin{coro}\label{gl}
    Let $\sigma\in \GL_n(\C)$ be a finite order matrix. Consider the action of $\Gamma=\langle \sigma\rangle$ in $\GL_n(\C)$ by left or right translations. Then $[\GL_n(\C)]\in A_\Gamma$ and
    \[[\GL_n(\C)]^\Gamma = (q^n-1)(q^n-q)\cdots(q^n-q^{n-1})\otimes T_\Gamma.\]
\end{coro}
\begin{proof}
    Let us prove it for right translations. the other case is totally analogous. Since $\sigma$ has finite order, it is diagonalizable. Hence, after a conjugation, we may assume that $\sigma$ is diagonal. Then this follows from the previous lemma.
\end{proof}

\subsection{Permutation of factors action}\label{permutaion-of-factors}

\begin{Def}
    Let $\Gamma\subset\Gamma'$ be finite cyclic groups. Set $K=\Gamma'/\Gamma$ and $N=|K|$. Fix a generator $\sigma$ of $\Gamma'$. For a $\Gamma$-variety $X$, endow $X^K$ with the action given by
    \[\sigma\cdot(x_1,\ldots,x_N)=(\sigma^N\cdot x_N,x_1,\ldots,x_{N-1}).\]
    We call this the permutation of factors action and denote it $\Per_\Gamma^{\Gamma'}(X)$.
\end{Def}

\begin{rmk}
    There is another description of the action which is generator free. View $X^K\subset (\Ind_\Gamma^{\Gamma'}(X))^K$ as tuples $(x_i,\gamma_i)$ with $\gamma_i\Gamma=i$. Then $\Gamma'$ acts as
    \[(\gamma'\cdot(x_i,\gamma_i))_j=(x_{\gamma'^{-1}j},\gamma'\gamma_{\gamma'^{-1}j}).\]
\end{rmk}

\begin{lema}
    Let $\Gamma\subset\Gamma'$ be finite cyclic groups and $X,Y$ be $\Gamma$-varieties. Then $\Per_\Gamma^{\Gamma'}(X)\times \Per_\Gamma^{\Gamma'}(Y)\simeq \Per_\Gamma^{\Gamma'}(X\times Y)$ as $\Gamma'$-varieties. In particular, $\Per_\Gamma^{\Gamma'}(X^n)$ and $\Per_\Gamma^{\Gamma'}(X)^n$ agree for any natural $n$.
\end{lema}
\begin{proof}
    An isomorphism is given by $((x_1,\ldots,x_N),(y_1,\ldots,y_N))\mapsto ((x_1,y_1),\ldots,(x_N,y_N))$.
\end{proof} 

\begin{lema}\label{peraction}
    Let $\Gamma\subset\Gamma'$ be finite cyclic groups, $X$ be a $\Gamma$-variety and $Y\subset X$ an invariant closed subvariety. Call $N=|\Gamma'/\Gamma|$. Then 
    \begin{align*}
        [\Per_\Gamma^{\Gamma'}(X)] &= \sum_{d|N}\frac{d}{N}\Ind_{\Gamma_d}^{\Gamma'}\left(\sum_{d|d'|N}\left(\Per_\Gamma^{\Gamma_d}(X\setminus Y)^{\frac{d'}{d}}+\Per_\Gamma^{\Gamma_d}(Y)^{\frac{d'}{d}}\right)^{\frac{N}{d'}}\mu\left(\frac{d'}{d}\right)\right)
    \end{align*}
    in $K_0^{\Gamma'}(\Var_\C)$ where $\Gamma_d$ is the unique index $d$ subgroup of $\Gamma'$.
\end{lema}
\begin{proof}
    Call $K=\Gamma'/\Gamma$. There is a cell decomposition of $\Per_\Gamma^{\Gamma'}(X)=X^K$ indexed by $\{T,F\}^K$ where each index indicate if its associated coordinate belong to $Y$ or not. Note this decomposes $X^K$ as a disjoint union of $\binom{N}{j}$ copies of $(X\setminus Y)^j\times Y^{N-j}$ for each $0\leq j\leq N$. For a given cell, indexed by $\alpha\in \{T,F\}^K$ with $j$ trues, to be fixed by $\sigma^d$, $d|N$, one needs that $\alpha_{i}=\alpha_{i+d}$ for any $i$. This means that $N|dj$ and there are $\binom{d}{\frac{jd}{N}}$ such indexes. If we look at those whose stabilizer is generated by $\sigma^d$ we get
    \[\Phi(N,j,d):=\sum_{d'|d,N|d'j}\binom{d'}{\frac{jd'}{N}}\mu\left(\frac{d}{d'}\right)=\sum_{d|d'|\gcd(N,j)}\binom{\frac{N}{d'}}{\frac{j}{d'}}\mu\left(\frac{d'}{d}\right)\]
    such $\alpha$'s, where in the second equality we made the changes of variables $(d,d')\mapsto (\frac{N}{d},\frac{N}{d'})$. Fix one of this $\alpha$'s and look at its orbit $\mcO$, which has size $d$. We have
    \[\bigsqcup_{\alpha'\in\mcO}\Per_\Gamma^{\Gamma'}(X)_{\alpha'}=\Ind_{\Gamma_d}^{\Gamma'}(\Per_\Gamma^{\Gamma'}(X)_\alpha)\]
    and, as a $\Gamma_d$-variety,
    \[\Per_\Gamma^{\Gamma'}(X)_\alpha \simeq (\Per_\Gamma^{\Gamma_d}(X\setminus Y)^{K_d})^{\frac{N}{d}-\frac{j}{d}}\times\Per_\Gamma^{\Gamma_d}(Y)^{\frac{j}{d}}\]
    Hence,
    \begin{align*}
        [X^K] &= \sum_{0\leq j \leq  N}\sum_{d|\gcd(N,j)}\frac{d}{N}\Phi(N,j,d) \Ind_{\Gamma_d}^{\Gamma'}(\Per_\Gamma^{\Gamma_d}((X\setminus Y)^{\frac{N}{d}-\frac{j}{d}}\times Y^\frac{j}{d}))\\
        &= \sum_{d|N}\frac{d}{N}\sum_{0\leq j \leq \frac{N}{d}}\Phi(N,dj,d) \Ind_{\Gamma_d}^{\Gamma'}(\Per_\Gamma^{\Gamma_d}((X\setminus Y)^{\frac{N}{d}-j}\times Y^j))
    \end{align*}
    Now, if we expand the binomials in the left hand side of our statement and look how many times appears $\Per_\Gamma^{\Gamma_d}((X\setminus Y)^{\frac{N}{d}-j}\times Y^j)$, we get
    \[ \sum_{d|d'|\gcd(N,dj)}\binom{\frac{N}{d'}}{\frac{j}{d'}}\mu\left(\frac{d'}{d}\right)=\Phi(N,dj,d).\]
\end{proof}

\begin{thm}\label{thm:per-act}
    For any two finite cyclic groups $\Gamma\subset\Gamma'$ there exists a unique function $\Per_\Gamma^{\Gamma'}:K_0^\Gamma(\Var_\C)\to K_0^{\Gamma'}(\Var_\C)$ such that:
    \begin{enumerate}
        \item $\Per_\Gamma^{\Gamma'}([X])$ is the class of the permutation of factors action $\Per_\Gamma^{\Gamma'}(X)$ for any $\Gamma$-variety $X$.
        \item $\Per_\Gamma^{\Gamma'}(\xi_1\xi_2)=\Per_\Gamma^{\Gamma'}(\xi_1)\Per_\Gamma^{\Gamma'}(\xi_2)$ for any $\xi_1,\xi_2$. 
        \item For any $\xi_1,\xi_2\in K_0^\Gamma(\Var_\C)$, the following binomial formula holds  
        \[\Per_\Gamma^{\Gamma'}(\xi_1+\xi_2)=\sum_{d|N}\frac{d}{N}\Ind_{\Gamma_d}^{\Gamma'}\left(\sum_{d|d'|N}\left(\Per_\Gamma^{\Gamma_d}\left(\xi_1\right)^{\frac{d'}{d}}+\Per_\Gamma^{\Gamma_d}\left(\xi_2\right)^{\frac{d'}{d}}\right)^{\frac{N}{d'}}\mu\left(\frac{d'}{d}\right)\right)\]
        where $N=|\Gamma/\Gamma'|$ and $\Gamma_d$ is the unique index $d$ subgroup of $\Gamma'$..
    \end{enumerate}
\end{thm}
\begin{proof}
    We prove this by induction on the number of divisors of $N$. For the base case $N=1$ ($\Gamma=\Gamma'$) there is nothing to be prove. Fix $N>1$ and assume the result holds for all proper divisors of $N$. Consider the subset $R\subset \prod_{m|N} K_0^{\Gamma_m}(\Var_\C)$ of tuples $(\xi_m)$ with $\Res_{\Gamma_d}^{\Gamma_m}(\xi_m)=\xi_d^{\frac{m}{d}}$ for any $d|m|N$. Let us define a sum $*$ on $R$ by
    \[(X*Y)_m:=\sum_{d|m}\frac{d}{m}\Ind_{\Gamma_d}^{\Gamma_m}\left(\sum_{d|d'|m}(X_d^{\frac{d'}{d}}+Y_d^{\frac{d'}{d}})^{\frac{m}{d'}}\mu\left(\frac{d'}{d}\right)\right)\]
    
    We need to check that $X*Y\in R$. This is true because
    \begin{align*}
        \Res_{\Gamma_k}^{\Gamma_m}(X*Y)_m&=\sum_{d|m}\frac{d}{m}\Res_{\Gamma_k}^{\Gamma_m}\Ind_{\Gamma_d}^{\Gamma_m}\left(\sum_{d|d'|m}(X_d^{\frac{d'}{d}}+Y_d^{\frac{d'}{d}})^{\frac{m}{d'}}\mu\left(\frac{d'}{d}\right)\right)\\
        &=\sum_{d|m}\frac{d}{m}\frac{md}{k\gcd(k,d)}\Ind_{\Gamma_{\gcd(k,d)}}^{\Gamma_k}\Res_{\Gamma_{\gcd(k,d)}}^{\Gamma_d}\left(\sum_{d|d'|m}(X_d^{\frac{d'}{d}}+Y_d^{\frac{d'}{d}})^{\frac{m}{d'}}\mu\left(\frac{d'}{d}\right)\right)\\
        &=\sum_{d|m}\frac{\gcd(k,d)}{k}\Ind_{\Gamma_{\gcd(k,d)}}^{\Gamma_m}\left(\sum_{d|d'|m}(X_{\gcd(k,d)}^{\frac{d'}{\gcd(k,d)}}+Y_{\gcd(k,d)}^{\frac{d'}{\gcd(k,d)}})^{\frac{m}{d'}}\mu\left(\frac{d'}{d}\right)\right)
    \end{align*}
    as $X,Y\in R$. Rearranging,
    \begin{align*}
        \Res_{\Gamma_k}^{\Gamma_m}(X*Y)_m&=\sum_{e|k}\frac{e}{k}\Ind_{\Gamma_{e}}^{\Gamma_m}\left(\sum_{d|m,\gcd(d,k)=e}\sum_{d|d'|m}(X_{e}^{\frac{d'}{e}}+Y_{e}^{\frac{d'}{e}})^{\frac{m}{d'}}\mu\left(\frac{d'}{d}\right)\right)\\
        &=\sum_{e|k}\frac{e}{m}\Ind_{\Gamma_{e}}^{\Gamma_m}\left(\sum_{e|d'|m}\sum_{d|d',\gcd(d,k)=e}(X_{e}^{\frac{d'}{e}}+Y_{e}^{\frac{d'}{e}})^{\frac{m}{d'}}\mu\left(\frac{d'}{d}\right)\right)\\
        &=\sum_{e|k}\frac{e}{m}\Ind_{\Gamma_{e}}^{\Gamma_m}\left(\sum_{e|d'|k}(X_{e}^{\frac{d'}{e}}+Y_{e}^{\frac{d'}{e}})^{\frac{m}{d'}}\mu\left(\frac{d'}{e}\right)\right)
    \end{align*}
    for any $k|m|N$ by Lemma \ref{moebius-lemma-2}. On the other hand, $(X*Y)_d^{j}$ is equal to
    \begin{align*}
        &=\left(\sum_{e|d}\frac{e}{d}\Ind_{\Gamma_{e}}^{\Gamma_d}\left(\sum_{e|f|d}(X_{e}^{\frac{f}{e}}+Y_{e}^{\frac{f}{e}})^{\frac{d}{f}}\mu\left(\frac{f}{e}\right)\right)\right)^{j}\\
        &= \sum_{e_l|d}\prod_{i=1}^{j}\frac{e_l}{d}\Ind_{\Gamma_{e_l}}^{\Gamma_d}\left(\sum_{e_l|f|d}(X_{e_l}^{\frac{f}{e_l}}+Y_{e_l}^{\frac{f}{e_l}})^{\frac{d}{f}}\mu\left(\frac{f}{e_l}\right)\right)\\
        &= \sum_{e_l|d}\left(\prod_{l=2}^{j}\frac{d}{\operatorname{lcm}(\gcd(e_1,\ldots,e_{l-1}),e_l)}\right)\Ind_{\Gamma_{\gcd(e_l)}}^{\Gamma_d}\left(\prod_{i=1}^{j}\frac{e_l}{d}\sum_{e_l|f|d}(X_{\gcd(e_l)}^{\frac{f}{\gcd(e_l)}}+Y_{\gcd(e_l)}^{\frac{f}{\gcd(e_l)}})^{\frac{d}{f}}\mu\left(\frac{f}{e_l}\right)\right)\\
        &= \sum_{e_l|d}\frac{\gcd(e_l)}{d}\Ind_{\Gamma_{\gcd(e_l)}}^{\Gamma_d}\left(\prod_{i=1}^{j}\sum_{e_l|f|d}(X_{\gcd(e_l)}^{\frac{f}{\gcd(e_l)}}+Y_{\gcd(e_l)}^{\frac{f}{\gcd(e_l)}})^{\frac{d}{f}}\mu\left(\frac{f}{e_l}\right)\right)
    \end{align*}
    by Lemma \ref{ind-times-ind}. Then, 
    \begin{align*}
        (X*Z)_d^{j}&= \sum_{e|d}\frac{e}{d}\Ind_{\Gamma_e}^{\Gamma_d}\left(\sum_{e|f_l|d}\left(\sum_{\gcd(e_l)=e,e_l|f_l}\prod_{l=1}^j\mu\left(\frac{f_l}{e_l}\right)\right)\prod_{l=1}^j(X_{e}^{\frac{f_l}{e}}+Y_{e}^{\frac{f_l}{e}})^{\frac{d}{f_l}}\right)\\
        &= \sum_{e|d}\frac{e}{d}\Ind_{\Gamma_e}^{\Gamma_d}\left(\sum_{e|f|d}\mu\left(\frac{f}{e}\right)(X_{e}^{\frac{f}{e}}+Y_{e}^{\frac{f}{e}})^{\frac{dj}{f}}\right)
    \end{align*}
    by Lemma \ref{moebius-lemma-1}, and, therefore, $\Res_{\Gamma_k}^{\Gamma_m}(X*Y)_m= (X*Y)_k^{\frac{m}{k}}$ as desired.
    
    The operation has a neutral element $(0)$ and inverses. It is clearly commutative. For associativity, note the $m$-th coordinate of $(X*Z)*P$ is
    \begin{align*}
        \sum_{d|m}\frac{d}{m}\Ind_{\Gamma_d}^{\Gamma_m}\left(\sum_{d|d'|m}((X*Z)_d^{\frac{d'}{d}}+P_d^{\frac{d'}{d}})^{\frac{m}{d'}}\mu\left(\frac{d'}{d}\right)\right)
    \end{align*}
    By previous computations
    \begin{align*}
        ((X*Z)_d^{\frac{d'}{d}}+P_d^{\frac{d'}{d}})^{\frac{m}{d'}} &= \sum_{0\leq j\leq \frac{m}{d'}} \binom{\frac{m}{d'}}{j}(X*Z)_d^{\frac{d'}{d}j}P_d^{\frac{d'}{d}(\frac{m}{d'}-j)}\\
        &=\sum_{e|d}\frac{e}{d}\Ind_{\Gamma_e}^{\Gamma_d}\left( \sum_{e|f|d}\mu\left(\frac{f}{e}\right) \sum_{0\leq j\leq \frac{m}{d'}}\binom{\frac{m}{d'}}{j}(X_e^{\frac{f}{e}}+Z_e^{\frac{f}{e}})^{\frac{d'}{f}j}P_e^{\frac{d'}{e}(\frac{m}{d'}-j)}\right)
    \end{align*}
    and the $m$-th coordinate of $(X*Z)*P$ is
    \begin{align*}
        &\sum_{d|m}\sum_{d|d'|m}
        \sum_{e|d}\frac{e}{m}\Ind_{\Gamma_e}^{\Gamma_m}\left( \sum_{e|f|d}\mu\left(\frac{f}{e}\right) \sum_{0\leq j\leq \frac{m}{d'}}\binom{\frac{m}{d'}}{j}(X_e^{\frac{f}{e}}+Z_e^{\frac{f}{e}})^{\frac{d'}{f}j}P_e^{\frac{d'}{e}(\frac{m}{d'}-j)}\right)
        \mu\left(\frac{d'}{d}\right)
    \end{align*}
    Rearranging the summations we obtain
    \begin{align*}
        &\sum_{e|m}\frac{e}{m}\sum_{e|d'|m}\sum_{0\leq j\leq \frac{m}{d'}}\binom{\frac{m}{d'}}{j}\Ind_{\Gamma_e}^{\Gamma_m}\left( \sum_{e|f|d'}\left(\sum_{f|d|d'}\mu\left(\frac{d'}{d}\right)\right)\mu\left(\frac{f}{e}\right) (X_e^{\frac{f}{e}}+Z_e^{\frac{f}{e}})^{\frac{d'}{f}j}P_e^{\frac{d'}{e}(\frac{m}{d'}-j)}\right)
    \end{align*}
    By Möbius inversion this becomes
    \begin{align*}
        &\sum_{e|m}\frac{e}{m}\sum_{e|d'|m}\mu\left(\frac{d'}{e}\right)\Ind_{\Gamma_e}^{\Gamma_m}\left( \sum_{0\leq j\leq \frac{m}{d'}}\binom{\frac{m}{d'}}{j} (X_e^{\frac{d'}{e}}+Z_e^{\frac{d'}{e}})^{j}P_e^{\frac{d'}{e}(\frac{m}{d'}-j)}\right)\\
        &=\sum_{e|m}\frac{e}{m}\sum_{e|d'|m}\mu\left(\frac{d'}{e}\right)\Ind_{\Gamma_e}^{\Gamma_m}\left( \left((X_e^{\frac{d'}{e}}+Z_e^{\frac{d'}{e}}+P_e^{\frac{d'}{e}}\right)^{\frac{m}{d'}}\right)
    \end{align*}
    which agrees with $(X*(Z*P))_m$ by similar arguments.

    Therefore, by the previous lemma the function $X\mapsto (\Per_\Gamma^{\Gamma_m}(X))_{m|N}$ defines an additive morphism $(K_0^\Gamma(\Var_\C),+)\to (R,*)$. Composing with the projection $\prod_{m|N}K_0^{\Gamma_m}(\Var_\C)\to K_0^{\Gamma'}(\Var_\C)$ we get the existence of the desired function which satisfies $(1)$ and $(3)$.

    For $(2)$, recall that $\Per_\Gamma^{\Gamma'}(X\times Y)=\Per_\Gamma^{\Gamma'}(X)\times \Per_\Gamma^{\Gamma'}(Y)$ for any two $\Gamma$-varieties $X,Y$. Hence, it is enough to prove that $(R,*)$ becomes a ring with the coordinate wise product $\cdot$. That $R$ is closed under $\cdot$ is clear. We compute the $m$-th coordinate of $(X*Y)\cdot Z$
    \begin{align*}
        \sum_{d|m}\frac{d}{m}\Ind_{\Gamma_d}^{\Gamma_m}\left(\sum_{d|d'|m}(X_d^{\frac{d'}{d}}+Y_d^{\frac{d'}{d}})^{\frac{m}{d'}}\mu\left(\frac{d'}{d}\right)\right) Z_m = \sum_{d|m}\frac{d}{m}\Ind_{\Gamma_d}^{\Gamma_m}\left(\sum_{d|d'|m}(X_d^{\frac{d'}{d}}+Y_d^{\frac{d'}{d}})^{\frac{m}{d'}}Z_d^{\frac{m}{d}}\mu\left(\frac{d'}{d}\right)\right)
    \end{align*}
    and the one of $(X\cdot Z)*(Y\cdot Z)$
    \begin{align*}
        \sum_{d|m}\frac{d}{m}\Ind_{\Gamma_d}^{\Gamma_m}\left(\sum_{d|d'|m}((X_dZ_d)^{\frac{d'}{d}} +(Y_dZ_d)^{\frac{d'}{d}})^{\frac{m}{d'}}\mu\left(\frac{d'}{d}\right)\right) = \\
        \sum_{d|m}\frac{d}{m}\Ind_{\Gamma_d}^{\Gamma_m}\left(\sum_{d|d'|m}(X_d^{\frac{d'}{d}}+Y_d^{\frac{d'}{d}})^{\frac{m}{d'}}Z_d^{\frac{m}{d}}\mu\left(\frac{d'}{d}\right)\right)
    \end{align*}
    Hence, they agree and $(R,*,\dot)$ is a commutative ring. 
\end{proof}

The following corollary finishes the proof of Theorem \ref{thm:summarize-per-action}.

\begin{coro}\label{peract-on-niceclass}
    Let $\Gamma\subset\Gamma'$ be finite cyclic groups. Call $N=|\Gamma'/\Gamma|$. Take $\xi\in A_\Gamma$. Then $\Per_\Gamma^{\Gamma'}(\xi)\in A_{\Gamma'}$.
\end{coro}
\begin{proof}
    It is enough to check:
    \begin{enumerate}
        \item The result holds $\Ind_H^\Gamma(\A^n)$ for a subgroup $H\subset \Gamma$ and a linear action of $H$ on $\A^n$.
        \item If the result holds for $\xi$ and $\xi'$, the result holds for:
            \begin{enumerate}
                \item $\xi\xi'$, and
                \item $\xi+\xi'$
            \end{enumerate}
        \item If the result holds for $\xi$, it holds for $-\xi$
    \end{enumerate}

    For $(1)$, let $X=\Ind_H^\Gamma(\A^n)$ for some affine space with a linear $H$-action. Call $K=\Gamma'/\Gamma$. Then $(\Ind_\Gamma^{\Gamma'}(X))^K=(\Ind_H^{\Gamma'}(\A^n))^K$. Recall that $\Per_\Gamma^{\Gamma'}(X)$ consist of points $(x_i,\gamma_i)$ with $\gamma_i\Gamma=i$ and the $\Gamma'$-action is given by
    \[(\gamma'\cdot (x_i,\gamma_i))_j=(x_{\gamma'^{-1}j},\gamma'\gamma_{\gamma'^{-1}j})\]
    Note $\Per_\Gamma^{\Gamma'}(X)=(\bigsqcup_{\Gamma/H}\A^n)^K$ has a cell decomposition indexed by $\prod_{i\in K} i\Gamma/H$. It is given by $(x_i,\gamma_i)\mapsto (\gamma_iH)$. The action of $\Gamma'$ moves the cells in a compatible way with the permutation of factors action on $\prod_{i\in K} i\Gamma/H$. Note this last action is nothing but $\Per_H^\Gamma(\Gamma/H)$ via $\Ind_\Gamma^{\Gamma'}(\Gamma/H)\simeq \Gamma'/H$, $(\gamma H,\gamma')\mapsto \gamma'\gamma H$.

    Take an element $(\gamma_i)\in \prod_{i\in K} i\Gamma/H$ and assume that $\gamma'\cdot(\gamma_i)=(\gamma_i)$. This means that $\gamma_jH=\gamma'\gamma_{\gamma'^{-1}j}H$ for all $j$. Note the induced action on $\A^{nN}$ is 
    \[(\gamma'\cdot(x_i))_j=((\gamma'\gamma_{\gamma'^{-1}j}\gamma_j^{-1})x_{\gamma'^{-1}j})\]
    which is linear as $H$ was acting in a linear way. So $\Per_\Gamma^{\Gamma'}(X)\in A_{\Gamma'}$. Moreover, $[\Per_\Gamma^{\Gamma'}(X)]^{\Gamma'}=q^{nN}\otimes \Q^{\Per_\Gamma^{\Gamma'}(\Gamma/H)}$. The result holds in this case.

    For $(2)$ and $(3)$ use that $\Per_\Gamma^{\Gamma'}(-)$ is multiplicative and satisfies the binomial formula.
\end{proof}

\begin{lema}\label{classperact}
    Let $\Gamma\subset \Gamma'$ be finite cyclic groups and $N=|\Gamma'/\Gamma|$. For each divisor $d$ of $N$, call $\Gamma_d$ the unique index $d$ subgroups of $\Gamma'$. There are unique functions $\Omega_d: R_\Q(\Gamma)[q]\to R_\Q(\Gamma_d)[q]$, $d|N$, such that
    \begin{enumerate}
        \item $\Omega_d(q^n\otimes\Q^{\Gamma/H})=q^{nd}\otimes\Q^{\Per_\Gamma^{\Gamma_d}(\Gamma/H)}$ for any subgroup $H\subset \Gamma$,
        \item $\Omega_d(V_1V_2)=\Omega_d(V_1)\Omega_d(V_2)$ for any $V_1,V_2$,
        \item the binomial formula holds  
        \[\Omega_m(V_1+V_2)=\sum_{d|m}\frac{d}{m}\Ind_{\Gamma_d}^{\Gamma_m}\left(\sum_{d|d'|m}(\Omega_d(V_1)^{\frac{d'}{d}}+\Omega_d(V_2)^{\frac{d'}{d}})^{\frac{m}{d'}}\mu\left(\frac{d'}{d}\right)\right)\]
        for any $V_1,V_2\in R_\Q(\Gamma)[q]$, and
        \item for any $\xi\in A_\Gamma$,
        \[[\Per_\Gamma^{\Gamma'}(\xi)]^{\Gamma'}=\Omega_N(E^\Gamma(\xi;q))\]
    \end{enumerate}
\end{lema}
\begin{proof}
    Use the maps $R_\Q(\Gamma)[q]\to A_\Gamma$, $q^n\otimes \Q^{\Gamma/H}\mapsto \Ind_H^\Gamma(\A^n)$, and $E^\Gamma(-;q):A_\Gamma\to R_\Q(\Gamma)[q]$ which are mutual ring inverses.    
\end{proof}

\begin{lema}
    Let $\Gamma\subset\Gamma'$ be finite cyclic groups. Assume that $N=|\Gamma'/\Gamma|$ is prime. Then
    \[\Omega_N(p\otimes\Q^{\Gamma/H})=p(q^N)\otimes\Q^{\Per_\Gamma^{\Gamma'}(\Gamma/H)}+ \left(\frac{|\Gamma|}{|H|}\right)^{N-1}\frac{1}{N}(p^N-p(q^N))\otimes \Q^{\Gamma'/H}\]
    for any $p\in\Q[q]$ and subgroup $H\subset \Gamma$.
\end{lema}
\begin{proof}
    Both sides agree for $p=q^n$. Hence, we only need to show that the right hand side satisfies the binomial formula. That is
    \begin{align*}
        &(p_1+p_2)(q^N)\otimes\Q^{\Per_\Gamma^{\Gamma'}(\Gamma/H)}+ \left(\frac{|\Gamma|}{|H|}\right)^{N-1}\frac{1}{N}((p_1+p_2)^N-(p_1+p_2)(q^N))\otimes \Q^{\Gamma'/H}\\
        &= p_1(q^N)\otimes\Q^{\Per_\Gamma^{\Gamma'}(\Gamma/H)}+\left(\frac{|\Gamma|}{|H|}\right)^{N-1} \frac{1}{N}(p_1^N-p_1(q^N))\otimes \Q^{\Gamma'/H}+p_2(q^N)\otimes\Q^{\Per_\Gamma^{\Gamma'}(\Gamma/H)}\\
        &+\left(\frac{|\Gamma|}{|H|}\right)^{N-1} \frac{1}{N}(p_2^N-p_2(q^N))\otimes \Q^{\Gamma'/H}\\
        &+\frac{1}{N}\Ind_\Gamma^{\Gamma'}(((p_1\otimes\Q^{\Gamma/H}+p_2\otimes\Q^{\Gamma/H})^N-(p_1\otimes\Q^{\Gamma/H})^N-(p_2\otimes\Q^{\Gamma/H})^N))
    \end{align*}
    which is true as 
    \[\Ind_\Gamma^{\Gamma'}(\Q^{\Gamma/H})^N=\left(\frac{|\Gamma|}{|H|}\right)^{N-1}\Q^{\Gamma'/H}\]
    by Lemma \ref{ind-times-ind}.
\end{proof}

\begin{coro}\label{pfinzq}
    Let $\Gamma$ be a finite cyclic group of size $N$ prime. Take $\xi\in \Z[q]$. Then $\Per_{\{e\}}^{\Gamma}(\xi)\in A_{\Gamma}$ and 
    \[[\Per_{\{e\}}^{\Gamma}(\xi)]^{\Gamma}=E(\xi;q^{N})+\frac{1}{N}\Ind_{\{e\}}^{\Gamma}(E(\xi;q)^{N}-E(\xi;q^{N})))\]
\end{coro}

\begin{lema}
    Let $H\subset \Gamma\subset\Gamma'$ be finite cyclic groups. Call $N=|\Gamma'/\Gamma|$ and $h=|\Gamma/H|$. Then
    \[\Q^{\Per_\Gamma^{\Gamma'}(\Gamma/H)}=\sum_{M|d|N} \frac{1}{hd}\left(\sum_{M|d'|d}h^{d'}\mu\left(\frac{d}{d'}\right)\right)\Q^{hd}\]
    where $M$ is the product of all prime divisors $p$ of $h$ elevated to the $p$-adic valuation of $N$.
\end{lema}
\begin{proof}
    Let us describe all orbits in $\Per_\Gamma^{\Gamma'}(\Gamma/H)$. Equivalently, its stabilizers. Fix a generator $\sigma$ of $\Gamma'/H$. Note that a point $(x_1,\ldots,x_N)$ is fixed by $\sigma^{kN}$ if and only if $h=|\Gamma/H|$ divides $k$. On the other hand, it is fixed by $\sigma^{l+kN}$, $0< l< N$, $0\leq k< h$, if
    \[(x_1,\ldots,x_N)=(\sigma^{N(k+1)}x_{N-l+1},\ldots, \sigma^{N(k+1)}x_N,\sigma^{Nk}x_1,\ldots,\sigma^{Nk}x_{N-l})\]
    That is $x_i=\sigma^{Nk}x_{i-l}$ for all $i> l$ and $x_i=\sigma^{N(k+1)}x_{N-l+i}$ for $1\leq i\leq l$. Write $N=r+ls$ for $0\leq r<l$, $0\leq s$. Hence, 
    \[x_i=\sigma^{N(k+1)}x_{N-l+i}=\sigma^{N(k+1)}x_{r+l(s-1)+i}=\sigma^{N(ks+1)}x_{r+i}=\left\{\begin{array}{cc}
        \sigma^{N(ks+1)}x_{r+i} & \text{if }i\leq l-r \\
        \sigma^{N(k(s+1)+1)}x_{r+i-l} & \text{if }i>l-r 
    \end{array}\right.\]
    for $1\leq i\leq l$. Let $r=r'\gcd(r,l)$ and $l=l'\gcd(r,l)$. Note that the set $A(i)=\{i+rj:j\in\Z\}$ has $l'$ residues modulo $l$. Hence, one gets
    \[x_i = \sigma^{N(l'(ks+1)+t(i)k)}x_i\]
    where $t(i)$ is the number of residues of $A(i)$ bigger than $l-r$. Hence, we need
    \[h | l'(ks+1)+t(i)k\]
    for all $i$. If this happens there are exactly $|\Gamma/H|^{\gcd(r,l)}$ points stabilized by $\sigma^{l+kN}$. 

    Now the number of residues of $A(i)$ agrees with the ones of $B(i)=\{i+\gcd(r,l)j:j\in \Z\}$. Hence $t(i)=r'$ and the required condition is
    \begin{align*}
        h\gcd(l,r)&|l(ks+1)+rk\\
        h\gcd(l,N)&|kN+l    
    \end{align*}

    We want to find all divisors of $hN$ which can be written as $kN+l$ satisfying the previous condition. We know these divisors need to be of the form $hd$ for some $d|N$ but $ N\not |d\gcd(h,N)$. Note $l$ is the residue of $hd$ modulo $N$ and the condition is $\gcd(hd,N)|d$. Equivalently $\gcd(h,\frac{N}{d})=1$. So if $M$ the part of $N$ which is not coprime with $h$, $d$ satisfies $M|d|N$ and $d<N$. Any such $d$ has $h^d$ fixed points. Hence, 
    \[\Q^{\Per_\Gamma^{\Gamma'}(\Gamma/H)}=\sum_{M|d|N} \frac{1}{hd}\left(\sum_{M|d'|d}h^{d'}\mu\left(\frac{d}{d'}\right)\right)\Q^{hd}\]
\end{proof}

\begin{ex}\label{binom2to4}
    Let us explore $\Gamma'=\mu_4$ and $\Gamma=\mu_2$ the cyclic groups of size $4$ and $2$ respectively. Then $R_\Q(\Gamma)$ has basis $\Q^1,\Q^2$ for which $\Omega_2(\Q^1)=\Q^1$ and $\Omega_2(\Q^2)=\Q^4$. Binomial formula reads:
    \begin{align*}
        \Omega_2(V_1+V_2)&=\Omega_2(V_1)+\Omega_2(V_2)+\frac{1}{2}\Ind_{\mu_2}^{\mu_4}((V_1+V_2)^2-V_1^2-V_2^2)\\
        &=\Omega_2(V_1)+\Omega_2(V_2)+\Ind_{\mu_2}^{\mu_4}(V_1V_2)
    \end{align*}
    Hence,
    \begin{align*}
        \Omega_2(p_1\otimes\Q^1+p_2\otimes\Q^2)&=\Omega_2(p_1\otimes\Q^1)+\Omega_2(p_2\otimes\Q^2)+ p_1p_2\otimes\Q^4\\
        &=p_1(q^2)\otimes\Q^1+\frac{1}{2}(p_1^2-p_1(q^2))\otimes\Q^2+ p_2(q^2)\otimes\Q^4\\&+2\frac{1}{2}(p_2^2-p_2(q^2))\otimes\Q^4+p_1p_2\otimes\Q^4\\
        &= p_1(q^2)\otimes\Q^1+\frac{1}{2}(p_1^2-p_1(q^2))\otimes\Q^2+(p_2^2+p_1p_2)\otimes\Q^4
    \end{align*}
    for any $p_1,p_2\in \Z[q]$.
\end{ex}

\begin{lema}\label{permutation-loc-trivial-fib}
    Let $\Gamma\subset \Gamma'$ be finite cyclic groups and $\pi: Y\to X$ be a $\Gamma$-equivariantly locally trivial fiber bundle with fiber $F$. Let $Z\to \Per_\Gamma^{\Gamma'}(X)$ be a $\Gamma'$-equivariant morphism. Then
    \[ [\Per_\Gamma^{\Gamma'}(Y)\times_{\Per_\Gamma^{\Gamma'}(X)} Z] = [Z]\cdot[\Per_\Gamma^{\Gamma'}(F)] \]
    in $K_0^{\Gamma'}(\Var_\C)$.
\end{lema}
\begin{proof}
    Being $\pi$ equivariantly locally trivial, there is a equivariant cell decomposition $X=\bigsqcup X_\alpha$ by locally closed subvarieties such that $\pi$ is trivial over each cell. Hence, $Y=\bigsqcup X_\alpha \times F$. Note that $\Per_\Gamma^{\Gamma'}(X_\alpha\times F)\simeq \Per_\Gamma^{\Gamma'}(X_\alpha)\times \Per_\Gamma^{\Gamma'}(F)$ and 
    \[ \Ind_{\Gamma}^{\Gamma'}((Z_1\times F)^j\times (Z_2\times F)^{N-j})\simeq \Ind_{\Gamma}^{\Gamma'}(Z_1^j\times Z_2^{N-j})\times \Per_\Gamma^{\Gamma'}(F)\]
    where the last isomorphism is 
    \[((z,f_1,\ldots,f_j),(z_2,f_{j+1},\ldots,f_N),\gamma')\mapsto ((z_1,z_2,\gamma'),\gamma'\cdot(f_{1},f_{2},\ldots,f_n)).\]
    Therefore, by the proof of Lemma \ref{peraction}, there is a cell decomposition of $\Per_\Gamma^{\Gamma'}(X)$ such that $\Per_\Gamma^{\Gamma'}(Y)\to \Per_\Gamma^{\Gamma'}(X)$ is just the multiplication by $\Per_\Gamma^{\Gamma'}(F)$ over each cell. Hence, the same holds for the base change to $Z$.
\end{proof}

\subsection{Example: conjugacy classes in the general linear group}\label{subsec:conjcl}

\begin{lema}
    Let $\lambda_1,\ldots, \lambda_k$ be the sizes of a partition of $\{1,\ldots,n\}$ in subsets and consider the subgroup $H=\prod \GL_{\lambda_i}(\C)$ of $\GL_n(\C)$. Let $\sigma$ be a permutation of $\{1,\ldots,n\}$ which fixes the previous partition. Then the action of $\sigma$ by right translation on $\GL_n(\C)$ descent to $\GL_n(\C)/H$ and defines a class which lives in $A_{\langle \sigma\rangle}$.
\end{lema}
\begin{proof}
    First of all, note that the action of $\Gamma=\langle \sigma\rangle$ by right translations on $\GL_n(\C)$ descents to $\GL_n(\C)/H$ because $gh\sigma =g\sigma (\sigma^{-1}h\sigma)$ and $\sigma^{-1}H\sigma = H$. 

    Up to conjugation, we may assume that $\lambda_1\leq \lambda_2\leq \ldots \leq \lambda_k$. Define $\Psi: \GL_n(\C) \to \C^{\binom{n}{\lambda_1}}\times \cdots \times \C^{\binom{n}{\lambda_k}}$ as follows. The first $\binom{n}{\lambda_1}$ coordinates of $\Psi(g)$ are given by the determinants of the square $\lambda_1\times \lambda_1$ submatrices of $g$ built with the first $\lambda_1$ columns of $g$. The next $\binom{n}{\lambda_2}$ coordinates are the determinants of size $\lambda_2$ square submatrices of $g$ built with the next $\lambda_2$ columns. And so on. For example, for $n=3$, $\lambda_1=2$ and $\lambda_2 = 1$, one has
    \[\Psi\left(\begin{array}{ccc}
        a & b & c \\
        d & e & f \\
        g & h & k
    \end{array} \right) = (ae-bd,dh-eg,ah-bg,c,f,k).\]
    
    Note that $\Psi$ is $H$-invariant and, therefore, defines $\ov{\Psi}:\GL_n(\C)/H\to \C^{\binom{n}{\lambda_1}}\times \cdots \times \C^{\binom{n}{\lambda_k}}$. On the other hand, there is a cell decomposition $\{X_\omega\}$ of $\C^{\binom{n}{\lambda_1}}\times \cdots \times \C^{\binom{n}{\lambda_k}}$ given by the vanishing or not of each coordinate. Index this decomposition by $\omega\in \mcI :=\{0,1\}^{\binom{n}{\lambda_1}}\times\cdots\{0,1\}^{\binom{n}{\lambda_k}}$ where a $1$ indicates indicates that the associated coordinated vanishes and the order is given by $0<1$. By taking preimages we get a decomposition
    \[ \GL_n(\C)/H = \bigsqcup_\omega \ov{\Psi}^{-1}(X_\omega). \]

    This decomposition is not $\sigma$-invariant. Indeed, $\sigma\cdot X_\omega = X_{\sigma\cdot \omega}$ where $\sigma$ acts in $\mcI$ preserving each factor and mapping the $i$-coordinate in the $j$-factor to the $\min\{\sigma(i+l):0\leq l<j\}$ one. Define then
    \[ Y_\omega = \left(\Gamma\cdot X_\omega\right)\setminus\left(\bigcup_{\omega'<\omega} \Gamma\cdot X_{\omega'}\right) \]
    for each $\omega$. Hence, we get a $\sigma$-invariant decomposition $\{Y_{[\omega]}\}$ of $\GL_n(\C)/H$ indexed by $\mcI/\Gamma$.

    Now, note that each $Y_{[\omega]}$ has as much connected components as elements in the orbit of $\omega$. Denote with $\Gamma_\omega$ the stabilizer of $\omega$. We need to prove
    \[Z_\omega := X_\omega\setminus\left(\bigcup_{\omega'<\omega}  X_{\omega'}\right)\]
    has class in $A_{\Gamma_\omega}$.

    By hypothesis, $\sigma$ induces a permutation $\tau$ of $\{1,\ldots,k\}$ such that $\lambda_{\sigma(i)}=\lambda_i$ for each $i$. Assume that $\Gamma_\omega=\langle \sigma^m\rangle$. Let $\mcO_1,\ldots, \mcO_l$ the orbits of $\kappa:=\tau^m$. Let $\pi_i:\{0,1\}^{\binom{n}{\lambda_1}}\times \cdot\{0,1\}^{\binom{n}{\lambda_k}}\to \{0,1\}^{\binom{n}{\lambda_i}}$ the projection on the $i$-th factor. Then, by its very definition, we have $\pi_{\kappa(i)}(\omega)=\pi_i(\omega)$. Moreover, 
    \[Z_\omega = \prod_{i=1}^l \prod_{\lambda_j\in\mcO_i}T_{\omega,\lambda_i}\]
    where $T_{\omega,\lambda_j}$ consist of thus $n\times \lambda_j$ matrices whose $\lambda_j\times\lambda_j$ square matrices are invertible precisely when their associated coordinated of $\pi_j(\omega)$ vanish, filed by zeros to make them square matrices. The action of $\sigma^m$ in $Z_w$ is compatible with the product structure meaning that its maps the $(i,\lambda_j)$ factor to the $(i,\sigma^m\lambda_j)$ one. Hence, it suffices to prove that $T_{\omega,\mcO_i}:=\prod_{\lambda_j\in \mcO_i} T_{\omega,\lambda_j}$ has class in $A_{\Gamma_\omega}$.

    All $T_{\omega,\lambda_j}$ naturally identifies to each other. We have described a model in the previous paragraph. Call it $S_{\omega,\mcO_i}$. Then $T_{\omega,\mcO_i}\simeq S_{\omega,\mcO_i}^{|\mcO_i|}$ and the action of $\sigma^m$ in the right hand size is described by certain automorphisms $\phi_j$ of our model. Note that, if $N=|\mcO_i|$, $\phi_N\cdots\phi_2\phi_1$ is the identity. Then the automorphism 
    \[(x_1,\ldots, x_N)\mapsto (x_1,\phi_1(x_2),\phi_2\phi_1(x_3),\ldots, \phi_{N-1}\cdots\phi_1(x_N))\]
    identifies the action with the permutation of factors action.

    Up to isomorphism we may assume that $S_{\omega,\mcO_i}$ is 
    \[W_{n,a,\mcJ}:=\left\{\begin{array}{cc}
        g\in M_{n,a}\C:\text{the }a\times a\text{ invertible submatrices of }g\\
         \text{ are exactly thus whose set of rows belong to }\mcJ
    \end{array}\right\}\]
    for some family $\mcJ$ of size $j$ subsets of $\{1,\ldots,n\}$. Note that this variety is a disjoint union of $V_{n,a,r}$ for suitable $r$'s. Hence this follows by Corollary \ref{peract-on-niceclass}.
\end{proof}

\begin{lema}
    Let $\lambda_1,\ldots, \lambda_k$ be the sizes of a partition of $\{1,\ldots,n\}$ in subsets and consider the subgroup $H=\prod \GL_{\lambda_i}(\C)$ of $\GL_n(\C)$. Let $\sigma$ be a permutation of $\{1,\ldots,n\}$ which fixes the previous partition and $\Gamma=\langle \sigma\rangle$. Then the action of $\sigma$ by right translation on $\GL_n(\C)$ descent to $\GL_n(\C)/H$ and defines a class which lives in $A_\Gamma$ and satisfies
    \[[\GL_n(\C)/H]^\Gamma=\left(\prod_{i=0}^{n-1}(q^n-q^i)\otimes T_\Gamma\right)/\prod_{\mcO} [\GL_{\lambda(\mcO)}(\C)^\mcO]^{\Gamma} \]
    where $\mcO$ runs over all orbits of $\sigma$ in $\{\lambda_1,\ldots,\lambda_k\}$, $\lambda(\mcO)$ is the size of an element in $\mcO$ and $\GL_{\lambda(\mcO)}(\C)^\mcO$ is endow with the permutation of factors action induced by the action of $\Gamma$ in $\mcO$.
\end{lema}
\begin{proof}
    Let $\mcC=\GL_n(\C)/H$. Since $[\mcC]^{\Gamma}\in \Q[q]\otimes R_\Q(\Gamma)$, it suffices to compute its equivariant $E$-polynomial. For it, we can apply the arguments of \cite[proposition 2.6]{LMN} to the locally trivial fibration
    \[H\to \GL_r(\C) \to \mcC \]
    to deduce $E^\Gamma(\GL_r(\C))=E^\Gamma(H)E^\Gamma(\mcC)$, where $E^\Gamma(H)$ is computed in terms of the monodromy action. This action can be described as follows. Let us recall first how the previous fibration can be trivialized. This is a standard argument to prove that $G\to G/H$ is locally trivial. Call $\lambda$ the one-parameter subgroup generated by a diagonal $D\in\mcC$. Let $P = U(\lambda)\rtimes H$ the parabolic subgroup associated with $\lambda$ and $U(\lambda)$ its unipotent radical. Recall that $P$ and $U(\lambda)$ are defined by 
    \[P=\{g\in G:\lim_{t\to 0} \lambda(t)g\lambda(t)^{-1} \text{ exists}\}\]
    and
    \[U(\lambda)=\{g\in G:\lim_{t\to 0} \lambda(t)g\lambda(t)^{-1}=1\}.\]

    By inspecting tangent spaces, one finds that the multiplication map $U(\lambda^{-1}) \times P \rightarrow G$ is etale and injective and hence an open immersion. Therefore $U(\lambda^{-1}) \times U(\lambda)$ is a locally closed subvariety of $G$ mapping isomorphically onto a dense open subvariety $\Omega \subset G/H$. 
        
    We have a trivialization over $\Omega$ given by $s:\Omega\times H\simeq U(\lambda^{-1}) \times U(\lambda)\times H\to G$ where the second map is the multiplication map. For the point $\sigma D\sigma^{-1}$ the trivialization is given by $s(\sigma^{-1}-\sigma)\sigma:\sigma \Omega \sigma^{-1}\times H\to G$. Now choice $\omega\in \Omega\cap \sigma^{-1}\Omega\sigma$. We have the following identifications
    \[h\in H\to s(\omega)h\in \pi^{-1}(\omega)\]
    and
    \[h\in H \to s(\sigma^{-1}\omega\sigma)h\sigma \in \pi^{-1}(\omega) \]
    Hence, the monodromy action, before looking at cohomology, is given by
    \[h\in H\mapsto s(\sigma^{-1}\omega\sigma)^{-1}s(\omega)h\sigma^{-1}\in H\]

    Consider
    \[ u = \operatorname{Id} + \sum_i \lambda_{i\neq \sigma(i)}E_{i,\sigma^{-1}(i)} \in U(\lambda)\]
    for some root of unity $\lambda_i$ of high order. Note 
    \[ u\sigma = \sum \lambda_i E_{i,i}+E_{i,\sigma(i)}\in U(\lambda^{-1})\times H\]
    We have
    \[u\sigma = \tilde u h_0\]
    with $\tilde u\in U^-$ and $h_0\in H$ with finite order.
    Hence, $\omega=\pi(u)\in \Omega\cap \sigma^{-1}\Omega\sigma$ and the monodromy is induced by
    \[h\mapsto \tilde u^{-1}uh\sigma = h_0\sigma^{-1}h\sigma\]
    Now the multiplication by $h_0$ induces a linear action of a finite cyclic group in $H$ not permuting components. Hence it is trivial in cohomology and the monodromy is the permutation action induced by $\sigma$. Then
    \[ E^\Gamma(H) = \prod_{\mcO}E(\GL_{\lambda(\mcO)}(\C)^{\mcO}) \]
    where $\mcO$ runs over all orbits of $\Gamma$ in $\{\lambda_1,\ldots,\lambda_k\}$. 
\end{proof}

\begin{coro}\label{Rkappa-class-is-good}
    Let $D$ be a diagonal matrix of $\GL_n(\ov{\Q})$ and assume that its conjugacy class $\mcC$ is fixed by the left translation by a finite order $\xi\in Z(\GL_n(\ov{\Q}))$. Let $\Gamma=\langle \xi\rangle$. Then $[\mcC]\in A_\Gamma$ and
    \[[\mcC]^\Gamma=\left(\prod_{i=0}^{n-1}(q^n-q^i)\otimes T_\Gamma\right)/\prod_{\mcO} [\GL_{\lambda(\mcO)}(\C)^\mcO]^{\Gamma} \]
    where $\mcO$ runs over all orbits of $\Gamma$ in the set (without repetitions) of eigenvalues of $D$, $\lambda(\mcO)$ is the multiplicity in $D$ of any element of $\mcO$ and $\GL_{\lambda(\mcO)}(\C)^\mcO$ is endow with the permutation of factors action induced by the action of $\Gamma$ in $\mcO$.
\end{coro}

\begin{ex}\label{ex:reg-ss}
    The easiest case is when all $\lambda$ have size one ($D$ regular) and $\sigma$ acts transitively with prime order $n$. In this case,
    \begin{align*}
        [\Per_{\{e\}}^{\Gamma}(\C^\times)]^\Gamma&=(q^n-1)\otimes T_\Gamma-\sum_{1\leq j<n}\frac{1}{n}\binom{n}{j} (q-1)^j \otimes \Q^\Gamma\\
        &= (q^n-1)\otimes T_\Gamma-\frac{q^n+(-1)^n-(q-1)^n}{n} \otimes \Q^\Gamma
    \end{align*}
    by Lemma \ref{peraction}. 
    Note that
    \[\Q^\Gamma\cdot \Q^\Gamma= n\Q^\Gamma\]
    Hence,
    \begin{align*}
        [\GL_n(\C)/(\C^\times)^n]^\Gamma &= \left(\prod_{i=0}^{n-1}(q^n-q^i)\otimes T_\Gamma\right)/\left((q^n-1)\otimes T_\Gamma-\frac{q^n-1-(q-1)^n}{n} \otimes \Q^\Gamma\right)\\
        &= \prod_{i=1}^{n-1}(q^n-q^i)\otimes T_\Gamma + p\otimes \Q^\Gamma
    \end{align*}
    where
    \begin{align*}
        p((q^n-1)-q^n+1+(q-1)^n) &= \frac{q^n-1-(q-1)^n}{n}\prod_{i=1}^{n-1}(q^n-q^i)\\
        p &= \frac{q^n-1-(q-1)^n}{n(q-1)}\prod_{i=1}^{n-1}\frac{q^n-q^i}{q-1}
    \end{align*}
    That is
    \begin{align*}
        [\GL_n(\C)/(\C^\times)^n]^\Gamma
        &= \prod_{i=1}^{n-1}(q^n-q^i)\otimes T_\Gamma + \frac{q^n-1-(q-1)^n}{n(q-1)}\prod_{i=1}^{n-1}\frac{q^n-q^i}{q-1}\otimes \Q^\Gamma
    \end{align*}
    For $n=2$
    \begin{align*}
        [\GL_2(\C)/(\C^\times)^2]^\Gamma
        &= (q^2-q)\otimes T_\Gamma +q\otimes \Q^2
    \end{align*}
    and for $n=3$
    \begin{align*}
        [\GL_3(\C)/(\C^\times)^3]^\Gamma
        &= (q^3-q^2)(q^3-q)\otimes T_\Gamma + q^4(q+1)\otimes \Q^3
    \end{align*}
\end{ex}

\begin{ex}\label{regss-n=4}
    Let us analyze the same situation than before but with $n=4$. In this case,
    \begin{align*}
        [\Per_{\{e\}}^{\Gamma}(\C)]^\Gamma&= [\Per_{\{e\}}^{\Gamma}(\C^\times)]^\Gamma+1+\frac{1}{2}\Ind_{\mu_2}^\Gamma(([\Per_{\mu_2}^{\Gamma}(\C^\times)]^{\mu_2}+1)^2-([\Per_{\mu_2}^{\Gamma}(\C^\times)]^{\mu_2})^2-1)\\&+\frac{1}{4}\Ind_{\{e\}}^\Gamma(q^4-((q-1)^2+1)^2)\\
        q^4\otimes\Q^1 &= [\Per_{\{e\}}^{\Gamma}(\C^\times)]^\Gamma + 1\otimes\Q^1 +(q^2-1)\otimes\Q^2-(q-1)\otimes\Q^4+(q - 1) (q^2 - q + 1)\otimes\Q^4\\
        [\Per_{\{e\}}^{\Gamma}(\C^\times)]^\Gamma &= (q^4-1)\otimes\Q^1-(q^2-1)\otimes\Q^2 -q(q-1)^2\otimes\Q^4 
    \end{align*}
    by the binomial formula. Hence,
    \begin{align*}
        [\GL_4(\C)/(\C^\times)^4]^\Gamma
        &= (q^4-q)(q^4-q^2)(q^4-q^3)\otimes \Q^1+ (q^4-q)q^2(q^4-q^3)\otimes \Q^2\\&+ (q^2+q+1)q^7(q^2+q^1)\otimes\Q^4
    \end{align*}
\end{ex}

\begin{ex}\label{ex:2-2}
    Take the partition $\{\{1,2\},\{3,4\}\}$ and $\sigma=(13)(24)$. Then there is one orbit of size $2$ with $\lambda=2$. We have
    \begin{align*}
        [\Per_{\{e\}}^{\Gamma}(\GL_2(\C))]^\Gamma &= (q^4-q^2)(q^4-1)\otimes T_\Gamma + \frac{1}{2}((q^2-q)^2(q^2-1)^2-(q^4-q^2)(q^4-1))\otimes \Q^2\\    
        &= (q^4-q^2)(q^4-1)\otimes T_\Gamma - q^3(q-1)^2(q+1)^2\otimes \Q^2
    \end{align*}
    by Corollary \ref{pfinzq}. Hence,
    \begin{align*}
        [\GL_4(\C)/\GL_2(\C)^2]^\Gamma &= (q^4-q^3)(q^4-q) \otimes T_\Gamma + p \otimes \Q^2
    \end{align*}
    where
    \begin{align*}
        p ((q^4-q^2)(q^4-1) -2q^3(q-1)^2(q+1)^2) &=(q^4-q^3)(q^4-q)q^3(q-1)^2(q+1)^2 \\
        p(q^2+1 -2q) &=q^5(q-1)^2(q^2+q+1) \\
        p &= q^5(q^2+q+1)
    \end{align*}
    So
    \begin{align*}
        [\GL_4(\C)/\GL_2(\C)^2]^\Gamma &= (q^4-q^3)(q^4-q) \otimes T_\Gamma + q^5(q^2+q+1) \otimes \Q^2
    \end{align*}
\end{ex}

\begin{conj}
    Let $G$ be a reductive algebraic group, $H\subset G$ the centralizer of a one-parameter subgroup $\lambda$ of $G$, and $\Gamma\subset G$ a finite cyclic group such that $H$ is invariant under conjugation by elements of $\Gamma$. Consider the action of $\Gamma$ by right translation on $G/H$. Then $[G/H]\in A_\Gamma$.
\end{conj}

We belive that the proof for $\GL_n$ could be addapted as follows. Let $P$ be the parabolic subgroup of $G$ associated with $\lambda$. Then $H\subset P$ and we have a morphism $\pi:G/H\to G/P$ which is a locally trivial fibration with fiber $U_+$, the unipotent group associated with $\lambda$. Fix a Borel subgroup $B$ of $P$ which contains $U_+$ and $\lambda$. We would like that for any Bruhat cell $Bw/P$ its preimage $X_w$ by $\pi$ is $\Gamma$-invariant, but that is not the case. Therefore, one should refine the Bruhat decomposition in such a way each cell is invariant. Our propose is
\[Y_w = (\Gamma\cdot X_w )\setminus (\Gamma (\sqcup_{w'<w}X_{w'})).\]

\subsection{Localization}

\begin{Def}
    Let $\Gamma$ be a finite cyclic group and $S\subset K_0(\Var_\C)\otimes R_\Q(\Gamma)\otimes\Q$ a multiplicative set. We said that a class $\xi\in K_0^\Gamma(\Var_\C)$ is quasi-nice if there exists a numerator $\xi_n\in A_\Gamma$ and a denominator $\xi_d\in A_\Gamma$ such that $\xi\xi_d=\xi_n$ and $[\xi_d]^\Gamma\in S$. 
    Set $[-]^\Gamma_S$ to be the composition of $[-]^\Gamma$ with the localization $K_0(\Var_\C)\otimes R_\Q(\Gamma)\otimes\Q\to (K_0(\Var_\C)\otimes R_\Q(\Gamma)\otimes\Q)_{S}$
\end{Def}

\begin{lema}
    Let $\Gamma$ be a finite cyclic group and $\xi\in K_0^\Gamma(\Var_\C)$ be a quasi-nice class. Then, for any $\xi'\in K_0^\Gamma(\Var_\C)$, 
    \[[\xi\xi']^\Gamma_S = [\xi]^\Gamma_S[\xi']^\Gamma_S.\]
\end{lema}
\begin{proof}
    Indeed,
    \[[\xi_d]^\Gamma[\xi\xi']^\Gamma = [\xi_d\xi\xi']^\Gamma=[\xi_n\xi']^\Gamma=[\xi_n]^\Gamma[\xi]^\Gamma=[\xi_d]^\Gamma[\xi]^\Gamma[\xi']^\Gamma\]
    and $[\xi_d]^\Gamma\in S$.
\end{proof}

\begin{lema}
    Let $\Gamma$ be a finite cyclic group and $\xi$, $\xi'$ and $\xi''$ be classes of $K_0^\Gamma(\Var_\C)$ with $\xi\xi'=\xi''$. Then
    \begin{enumerate}
        \item If $\xi$ and $\xi'$ are quasi-nice, $\xi''$ is quasi-nice.
        \item If $\xi'$ and $\xi''$ are quasi nice and $\xi'$ admits a numerator $\xi_n'$ with $[\xi_n']^\Gamma\in S$, $\xi$ is quasi-nice.
    \end{enumerate}
\end{lema}
\begin{proof}
    It is clear that products of quasi-nice classes are quasi-nice. Assume then that $\xi'$ and $\xi''$ are quasi-nice. Take pairs $(\xi'_d,\xi'_n)$ and $(\xi''_d,\xi''_n)$ as in the definition of quasi-nice. We have
    \[\xi \xi'_n\xi''_d = \xi\xi'\xi'_d\xi''_d=\xi''\xi'_d\xi''_d=\xi'_d\xi''_n\]
    and the result follows.
\end{proof}

\begin{lema}
    Let $\Gamma\subset \Gamma'$ be finite cyclic groups and $S\subset K_0(\Var_\C)\otimes R_\Q(\Gamma)\otimes \Q$ and $S'\subset K_0(\Var_\C)\otimes R_\Q(\Gamma')\otimes \Q$ be multiplicative sets. If $\Ind_\Gamma^{\Gamma'}(S)\subset S'$, induction of quasi-nice classes are quasi-nice.
\end{lema}
\begin{proof}
    Let $\xi$ be quasi-nice with associated pair $(\xi_d,\xi_n)$ for $\Gamma$. We have
    \[\frac{|\Gamma'|}{|\Gamma|}\Ind_\Gamma^{\Gamma'}(\xi_n)=\frac{|\Gamma'|}{|\Gamma|}\Ind_\Gamma^{\Gamma'}(\xi\xi_d)=\Ind_\Gamma^{\Gamma'}(\xi)\cdot\Ind_\Gamma^{\Gamma'}(\xi_d)\]
    and the result follows.
\end{proof}

\begin{lema}
    Let $\Gamma$ be a finite cyclic group and $S\subset K_0(\Var_\C)\otimes R_\Q(\Gamma)\otimes\Q$ be a multiplicative set. Let $\xi,\xi'\in K_0^\Gamma(\Var_\C)$. If $[\xi]_S^\Gamma=[\xi']_S^\Gamma$, for any subgroup $H\subset \Gamma$, 
    \[[\xi/H] = [\xi'/H]\]
    in $K_0(\Var_\C)_{\ov{S}}$ where $\ov{S}$ is generated by elements of the form $\sum_{d_1|d}d_1s_{d_1}$ where $d$ is a divisor of $|\Gamma|$ and $s = \sum_d s_d\otimes \Q^d\in S$.
\end{lema}
\begin{proof}
    It is enough to prove that if $s\xi=0$ for some $\xi\in K_0(\Var_\C)\otimes R_\Q(\Gamma)\otimes\Q$ and $s\in S$, 
    \[\ov{s}\langle T_H,\Res_H^\Gamma(\xi))=0\]
    for some $\ov{s}\in \ov{S}$.
    Now if $\xi = \sum \xi_{d} \otimes \Q_d$, $s=\sum s_{d}\otimes \Q_d$,
    \[s\xi = \sum_d \left(\sum_{\operatorname{lcm}(d_1,d_2)=d} \gcd(d_1,d_2)s_{d_1}\xi_{d_2}\right)\otimes \Q^d \]
    Hence, our hypothesis is that
    \[\sum_{\operatorname{lcm}(d_1,d_2)=d} \gcd(d_1,d_2)s_{d_1}\xi_{d_2}=0\]
    for any $d$. We claim that
    \[\left(\prod_{d_2|d}\left(\sum_{d_1|d_2}d_1s_{d_1}\right)\right)\xi_d=0\]
    for any $d$. Indeed, for $d=1$, it holds; $s_{1}\xi_{1}=0$. And, by the hypothesis, 
    \[\left(\sum_{d_1|d}d_1s_{d_1}\right)\xi_d\]
    can be written in terms of $\xi_{d'}$ for proper divisors $d'$ of $d$.

    Now, 
    \[\langle T_H, \Res_H^\Gamma(\xi)\rangle=\sum \xi_d\langle T_H,\Res_H^\Gamma(\Q^d)\rangle\]
    and 
    \[\ov{s} = \prod_{d}\left(\sum_{d_1|d}d_1s_{d_1}\right)\]
    makes the job.
\end{proof}

\begin{lema}
    Let $\Gamma\subset\Gamma'$ be finite cyclic groups and $S\subset \Q[q]\otimes R_\Q(\Gamma)$ be a multiplicative set. Let $N=|\Gamma'/\Gamma|$. Then if $\xi\in K_0^\Gamma(\Var_\C)$ is quasi-nice, $\Per_{\Gamma}^{\Gamma'}(\xi)$ is quasi-nice with respect to $\Per_{\Gamma}^{\Gamma'}(S):=\{\Omega_N(s):s\in S\}$.
\end{lema}
\begin{proof}
    Indeed, write $\xi\xi_d=\xi_n$. Then $\Per_{\Gamma}^{\Gamma'}(\xi)\cdot\Per_{\Gamma}^{\Gamma'}(\xi_d)=\Per_{\Gamma}^{\Gamma'}(\xi_n)$ and $\Per_{\Gamma}^{\Gamma'}(\xi_d),\Per_{\Gamma}^{\Gamma'}(\xi_n)\in A_{\Gamma'}$, by Theorem \ref{thm:per-act} and Corollary \ref{peract-on-niceclass}, and $[\Per_{\Gamma}^{\Gamma'}(\xi_d)]^\Gamma\in \Per_{\Gamma}^{\Gamma'}(S)$ by Lemma \ref{classperact}.
\end{proof}

\begin{lema}
    Let $\Gamma\subset \Gamma'$ be finite cyclic groups and $S\subset \Q[q]\otimes R_\Q(\Gamma)$ be a multiplicative set. Then if $[\xi]_S^\Gamma=[\xi']_S^\Gamma$, for some quasi-nice classes $\xi,\xi'$, then $[\Per_{\Gamma}^{\Gamma'}(\xi)]^{\Gamma'}_{S'}=[\Per_{\Gamma}^{\Gamma'}(\xi')]^{\Gamma'}_{S'}$ where $S'\subset K_0(\Var_\C)\otimes R_\Q(\Gamma')$ is the multiplicative set generated by $\Omega_N(s)$, $s\in S$.
\end{lema}
\begin{proof}
    Assume that $\xi\xi_d=\xi_n$ and $\xi'\xi'_d=\xi'_n$ with $\xi_d,\xi_n,\xi'_d,\xi'_n$ as in the definition of quasi-nice. Then $\Per_{\Gamma}^{\Gamma'}(\xi)\cdot\Per_{\Gamma}^{\Gamma'}(\xi_d)=\Per_{\Gamma}^{\Gamma'}(\xi_n)$. We know that $[\Per_{\Gamma}^{\Gamma'}(\xi'')]^{\Gamma'}=\Omega_N([\Per_{\Gamma}^{\Gamma'}(\xi'')]^\Gamma)$ for $\xi''\in A_\Gamma$. Hence $[\Per_{\Gamma}^{\Gamma'}(\xi)]^{\Gamma'}\Omega_N([\Per_{\Gamma}^{\Gamma'}(\xi_d)]^\Gamma)=\Omega_N([\Per_{\Gamma}^{\Gamma'}(\xi_n)]^\Gamma)$ and similar with $\xi'$.  
    
    Now recall that $\Omega_N$ is built as a composition of a ring homomorphism $f:\Q[q]\otimes R_\Q(\Gamma)\to R$ with a linear projection $\pi:R\to \Q[q]\otimes R_\Q(\Gamma')$. If $[\xi]_S^\Gamma=[\xi']_S^\Gamma$, $f([\xi_n]^\Gamma/[\xi_d]^\Gamma)=f([\xi'_n]^{\Gamma}/[\xi'_d]^{\Gamma})$ in $R_{S'}$ for $S'=\{f(s):s\in S\}$. Now, $\pi$ is multiplicative and therefore the result follows. 
\end{proof}

\section{Character varieties of torus knots}

We follow a similar notation to \cite{PM} for readers convenience. Let $n,m$ be coprime positive integers and 
\[\Gamma_{n,m}=\langle x,y | x^n=y^m\rangle\]
be the fundamental group of the torus knot $(n,m)$. For an integer $r$, denote with 
\[\mcR_r := \Hom{\Gamma_{n,m}}{\GL_r(\C)}=\{(A,B)\in \GL_r(\C)^2:A^n=B^m\}\]
the $\GL_r$-representation variety associated with $\Gamma_{n,m}$. Let $\mcR^{irr}_r$ and $\mcR^{red}_r$ be the subvarieties of those representations that are irreducible and reducible respectively. In these spaces acts $\GL_r(\C)$ by conjugation. Our objective is to compute the motive of 
\[\mcM_r^{irr} := \mcR^{irr}_r/\PGL_r(\C)\]
the moduli space of irreducible $r$-dimensional representations. Changing $\GL_r$ by $\SL_r$ in the definitions one gets similar spaces which we will denote $\mcR_r(\SL)$, $\mcR_r^{irr}(\SL)$, $\mcR_r^{red}(\SL)$ and $\mcM^{irr}_r(\SL)$. These are the ones studied on \cite{PM}. To avoid confusions, whenever both $\SL$ and $\GL$ versions appear, we will add $``(\GL)"$ to the notation of the $\GL$-versions. 

We will mainly work with $\mcR^{irr}_r$ rather than with $\mcM^{irr}_r$. The reason is simply the following conjecture that we will prove for $r\leq 3$ in section \ref{sec:pglquotient}, as it is useful to introduce a few notations first.

\begin{conj}\label{lem:mcM-vs-mcR}
    Let $n,m$ be coprime positive integers. For any positive integer $r$,
    \[[\mcM^{irr}_r]=\frac{[\mcR^{irr}_r]}{[\PGL_r(\C)]}=\frac{[\mcR^{irr}_r]}{q^{r-1}(q^r-1)(q^r-q)\cdots (q^r-q^{r-2})}\]
    in the localization of $K_0(\Var_\C)$ by $q$ and $q^i-1$ for $i=0,\ldots,r$.
\end{conj}

\begin{lema}
    For $(A,B)\in \mcR^{irr}_r$, $A^n=B^m$ is a central matrix.
\end{lema}
\begin{proof}
    This follows by Schur's lemma as $P=A^n=B^m$ commutes with both $A$ and $B$.
\end{proof}

Let $\pi:\mcR^{irr}_r\to \C^\times$ be the composite of $(A,B)\mapsto A^n$ with the inverse of $\C^\times \to Z(\GL_r(\C))$, $\omega\mapsto \omega\operatorname{Id}_r$. In addition, consider the pullback diagram
\[\xymatrix{
\hat\mcR^{irr}_r\ar[r]\ar[d] & \mcR^{irr}_r\ar[d]^{\pi} \\
\C^\times\ar[r]^{(-)^{nm}} & \C^\times 
}\]
where $(-)^{nm}:\C^\times \to \C^\times$ is the exponentiation by $nm$ map. Let $\mu_{nm}$ be the group of $nm$-roots of unity. There is an induced action on $\hat\mcR^{irr}_r$ such that $\mcR^{irr}_r=\hat\mcR^{irr}_r/\mu_{nm}$. The map $\pi$ is $\PGL_r$-invariant and, therefore, defines $\ov{\pi}:\mcM_r^{irr}\to \C^\times$, and there is an induced $\PGL_r$-action on $\hat\mcR_r^{irr}$. Set $\hat\mcM_r^{irr}$ as the pullback by $(-)^{nm}$ of $\mcM_r^{irr}$. Alternatively, $\hat\mcM_r^{irr}$ is the $\PGL_r$-quotient of $\hat\mcR_r^{irr}$.

On the other hand, let 
\[R_r:=\Hom{\Z_n\star \Z_m}{\GL_r}=\{(A,B)\in\GL_r(\C)^2:A^n=B^m=\operatorname{Id}_r\}\]
with the action of $\mu_{nm}$ given by $z\cdot(A,B)=(z^mA,z^nB)$. As before denote with $R^{irr}_r$ and $R^{red}_r$ the subvariaties of irreducible and reducible representations respectively. Let $M_r^{irr}$ the quotient of $R_r^{irr}$ by the conjugation action of $\PGL_r$. Note the actions of $\PGL_r$ and $\mu_{nm}$ commute. In particular, there is an induced action of $\mu_{nm}$ on $M_r^{irr}$.

\begin{lema}\label{lem:key-lemma}
    The pullback $\hat \mcR_r^{irr}$ is $\mu_{nm}\times \PGL_r$-equivariantly isomorphic to the product $\C^\times \times R^{irr}_r$ where the action is given by $(z,g)\cdot (\omega,A,B)=(z\omega,z^{m}gAg^{-1},z^{n}gBg^{-1})$. In particular, $\hat\mcM_r^{irr}$ is $\mu_{nm}$-equivariantly isomorphic to $\C^\times\times M_r^{irr}$.
\end{lema}
\begin{proof}
    By its very definition,
    \[\hat\mcR_r^{irr} = \{(\omega,A,B)\in \C^\times\times \mcR_r^{irr}: A^n=B^m=\omega^{nm}\operatorname{Id}_{r}\}\]
    where the action is given by $(z,g)\cdot(\omega,A,B)=(z\omega,gAg^{-1},gBg^{-1})$.
    
    Consider the involution $f:\C^\times\times\GL_2(\C)^2\to \C^\times\times\GL_2(\C)^2$ given by
    \[f(\omega,A,B)=(\omega,\omega^{m}A^{-1},\omega^{n}B^{-1}).\]
    Let $\mu_{nm}\times\PGL_r$ act in the domain by $(z,g)\cdot(\omega,A,B)=(z\omega,gAg^{-1},gBg^{-1})$ and, on the codomain, by $(z,g)\cdot(\omega,A,B)=(z\omega,z^mgAg^{-1},z^ngBg^{-1})$. Then $f$ is an equivariant isomorphism. Therefore, the lemma is equivalent to $f(\mcR^{irr}_r)=\C^\times\times R_r^{irr}$.

    First, let $(\omega, A, B)\in \mcR_r^{irr} $. Then
    \[(\omega^mA^{-1})^n=\omega^{nm}\operatorname{Id}_r A^{-n}=\operatorname{Id_r}\]
    and, similarly, $(\omega^nB^{-1})^m$. Therefore, $f(\omega,A,B)\in\C^\times\times R_r$. Moreover, the representation associated with $(A,B)$ is irreducible if and only if the one associated with $(\omega^mA^{-1},\omega^nB^{-1})$ is. This shows that $f(\mcR^{irr}_r)\subset \C^\times\times R_r^{irr}$.
    
    Conversely, if $(\omega,A,B)\in \C^\times\times R_r^{irr}$, $f(\omega,A,B)=(\omega, \omega^mA^{-1},\omega^nB^{-1})$ satisfies
    \[ (\omega^mA^{-1})^n=\omega^{nm}A^{-n}=\omega^{nm}\operatorname{Id}_r\]
    and, similarly, $(\omega^nB^{-1})^m=\omega^{nm}\operatorname{Id}_r$. It follows, $\C^\times\times R_r^{irr}\subset f(\mcR^{irr}_r)$.
\end{proof}

There is an stratification of $R_r$ we describe now. The eigenvalues of $A$ and $B$ are roots of unity of order dividing $n$ and $m$ respectively. Let $\boldsymbol{\epsilon}=\{\epsilon_1^{a_1},\ldots,\epsilon_p^{a_p}\}$ and $\boldsymbol{\varepsilon}=\{\varepsilon_1^{b_1},\ldots,\varepsilon_q^{b_q}\}$ be the multi-sets of eigenvalues of $A$ and $B$ respectively. The pair
\[\kappa = (\boldsymbol{\epsilon},\boldsymbol{\varepsilon})\]
is called the configuration of eigenvalues. There are finitely many possibilities for $\kappa$ and there is a decomposition
\[R = \bigsqcup_\kappa R_\kappa\]
given by the possibles eigenvalues. It is equivariant with respect to the conjugation action of $\GL_r(\C)$. The action of $\mu_{nm}$ is compatible with this stratification in the sense that $\omega\cdot R_\kappa = R_{\omega\cdot \kappa}$ where $\omega\cdot (\boldsymbol{\epsilon},\boldsymbol{\varepsilon})=(\omega^m\boldsymbol{\epsilon},\omega^n\boldsymbol{\varepsilon})$.

This stratification can be further refined. For a finite dimensional representation $V$ of $\Z_n\star\Z_m$, there is filtration, the semi-simple filtration, 
\[0=V_0\subset V_1\subset \cdots\subset V_s=V\]
such that $V_{i}/V_{i-1}$ is the maximal semisimple subrepresentation of $V/V_{i-1}$ for each $i$. Each graded piece decomposes as
\[ \operatorname{Gr}_i(V_{\bullet}) :=V_{i}/V_{i-1} = \bigoplus_{j=1}^{r_i} W_{i,j}^{m_{i,j}} \]
where $W_{i,j}$ are irreducible representations of $\Gamma$, non-isomorphic for fixed $i$, and $m_{i,j}$ are positive integers. The representations $W_{i,j}$ are called the isotypic pieces of $V$. The collection
\[\xi = \left\{\begin{array}{c}
    \{(\dim W_{1,1},m_{1,1}),(\dim W_{1,2},m_{1,2}),\ldots,(\dim W_{1,r_1},m_{1,r_1})\},\\
    \vdots\\
    \{(\dim W_{s,1},m_{s,1}),(\dim W_{s,2},m_{s,2}),\ldots,(\dim W_{s,r_s},m_{s,r_s})\}
\end{array}\right\}\]
is called the shape of $V$. For given $i,j$, we have the multi-sets $\sigma_{i,j}(\gamma)=\{\epsilon_{i,j,1}^{n_{i,j,1}},\ldots,\epsilon_{i,j,l}^{n_{i,j,l}}\}$ of eigenvalues of $\gamma\in \Gamma$ in $W_{i,j}$. The collection of eigenvalues of $V$ is
\[\sigma := \{\sigma_{i,j}(\gamma)\}\]
and
\[\tau := (\xi,\sigma)\]
is called the type of $V$. There is an stratification by types
\[R_r = \bigsqcup_\tau R(\tau)\]
which is $\GL_r$-equivariant and compatible with the $\mu_{nm}$-action as before. Note that irreducible representations are those of shape $\{\{(r,1)\}\}$.

For a given $\kappa$, let $\mcT_\kappa$ be the collection of types whose eigenvalues are taken from $\kappa$. Let $\mcT_\kappa^*\subset \mcT_\kappa$ be the subset of those types which do not correspond to irreducible representations. Then
\[R_\kappa = \bigsqcup_{\tau \in \mcT_\kappa} R(\tau),\]
\[R^{red}_\kappa = \bigsqcup_{\tau \in \mcT_\kappa^*}R(\tau),\]
and
\[R_r^{irr} = \bigsqcup_\kappa R^{irr}_\kappa =\bigsqcup_\kappa\left(R_\kappa-\bigsqcup_{\tau\in\mcT_\kappa^*}R(\tau)\right).\]
This decomposition is compatible with the $\mu_{nm}$-action if we endow the types with the action given by multiplying the eigenvalues as before.

\begin{lema}
    Let $\kappa$ be a configuration of eigenvalues, $[\kappa]$ its $\mu_{nm}$-orbit and $\Gamma_\kappa$ its stabilizer. Then
    \begin{enumerate}
        \item There is a $\GL_r(\C)$ and $\mu_{nm}$ equivariant isomorphism \[\bigsqcup_{\kappa'\in[\kappa]}R_\kappa\simeq \Ind_{\Gamma_\kappa}^{\mu_{nm}}(R_\kappa).\] 
        \item $\bigsqcup\limits_{\kappa'\in[\kappa]}\mcT_{\kappa'}^*$ is $\mu_{nm}$-invariant.
    \end{enumerate}
\end{lema}
\begin{proof}\noindent
    \begin{enumerate}
        \item Take $((A,B),z)\mapsto (z^mA,z^nB)$. It is $\GL_r(\C)$-equivariant because $z$ is central.
        \item This amounts to say that the action of $\mu_{nm}$ preserves irreducibility.
    \end{enumerate}
\end{proof}

\begin{lema}
    Let $\kappa$ be a configuration of eigenvalues, $\tau\in \mcT_\kappa^*$ be a type, and $\Gamma_\kappa$ and $\Gamma_{\tau}$ be its stabilizers. Then there is a $\Gamma_\kappa$-equivariant isomorphism \[\bigsqcup_{\tau'\in[\tau]}R(\tau')\simeq \Ind_{\Gamma_{\tau}}^{\Gamma_\kappa}(R(\tau))\]
    where $[\tau]$ is the $\Gamma_\kappa$-orbit of $\tau$.
\end{lema}
\begin{proof}
    Same proof as before.
\end{proof}

In conclusion, the equivariant class of $R_r^{irr}$ is determine by those of $R_\kappa$ and $R(\tau)$, $\tau\in \mcT_\kappa^*$. The former has been proven to belong to $A_{\Gamma_\kappa}$ and computed in Corollary \ref{Rkappa-class-is-good}. Indeed, 
\[R_\kappa = \GL_r(\C) \cdot \Sigma_{\boldsymbol{\epsilon}} \times \GL_r(\C)\cdot \Sigma_{\boldsymbol{\varepsilon}}\]
where $\Sigma_{\boldsymbol{\epsilon}}$ and $\Sigma_{\boldsymbol{\varepsilon}}$ are diagonal matrices with eigenvalues $\boldsymbol{\epsilon}$ and $\boldsymbol{\varepsilon}$ respectively.

The following lemmas will be very useful in future computations. For $\xi\in K_0(\Var_\C)\otimes R_\Q(\Gamma)\otimes\Q$, write $c_d(\xi)\in K_0(\Var_\C)$ for the $d$-th coefficient of $\xi$ in the basis $\{\Q^d:d||\Gamma|\}$.

\begin{lema}\label{correction-lemma}
    For any configuration of eigenvalues $\kappa$ with stabilizer $\Gamma_\kappa$,
    \[[R_\kappa^{red}]^{\Gamma_\kappa}= \frac{1}{|\Gamma_\kappa|}[R_\kappa^{red}]\otimes\Q^{\Gamma_\kappa} +\sum_{\tau} \sum_{d}c_d(R(\tau)) \otimes \left(\frac{|\Gamma_\tau|}{|\Gamma_\kappa|}\Ind_{\Gamma_\tau}^{\Gamma_\kappa}(\Q^d)-\frac{d}{|\Gamma_\kappa|}\Q^{\Gamma_\kappa}\right)\]
    where $\tau$ runs over all elements of $\mcT_\tau^*$ with non-trivial stabilizer $\Gamma_\tau$ and $d$ over all positive divisors of $|\Gamma_\tau|$ strictly smaller than $|\Gamma_\tau|$.
\end{lema}
\begin{proof}
    This follows by the previous lemma and that
    \[[R(\tau)]=\sum_{d||\Gamma_\tau|} dc_d(R(\tau))\]
    in $K_0(\Var_\C)$.
\end{proof}

\begin{lema}
    For any $r$,
    \[[R_r^{irr}]^{\mu_{nm}}=\frac{1}{nm}[R_r^{irr}]\otimes\Q^{nm}+\sum_{\kappa}\sum_{d}c_d(R_\kappa^{irr}) \otimes \left(\frac{|\Gamma_\kappa|}{nm}\Ind_{\Gamma_\kappa}^{\mu_{nm}}(\Q^d)-\frac{d}{nm}\Q^{nm}\right)\]
    where $\kappa$ runs over all configurations of eigenvalues with non-trivial stabilizer $\Gamma_\kappa$ and $d$ over all positive divisors of $|\Gamma_\kappa|$ strictly smaller than $|\Gamma_\kappa|$.
\end{lema}
\begin{proof}
    Same proof as before.
\end{proof}

\begin{lema}
    For any configuration of eigenvalues $\kappa$, the size of its stabilizer $\Gamma_\kappa$ divides $\gcd(r,nm)$.
\end{lema}
\begin{proof}
    Indeed, the action of $\Gamma_\kappa$ on $\C^\times$ is free and orbits in $\kappa$ have size summing up $r$. Hence each factor of $\Gamma_\kappa$ in $\mu_{n}\times \mu_{m}\simeq \mu_{nm}$ (recall $\gcd(n,m)=1$) divides $\gcd(r,n)$ and $\gcd(r,m)$ respectively. Therefore, $|\Gamma_\kappa| | \gcd(r,n)\gcd(r,m)=\gcd(r,nm)$.
\end{proof}

\begin{lema}\label{stab-tau-divides-ri}
    For any type $\tau$, the size of its stabilizer $\Gamma_\tau$ divides $r_i$ for $i=1,\ldots, s$.
\end{lema}
\begin{proof}
    Same proof as before.
\end{proof}

We are ready to prove Theorem \ref{coprime-case} which is an immediate consequence of the following.

\begin{thm}
    Let $n,m,r$ be positive integers with $n$ and $m$ coprime. If $r$ is coprime with $n$ and $m$, 
    \[[\mcR_r^{irr}(\GL)]=\frac{q-1}{r}[\mcR_r^{irr}(\SL)]\]
    in $K_0(\Var_\C)$ where $q$ is the class of the affine line.
\end{thm}
\begin{proof}
    Note that the action of $\mu_{nm}$ on $\C^\times$ defines a class in $A_{\mu_{nm}}$. Hence,
    \begin{align*}
        [\mcR_r^{irr}(\GL)] &= (q-1)\langle T_{\mu_{nm}},  [R_r^{irr}]^{\mu_{nm}}\rangle\\
        &= \frac{q-1}{nm}\sum_\kappa[R_\kappa^{irr}]+\sum_{\kappa}\sum_{d}(q-1)c_d(R_\kappa^{irr})\frac{|\Gamma_\kappa|-d}{nm}\\
        &= \frac{q-1}{nm}\sum_\kappa[R_\kappa^{irr}]
    \end{align*}
    by the previous lemmas. 
    
    On the other hand, by \cite[Lemma 3.1 and Proposition 3.2]{PM},
    \[[\mcR_r^{irr}(\SL)]=\sum_{\omega\in \mu_r}\frac{1}{nm}\sum_{\eta_1}\sum_{\eta_2}\sum_{\kappa}[R_\kappa^{irr}]\]
    where $\mu_r$ is the group of $r$-roots of unity, $\eta_1$ and $\eta_2$ are $n$ respectively $m$ roots of $\omega^{-1}$, and $\kappa$ is of the form $((\eta_1\epsilon_i),(\eta_2\varepsilon_j))$ where  $\epsilon_i^n=\omega=\varepsilon_j^m$ and $\prod \epsilon_i=1=\prod \varepsilon_j$. 

    For an arbitrary $\kappa$, $\prod \epsilon_i\in \mu_{n}$ and $\prod \varepsilon_j\in \mu_{m}$. Hence, because $\gcd(r,nm)=1$, there is only one solution to $\eta_1^r=\prod \epsilon_i^{-1}$ and $\eta_2^r=\prod\varepsilon_j^{-1}$ in $\mu_{n}$ and $\mu_{m}$ respectively. It follows that for $\omega=1$ every $\kappa$ can be realized exactly once. For arbitrary $\omega$, fix $n$ and $m$ roots $\nu_1$ and $\nu_2$ in $\mu_r$. Then $\epsilon_i'=\nu_1\epsilon_i$, $\varepsilon_j'=\nu_2\varepsilon_j$, $\eta_1'=\nu_1^{-1}\eta_1$ and $\eta_2'=\nu_2^{-1}\eta_2$ shows that $\kappa$ can also be realized with $\omega$.
    
    In conclusion,
    \[[\mcR_r^{irr}(\GL)]=\frac{q-1}{nm}\sum_\kappa [R_\kappa^{irr}]=\frac{q-1}{nm}\frac{nm}{r}[\mcR_r^{rr}(\SL)]=\frac{q-1}{r}[\mcR_r^{irr}(\SL)].\]
\end{proof}

A slightly modification of the previous proof holds a version for $\gcd(r,nm)\neq1$. The computations of \cite{PM} show that, for $r\leq 4$,
\[[\mcR_r^{irr}(\SL)]^*:=\frac{r}{nm}\sum_\kappa [R_\kappa^{irr}]\]
depends polynomialy on $n$ and $m$ and agrees with $[\mcR_r^{irr}(\SL)]$ for $\gcd(r,nm)=1$. Moreover, for $r=3$ or $r=4$ it actually holds $[\mcR_r^{irr}(\SL)]^*=[\mcR_r^{irr}(\SL)]$ (\cite{sl3}, \cite{PM}). But for $r=n=2$ one gets $[\mcR_r^{irr}(\SL)]=0$ and $[\mcR_r^{irr}(\SL)]^*=\frac{1}{2}(m-1) (q^2-q)(q^2-1)$. Looking at the first paragraph and the last equation of the previous proof one gets:

\begin{lema}\label{big-correction-lemma}
    For any $r\leq4$,
    \[[\mcR_r^{irr}(\GL)]=\frac{q-1}{r}[\mcR_r^{irr}(\SL)]^*+\sum_{\kappa}\sum_{d}(q-1)c_d(R_\kappa^{irr}) \frac{|\Gamma_\kappa|-d}{nm}\]
    where $\kappa$ runs over all configurations of eigenvalues with non-trivial stabilizer $\Gamma_\kappa$ and $d$ over all positive divisors of $|\Gamma_\kappa|$ strictly smaller than $|\Gamma_\kappa|$.
\end{lema}

\section{The core fibration}\label{sec:corefib}

Here we deal with $R(\tau)$ for a configuration of eigenvalues $\kappa$ and a type $\tau\in\mcT_\kappa^*$. The main tool in \cite{PM} is a locally trivial fibration, see section 4 of loc.cit.. We are going to check that it is equivariantly locally trivial. But with an slightly different definition. Fix such a $\kappa,\tau$. Denote with $x$ and $y$ generators of $\Z_n$ and $\Z_m$ respectively. Following the notation of the previous section, let $\kappa_{i,j}$ be $(\sigma_{i,j}(x),\sigma_{i,j}(y))$ the eigenvalues of $x$ and $y$ on the isotypic piece $W_{i,j}$. 

The fact that $\Gamma_\tau$ fixes $R(\tau)$ implies that each $\gamma\in\Gamma$ maps $\kappa_{i,j}$ to $\kappa_{i,j'}$ for some $j'$ with $m_{i,j}=m_{i,j'}$. There may be more than one possible $j'$. Hence, we will change our definition of type to group together all equal tuples $(i,\dim W_{i,j},m_{i,j},\kappa_{i,j})$ into $(i,d_{i,j},m_{i,j},\kappa_{i,j},M_{i,j})$ where $M_{i,j}$ is the number of occurrences of the tuple $(i,d_{i,j},m_{i,j},\kappa_{i,j})$. So that we have a induced action of $\Gamma_\tau$ on the indexes $(i,j)$. All the results of the previous section still work with this definition.

There is an algebraic map
\[\operatorname{Gr}_\bullet:R(\tau)\to \mcM(\tau)\subset \prod_{i=1}^s\prod_{j=1}^{r_i}\mathrm{Sym}^{M_{i,j}}((\mcM^{irr}_{\kappa_{i,j}})^{m_{i,j}})\]
which maps a pair of matrices $(A,B)\in R(\tau)$ to the isotypic components of its associated representation. Here $\mcM^{irr}_{\kappa_{i,j}}$ is the (coarse) moduli of irreducible representations of $\Z_n*\Z_m$ with eigenvalues $\kappa_{i,j}$. Explicitly, it is $R^{irr}_{\kappa_{i,j}}$ quotient by $\PGL_{\dim W_{i,j}}(\C)$. The variety $M(\tau)$ is given by those tuples which respect the conditions of being isomorphic or non-isomorphic imposed by $\tau$. Note $\operatorname{Gr}_\bullet$ is equivariant if $\Gamma_\tau$ acts on the indexes as before.
We would like to upgrade it to a map to
\[\mcI(\tau) \subset \prod_{i=1}^s\prod_{j=1}^{r_i}\mathrm{Sym}^{M_{i,j}}((R^{irr}_{\kappa_{i,j}})^{m_{i,j}}).\]
We will do it up to fiber bundles for the unsymmetrized versions:
\[\hat\mcM(\tau)\subset\prod_{i=1}^s\prod_{j=1}^{r_i}(\mcM^{irr}_{\kappa_{i,j}})^{M_{i,j}m_{i,j}},\]
\[\hat\mcI(\tau)\subset\prod_{i=1}^s\prod_{j=1}^{r_i}(R^{irr}_{\kappa_{i,j}})^{M_{i,j}m_{i,j}}\]
and 
\[\hat R(\tau)=R(\tau)\times_{\mcM(\tau)}\hat\mcM(\tau).\]
Remark that under the following hypothesis
\begin{center}
    (H) \it For each $M_{i,j}>1$, $\dim W_{i,j}=1$. \rm
\end{center}
the unsymmetrized version agrees with the previous ones because $\hat\mcI(\tau)=\mcI(\tau)$, see Corollary 4.7 of \cite{PM}. This assumption hold for $r\leq 4$ unless $r=4$, $s=1$, $m_{1,1}=1$, $M_{1,1}=2$ and $d_{1,1}=2$. However, in the previous case there are no irreducible representations, see \cite[Prop 8.1]{PM}. On the other hand, the assumption also rules out the non-isomorphic conditions imposed by $\tau$. Indeed, for any $d_{i,j}=1$ there is only one point in $R_{\kappa_{i,j}}$ and so $I(\tau)=\emptyset$ if some $M_{i,j}>1$. Now, the being isomorphic condition is immediate for $d_{i,j}=1$. Hence, the unique case $r\leq 4$ where it is relevant is $r=4$, $s=1$, $m_{1,1}=2$, $d_{1,1}=2$. But as before, there are no irreducible representations in this case. In conclusion, for $r\leq 4$, $M(\tau)$ and $I(\tau)$ are either empty or agree with its unsymmetrized versions.

We encourage the reader to review the notation that appeared previously. In particular, we recall that the $i$-th piece of the semisimple filtration is denoted by $V_i$. Its dimension will be called $d_i$.

\begin{lema}
    For any configuration of eigenvalues $\kappa$ and type $\tau\in \mcT_\kappa^*$, there is a equivariant fiber bundle $\tilde R(\tau)\to \hat R(\tau)$ with fiber
    \[F_0(\tau):=\prod_{i=1}^s\left(\C^{(d_i-d_{i-1})(r-d_i)}\times\GL_{r-d_i}(\C)\times\prod_{j=1}^{r_i} \GL_{d_{i,j}}(\C)^{m_{i,j}M_{i,j}}\times(\GL_{m_{i,j}}(\C)/(\C^\times)^{m_{i,j}})^{M_{i,j}}\right)\]
    and a equivariant lifting $\tilde{\operatorname{Gr}}_\bullet:\tilde R(\tau)\to \hat\mcI(\tau)$ of $\hat{\operatorname{Gr}}_\bullet:\hat R(\tau)\to \hat\mcM(\tau)$. In addition,
    \[ [\tilde R(\tau)] = [\hat R(\tau)][F_0(\tau)]\]
    in $K_0^{\Gamma_\tau}(\Var_\C)$ where the action of $\Gamma_\tau$ in the fiber is given by permutation of factors; $\gamma$ maps the $(i,j)$-factor to the $\gamma\cdot (i,j)$-one.
\end{lema}
\begin{proof}
    For any integers $r'<r$ there is a variety $\mcN_{r',r}$ whose points parametrize pairs $(\rho,W)$ where $\rho\in R_r$ and $W\subset\C^r$ is an invariant subspace of dimension $r'$. For instance, take the quiver $Q$ with one point and two loops. Its representation spaces paramatrize representations of $\Z*\Z$. Inside they lie as a closed subsets all $R_r$. Hence, one  can build $\mcN_{r',r}$ as a closed subvariety of the universal family of the quiver grassmaniann associated to $Q$ (see \cite{QuiverGrasmaniann}). Note there is a map $\mcN_{r',r}\to \operatorname{Gr}(r',r)$ as all elements of $R_r$ are representations on $\C^r$. Similarly, for two types $\tau_1$ and $\tau_2$ there is a parameter space $\mcN_{\tau_1,\tau_2}$ of tuples $(W,V)$ where $V\in R_{\tau_2}$ and $W\subset V$ is an invariant subspace of type $\tau_1$. We have a map from $\mcN_{\tau_1,\tau_2}$ to a grasmaniann.

    Let $\tau_1$ be a irreducible type and $\tau_2=m\tau_1$. In this case, the fiber of $\mcN_{\tau_1,\tau_2}\to R_{\tau_2}$ over $V\simeq W^m\in R_{\tau_2}$ is nothing but $\{f:W\to V\text{ non zero}\}/\C^\times=\P^{m-1}$. Passing to each affine cone, we get $\ov{\mcN}_{\tau_1,\tau_2}\to R_{\tau_2}$ whose fiber is $\C^m$. It has a associated $\GL_m$-principal bundle over $R_{\tau_2}$ which parametrizes basis of $\Hom{W}{V}$. Quotient by $(\C^\times)^m$ we get a locally trivial fiber bundle which parametrizes ways to decompose $V$ as direct sums of simple surepresentations. It is endow with a map to $\mcN_{\tau_1,\tau_2}^m$.
    
    We prove the lemma by induction on $s$, the length of the semisimple filtration. We show the base case and the inductive step at the same time. Let $\tau'$ be the type of $V/V_1$.

    There is a equivariant map $R(\tau)\to \prod_{j=1}^{r_1}\mathrm{Sym}^{M_{1,j}}(\mcN_{(m_{1,j}\kappa_{1,j}),\tau}) \times \mcN_{d,\tau}$ given by looking at $V_1$ and all its isotypic subrepresentations $(W_{1,j}^{m_{1,j}})$. By means of the previously explained $\GL_{m}/(\C^\times)^{m}$-fiber bundles, we find a $\prod_{j=1}^{r_1}(\GL_{m_{1,j}}(\C)/(\C^\times)^{m_{1,j}})^{M_{1,j}}$-bundle $R_0(\tau)\to \hat R(\tau)$ with a map
    \[\hat R(\tau)\to \prod_{j=1}^{r_1}(\operatorname{Gr}(d_{1,j},r))^{M_{1,j}m_{1,j}} \times \operatorname{Gr}(d_1,r)\]
    It is equivariant if $\Gamma_\tau$ acts by permuting the factors. 
    
    Over each grasmaniann $\operatorname{Gr}(d_{1,j},r)$ we have a $\GL_{d_{1,j}}(\C)$-bundle of $d_{1,j}$ linearly independent vectors of $\C^r$ which generate the associated subspace. Over $\operatorname{Gr}(d_1,r)$ we have a $\C^{d_1(r-d_1)}\times\GL_{r-d_1}(\C)$-bundle whose fiber over a point consists of subsets $B$ of $\C^r$ of $r-d_1$ vectors whose classes form a basis of $\C^r$ module the associated subspace to the point. 
    
    Hence, pullbacking these bundles to $ R_0(\tau)$ we get a $\prod \GL_{d_{1,j}}(\C)\times \C^{d_1(r-d_1)}\times\GL_{r-d_1}(\C)$-bundle $R_1(\tau)\to \hat R(\tau)$ whose points are elements of $\hat R(\tau)$ together with basis $(B_1,\ldots, B_{r_1})$ of $W_{1,j}$ and a subset $B\subset \C^r$ such that its elements form a basis of $\C^r$ modulo $V_1$. Note that we can get a basis of $\C^r$ given by the union of all $B_j$ and  $B$. In this basis, the pair of matrices $(A,B)$ looks like
    \[\left(\begin{array}{ccccc}
        A_{11} & 0  &\cdots & 0 & *\\
        0 & A_{22} & \ddots & \vdots & \vdots \\
        \vdots &\ddots & \ddots & 0  & *\\ 
        0  & \cdots & 0 &  A_{r_1r_1} & * \\
        0 & \cdots & 0 & 0 & A
    \end{array}\right)\]
    where $A_{ii}\in R^{irr}_{d_{i,j}}$ and $A\in R(\tau')$. Therefore, we have a map
    \[\hat R_1(\tau)\to \prod_{j=1}^{r_i}R^{irr}_{\kappa_{i,j}} \times \hat R(\tau')\]
    which partially lifts $\hat R(\tau)\to \hat\mcM(\tau)$ and the existence of the fibration and the lift follows. The last claim of the lemma is due to Lemma \ref{permutation-loc-trivial-fib}.
\end{proof}

\begin{lema}
    There is a equivariant decomposition $\tilde R(\tau)\to \ov{R}(\tau)\times F_1(\tau)$ where
    \[ F_1(\tau) =\prod_{i=1}^s \GL_{r-d_i}(\C)\]
    for which $\tilde{\operatorname{Gr}}_\bullet$ factors as $\ov{\operatorname{Gr}}_\bullet:\ov{R}\to \hat\mcI(\tau)$.
\end{lema}
\begin{proof}
    We prove this by induction on $s$. For $s=1$ there is nothing to prove. Assume $s\geq 2$. Over a point $(A,B)\in \hat R(\tau)$, the fiber consists of sets $B_{i,j},B_i\subset \C^r/V_{i-1}$ such that $B_{i,j}$ is a basis of $W_{i,j}$ and the image of $B_i$ in $\C^r/V_i$ is a basis. By means of $B_i$ we can lift $B_{i+1,j}$ to vectors on $\C^r/V_{i-1}$. Hence, we are able to get a basis $B$ where the pair of matrices $(A,B)$ looks like, for $s=2$,
    \[\left(\begin{array}{cccccc}
        A_{11} & 0 & 0 & * & * & * \\
        0 & \ddots & 0 & * & * & * \\
        0 & 0 & A_{r_1r_1} & *  & * &*\\ 
        0  & 0 & 0 & A_{r_1+1,r_1+1} & 0 & 0\\
        0 & 0 & 0 & 0 & \ddots & 0  \\
        0 & 0 & 0 & 0 & 0 & A_{r_2r_2}\\
    \end{array}\right)\]
    where $A_{ij}\in R_{\kappa_{i,j}}^{irr}$.
    Note $\tilde{\operatorname{Gr}}_\bullet$ factors trough $(B_{i,j},B_i)\mapsto B$. 

    On the other hand, $B$ and all $B_i$'s determine all $B_{i,j}$. Moreover, $B_1$ is determine by $B$ up to $\GL_{r-d_1}(\C)$. Inductively, $B_i$ is determine by $B$ up to $\prod_{j=1}^i\GL_{r-d_j}(\C)$.
\end{proof}

Rename the isotypic pieces $W_{i,j}$ with $U_\alpha$ so that
\[\operatorname{Gr}_\bullet(V)=U_{\nu_{i-1}+1}\bigoplus U_{\nu_{i+1}}\bigoplus\cdots\bigoplus U_{\nu_{i}}\]
for an increasing sequence $0=\nu_0<\nu_1<\ldots<\nu_i<\nu_{i+1}<\ldots$. Forgetting the action, $U_\alpha$ is a vector space of certain dimension $d_\alpha$. Note $\sum d_\alpha =r$.

\begin{lema}
    There is an equivariant vector bundle $\ov{R}(\tau)\to \ov{R}'(\tau)$ and a factorization $\ov{\operatorname{Gr}}'_\bullet:\ov{R}'(\tau)\to\hat\mcI(\tau)$ of $\ov{\operatorname{Gr}}_\bullet$ trough $\ov{R}'(\tau)$. It has rank $\sum (d_{i}-d_{i-1})(r-d_i)-\mcD$ where
    \[\mcD:=\sum a_\alpha(\xi)a_\beta(\xi)\]
    where $\alpha$ and $\beta$ are indexes with $\beta$ of bigger degree with respect to the semisimple filtration, $\xi$ is a root of unity, and $a_\alpha(\xi)$ is the  multiplicity of $\xi$ as a eigenvalue of $\kappa_\alpha$.
\end{lema}
\begin{proof}
    Define $\ov{R}'(\tau)$ as those tuples $(\rho,\mcB)$ such that the matrix $A$ in the basis $\mcB$ has zeros on the upper triangular pieces $\Hom{U_\beta}{U_\alpha}$. 
    
    For a given $A$, take $Q_A$ a matrix such that $Q_AAQ_A^{-1}$ is diagonal and let $U_A\subset \GL_n(\C)$ be the unipotent subgroup associated with $Q_AAQ_A^{-1}$ which has basis $E_{ij}$ where $i<j$ and $\epsilon_i\neq\epsilon_j$. Note that $U_A\simeq \A^{\sum(d_i-d_{i-1})(-d_i)-\mcD}$. Locally, we have a map $\ov{R}'(\tau)\times \A^\mcD \to \ov{R}(\tau)$, $(\rho,\mcB,p)\mapsto (Q_A^{-1}pQ_A\rho Q_A^{-1} p^{-1}Q_A,\mcB)$. This map is compatible with $\ov{\operatorname{Gr}}_\bullet$ and with the given action as $U_A$ is invariant. Moreover, the action is linear on $U_A$. To finish we need to prove it is an isomorphism. Surjectivity is due to \cite[Lemma 4.3]{PM}. Injectivity is clear as $U_A$ does not intersect the center of $Q_AAQ_A^{-1}$.
\end{proof}

\begin{thm}[\protect{\cite{PM}}] 
    For any configuration of eigenvalues $\kappa$ and $\tau\in\mcT_\kappa^*$, the map $\ov{\operatorname{Gr}}'_\bullet$ is a locally trivial fibration in the Zariski topology. 
\end{thm}

They also describe the fiber. Write $\rho=(\rho_\alpha)\in \hat\mcI(\tau)$ with $\rho_\alpha$ a irreducible representation on $U_\alpha$. We have a well defined injective map $\hat\mcI(\tau)\to\GL_r(\C)^2$, $\rho\mapsto(A,B)$ by building the block-diagonal matrices with entries $\rho_\alpha(x)$ and $\rho_\alpha(y)$. Call $\hat\mcI_0(\tau)$ the variety of pairs $(A,B)$ such that $B$ is diagonal and equal to a fixed matrix with eigenvalues $\boldsymbol{\varepsilon}$.

Consider
\[\mcM_0:=\bigoplus_{\alpha\leq \nu_i<\beta}\Hom{U_\beta}{U_\alpha}\]
and for $M\in\mcM_0$ and $\rho\in \hat \mcI_0(\tau)$ the pair of matrices $\rho_M$ given by
\begin{align*}
    \rho_M(x)\big|_{U_\beta}&=\rho_\beta(x) &\text{and}&& \rho_M(y)\big|_{U_\beta}&=\rho_\beta(y)+\sum_{\alpha<\beta}M_{\alpha\beta} 
\end{align*}
Define 
\[\mcM_\rho:=\{M\in\mcM_0:\rho_M\in \hat R(\tau)\}\]
There is a more explicit description, see \cite[thm 4.4]{PM}. Fix $\rho=(A_\alpha,B_\alpha)\in\hat\mcI_0(\tau)$ and a matrix $Q$ such that $QAQ^{-1}$ is diagonal. For each $\alpha \leq  \nu_i<\beta $, let 
\[\ell_{\alpha\beta}=\langle Q_\alpha\Theta Q_\beta^{-1}B_\beta-B_\alpha Q_\alpha\Theta Q_\beta^{-1}\rangle\]
where $\Theta\in \Hom{U_\beta}{U_\alpha}$ runs over all matrices with $\Theta_{ij}=0$ whenever $\epsilon_{\alpha,i}\neq \epsilon_{\beta,i}$. For $\nu_i<\beta\leq\nu_{i+1}$, set
\[L_\beta = \bigoplus \ell_{\alpha\beta}\]
where $\nu_{i-1}<\alpha\leq \nu_i$. Define also
\[p_\beta(M) = (M_{\alpha\beta})_{\nu_{i-1}<\alpha\leq\nu_i}\in \bigoplus \Hom{U_\beta}{U_\alpha}\]
for $M\in \mcM_0$. Then the defining conditions for $\mcM_\rho\subset \mcM_0$ are that for any $\nu_i<\beta_1,\ldots,\beta_k\leq\nu_{i+1}$ isomorphic components, no non-trivial linear combination of $p_{\beta_1}(M),\ldots, p_{\beta_k}(M)$ lies in $L_{\beta_j}$ (it is independent of $j$).

Proposition 4.5 of \cite{PM} states $\mcM_\rho$ is an algebraic variety independent of $\rho$ and that $\ov{\operatorname{Gr}}'_\bullet$ is a locally trivial fibration with fiber $\mcM_\rho\times \GL_r$. Note that $\tilde R(\tau)\to \hat R(\tau)$ has absorb both $\PGL$ and gauge group quotients that appears in \cite{PM}. This is due to the fact that $\GL_r$ acts freely and transitive on the basis of $\C^r$. Write $\mcM_\tau$ for $\mcM_\rho$ for some $\rho\in\hat\mcI_0(\tau)$.

\begin{lema}
    For any configuration of eigenvalues $\kappa$ and $\tau\in\mcT_\kappa^*$, the map $\ov{\operatorname{Gr}}'_\bullet$ is a equivariantly locally trivial fibration in the Zariski topology if we endow $\mcM_\tau$ with the permutation of factors action. 
\end{lema}
\begin{proof}
    By the previous explicit description it is clear that the permutation of factors action on $\mcM_0$ preserves $M_\tau$ and, in consequence, it defines an action there. The trivializations of \cite{PM} are based on the fact that any semisimple matrix $A\in \GL_k$ has a Zariski neighborhood $U$ of all semisimple matrices and a morphism $Q:U\to \GL_k$ such that $Q(B)BQ(B)^{-1}$ is diagonal for all $B\in U$. When fixing all eigenvalues, they can be built by choosing an order to all eigenvalues $\epsilon_i$ and triangulating the linear systems $B-\lambda_i\operatorname{Id}=0$. Now, the action of $\Gamma_\tau$ just permutes eigenvalues. Henceforth, $\Gamma_\tau$ preserves the previous construction just permuting the entries of $Q$.
\end{proof}

\begin{coro}\label{cor:corefib}
    For any configuration of eigenvalues $\kappa$ and $\tau\in\mcT_\kappa^*$ which satisfy hypothesis $(H)$,
    \[[R(\tau)][F_0(\tau)]=[\A^{\sum (d_i-d_{i-1})(r-d_i)-\mcD}][F_1(\tau)][\mcM_\tau][\GL_r(\C)][\mcI(\tau)]\]
    in $K_0^{\Gamma_\tau}(\Var_\C)$.
\end{coro}
\begin{proof}
    This follows from the previous results as $\hat\mcI(\tau)=\mcI(\tau)$ under $(H)$.
\end{proof}

We have shown $[R_\kappa]$, $[F_0(\tau)]$ $[F_1(\tau)]$, $[\GL_r(\C)]$ and $[\mcM_\tau]$ belong to $A_{\Gamma_\tau}$ in Lemma \ref{Vnmrv}, Corollary \ref{pfinzq} and Lemma \ref{Rkappa-class-is-good}. Hence, as long as condition $(H)$ holds one gets inductively that $[R_\kappa^{irr}]$ is quasi-nice. Remark that by examples \ref{ex:reg-ss} and \ref{ex:2-2} we know we only need to localize on $q,q-1,q^2-1,q^3-1,q^4-1$ to compute $[R_\kappa^{irr}/\Gamma_\kappa]$. Note that the case of $(\GL_2(\C)/(\C^\times)^2)^2$ in the following result correspond to $r=4$,$s=1$, $r_1=2$, $m_{1,j}=2$ and $d_{1,j}=1$, which can be discard due to \cite[Prop 8.1]{PM}.

\begin{coro}
    For $r\leq 4$ and any $\kappa$, $[R_\kappa^{irr}]$ is quasi-nice for the multiplicative set of $K_0(\Var_\C)\otimes R_\Q(\Gamma_\kappa)\otimes \Q$ generated by $\Q$ and inductions or permutation of factors actions of $\GL_d(\C)^m$ and $(\GL_d(\C)/(\C^\times)^d)^m$ for $0\leq dm\leq r$. 
\end{coro}

After localization, we can cancel out $F_1(\tau)$ with part of $F_0(\tau)$ to get
\[[R(\tau)][F_0'(\tau)]q^\mcD=[\mcM_\tau][\GL_r(\C)][\mcI(\tau)]\]
where
\[F_0'(\tau)=\prod_{i=1}^s\left(\prod_{j=1}^{r_i} \GL_{d_{i,j}}(\C)^{m_{i,j}M_{i,j}}\times(\GL_{m_{i,j}}(\C)/(\C^\times)^{m_{i,j}})^{M_{i,j}}\right).\]

\section{Character vs representation varieties}\label{sec:pglquotient}

In this section, we prove Conjecture \ref{lem:mcM-vs-mcR} for $r\leq 3$. First, we make a reduction that holds for arbitrary $r$. It follows a similar fashion to \cite[Proposition 7.3]{GP2}. Note the formula for $[\PGL_r(\C)]$ is well known and that there is nothing to prove for $r=1$. 
    
We first pullback to $R_r^{irr}$. By Lemma \ref{lem:key-lemma}, it is enough to prove that $[R_r^{irr}]=[M_r^{irr}]\cdot [\PGL_r(\C)]$ in $K_0^{\mu_{nm}}(\Var/\C)$, where $\mu_{nm}$ acts trivially on $\PGL_r(\C)$, because the classes of $\C^\times$ and $\PGL_r(\C)$ belong to $A_{\mu_{nm}}$, $[\C^\times]^{\mu_{nm}}=[\C^\times]\otimes \Q^1$, and $[\PGL_r(\C)]^{\mu_{nm}}=[\PGL_r(\C)]\otimes \Q^1$. For this we use the stratification by eigenvalues. It suffices to show that $[R_\kappa^{irr}]=[R_\kappa^{irr}/\PGL_r][\PGL_r]$ in $K_0^{\Gamma_\kappa}(\Var/\C)$ for any $\kappa$. Or even less, that $[R_\kappa^{irr}]^{\Gamma_\kappa}=[R_\kappa^{irr}/\PGL_r]^{\Gamma_\kappa}[\PGL_r]^{\Gamma_\kappa}$. The points of $R_\kappa^{irr}/\PGL_r$ can be identified with $r$-dimensional irreducible representations $\rho:\Z_n\star\Z_m\to \GL V$ whose eigenvalues are $\kappa$ up to isomorphism.

Fix a $\kappa$. Call $\epsilon_1,\ldots,\epsilon_p$ the different eigenvalues of $A$ and $a_1,\ldots,a_p$ its multiplicities. Let $\mcO_1,\ldots,\mcO_l$ the orbits of $\Gamma_\kappa$ in $\{\epsilon_1,\ldots,\epsilon_p\}$. Note that any $\epsilon_i$ has stabilizer $\mu_m\subset\mu_{nm}$. Denote $m_i$ for $a_j$ for some $\epsilon_j\in\mcO_i$. Note that any $A\in R_\kappa^{irr}$ is diagonalizable as it has finite order. Then, there is an $\Gamma_\kappa$-equivariant algebraic map
\[R_\kappa^{irr}\to \prod_{i=1}^l\Per_{\mu_m}^{\Gamma_\kappa}(\Gr[a_i][r])=\Per_{\mu_m}^{\Gamma_\kappa}\left(\prod_{i=1}^l\Gr[a_i][r]\right)\]
given by looking at the eigenspaces of $A$, where $\mu_m$ acts trivially on $\Gr[a_i][r]$. Let $\mcL_i\to \Gr[a_i][r]$ be the bundle of basis of each subspace. They are equivariantly locally trivial fiber bundles with fiber $\GL_{a_i}$. Therefore, if $\mcL$ is the pullback by the previous map of $\Per_{\mu_m}^{\Gamma_\kappa}\left(\prod_{i=1}^l\mcL_i\right)$, $\mcL\to R_\kappa^{irr}$ is equivariantly locally trivial with fiber $\Per_{\mu_m}^{\Gamma_\kappa}\left(\prod_{i=1}^l\GL_{a_i}\right)$. An element of $\mcL$ is a couple $(A,B,\mcB_1,\ldots,\mcB_k)$ where each $\mcB_i$ is a basis of the eigenspace of $A$ with eigenvalue $\epsilon_i$. Being $A$ diagonalizable, this is the same that a matrix $M\in \GL_r(\C)$ such that $MAM^{-1}$ is the diagonal matrix with the first $a_1$ entries $\epsilon_1$, the next $a_2$ being $\epsilon_2$, and so on. The action of $\mu_{nm}$ in $M$ is given by translating by powers of a fixed permutation matrix. Quotient by $\C^\times\simeq Z(\GL_r(\C))$ with get a equivariantly locally trivial fibration $\ov{\mcL}\to R_\kappa^{irr}$.

We can repeat this with $B$ to get $\ov{\mcL}'\to R_\kappa^{irr}$ such that
\[[\ov{\mcL}']=[R_\kappa^{irr}]\cdot \left(\Per_{\mu_m}^{\Gamma_\kappa}\left(\prod_{i=1}^l \GL_{a_i}\right)/\C^\times\right) \cdot \left(\Per_{\mu_n}^{\Gamma_\kappa}\left(\prod_{j=1}^k\GL_{b_j}\right)/\C^\times\right)\]
The points of $\ov{\mcL}'$ are tuples $(A,B,M,N)\in R_\kappa^{irr}\times\PGL_r(\C)^2$ such that $MAM^{-1}$ and $NBN^{-1}$ have specific diagonal forms $\Sigma_1$ and $\Sigma_2$. This is $\Gamma_{\kappa}$-equivariantly isomorphic to the variety $X$ of tuples $(M,N)\in \PGL_r(\C)^2$ such that $(M^{-1}\Sigma_1M,N^{-1}\Sigma_2N)$ gives an irreducible representation. The $\Gamma_{\kappa}$-action is given by left translation by certain permutation matrices $\sigma_1$ and $\sigma_2$.  

Now the $\PGL_r$-action can be lifted to $\ov{\mcL}'$ by $g\cdot(A,B,M,N)=(gAg^{-1},gBg^{-1},Mg^{-1},Ng^{-1})$. Now, let 
\[Y:=\{M\in \PGL_r(\C):(M^{-1}\Sigma_1M,\Sigma_2)\text{ is irreducible}\}\]
with action given by $M\mapsto \sigma_1M \sigma_2^{-1}$. Then $\ov{\mcL}'$ is $\Gamma_{\kappa}\times\PGL_r$-isomorphic to $Y\times\PGL_r(\C)$ via $(M,N)\mapsto (MN^{-1},N)$. Hence $[\ov{\mcL}']=[\ov{\mcL}'/\PGL_r(\C)]\cdot [\PGL_r(\C)]$ in $K_0^{\Gamma_{\kappa}}(\Var/\C)$. Therefore, it suffices to show that
\[[\ov{\mcL}'/\PGL_r(\C)]^{\Gamma_\kappa} = [R_\kappa^{irr}/\PGL_r(\C)]^{\Gamma_\kappa} \left[ \left(\Per_{\mu_m}^{\Gamma_\kappa}\left(\prod_{i=1}^l\GL_{a_i}\right)/\C^\times\right) \right]^{\Gamma_\kappa} \left[\left(\Per_{\mu_n}^{\Gamma_\kappa}\left(\prod_{j=1}^k\GL_{b_j}\right)/\C^\times\right)\right]^{\Gamma_\kappa}\]
in $K_0^{\Gamma_{\kappa}}(\Var/\C)$ as long as the classes of $\Per_{\mu_m}^{\Gamma_\kappa}\left(\prod_{i=1}^l\GL_{a_i}\right)/\C^\times$ and $\Per_{\mu_n}^{\Gamma_\kappa}\left(\prod_{j=1}^k\GL_{b_j}\right)/\C^\times$ are quasi-nice.

Let $Z_1$ and $Z_2$ the subvarieties of $\prod_{i=1}^l\Gr[a_i][r]$ and $\prod_{j=1}^k\Gr[a_j][r]$ of subspaces that are in direct sum. Note that the $\PGL_r$-action is free on the image of $R_r^{irr}\to (\prod_{i=1}^l\Gr[a_i][r])\times (\prod_{j=1}^k\Gr[a_j][r])$. Indeed, if $g\in \PGL_r(\C)$ preserves every eigenspaces of $A$ and $B$, then it commutes with both of them. Being the associated representation irreducible, $g$ must be trivial. Therefore, $\ov{\mcL}'/\PGL_r(\C)$ is the pullback of 
\[\left(\Per_{\mu_m}^{\Gamma_\kappa}\left(\prod_{i=1}^l\mcL_i\right)\times \Per_{\mu_n}^{\Gamma_\kappa}\left(\prod_{i=j}^l\mcL_j\right)\right) /\PGL_r(\C)\to (\Per_{\mu_m}^{\Gamma_\kappa}(Z_1)\times \Per_{\mu_n}^{\Gamma_\kappa}(Z_2))/\PGL_r(\C) \]

Now $\PGL_r$ acts transitively on both $\Per_{\mu_m}^{\Gamma_\kappa}(Z_1)$ and $\Per_{\mu_n}^{\Gamma_\kappa}(Z_2)$. Moreover, we have $\Per_{\mu_m}^{\Gamma_\kappa}(Z_1)=((\prod \GL_{a_i})/\C^\times)\backslash \PGL_r$ and $\Per_{\mu_n}^{\Gamma_\kappa}(Z_2)=((\prod \GL_{b_j})/\C^\times)\backslash \PGL_r$. Hence, the base is isomorphic to $ ((\prod \GL_{b_j})/\C^\times)\backslash \PGL_r/((\prod \GL_{a_i})/\C^\times)$. On the other hand, we have $\Per_{\mu_m}^{\Gamma_\kappa}\left(\prod_{i=1}^l\mcL_i\right) = \PGL_r =  \Per_{\mu_n}^{\Gamma_\kappa}\left(\prod_{i=j}^l\mcL_j\right)$. To finish the proof, we will show, case by case, an equivariant stratification of a neighborhood of the image of $R_\kappa^{irr}$ such that $\PGL_r\to  ((\prod \GL_{b_j})/\C^\times)\backslash \PGL_r/((\prod \GL_{a_i})/\C^\times)$ is equivariantly locally trivial over each stratum with fiber in $A_{\Gamma_\kappa}$ and with motivic isotypic decomposition $[(\prod \GL_{b_j})/\C^\times]^{\Gamma_\kappa} [(\prod \GL_{a_i})/\C^\times]^{\Gamma_\kappa}$. Note that the action of $\Gamma_{\kappa}$ is given by $g\mapsto \sigma_1 g\sigma_2$ for certain permutation matrices $\sigma_1$ and $\sigma_2$ of coprime order.
    
\subsection{Rank two} The only possibility is that $a_1=a_2=b_1=b_2=1$ and $\Gamma_\kappa$ is either trivial or cyclic of order two, and, in the second case, acts by permuting either the two rows or the two columns. It suffices to deal with $\Gamma_\kappa$ being not trivial and acting permuting the rows. 

Note $(\prod \GL_{b_j})/\C^\times = (\prod \GL_{a_i})/\C^\times \simeq \C^\times$ as varieties. The $\Gamma_\kappa$-action on the first one is $x\mapsto x^{-1}$ and trivial on the second. To check quasi-niceness, put away the fixed point $1$ and consider the product with $\C^\times$ with action $y\mapsto -y$. Then, $(x,y)\mapsto \left(\dfrac{xy}{x-1},\dfrac{y}{x-1}\right)$ and $(\lambda_1,\lambda_2)\mapsto (\lambda_1\lambda_2^{-1},\lambda_1-\lambda_2)$ are inverse equivariant isomorphisms between the previous product and $\{(\lambda_1,\lambda_2)\in\C^\times:\lambda_1\neq\lambda_2\}$ with action $(\lambda_1,\lambda_2)\mapsto (\lambda_2,\lambda_1)$. The last action has class on $A_{\Gamma_\kappa}$.

Let $U\subset \PGL_2(\C)$ the subset of matrices with all entries non-zero. Note that the image of $R_r^{irr}$ falls into $U$. The quotient map $\PGL_2\mapsto \C^\times\backslash \PGL_2/\C^\times$ turns out to be
\[\left(\begin{array}{cc}
    a & b \\
    c & d
\end{array}\right) \mapsto \frac{ad}{bc}\in \C^\times\setminus\{1\}.\]
The $\Gamma_\kappa$-action in the quotient is $x\mapsto x^{-1}$.

Over $\C^{\times}\setminus\{\pm 1\}$ we have the equivariant section 
\[x\mapsto \left(\begin{array}{cc}
    1+x & 1 \\
    1+x^{-1} & 1
\end{array}\right) \]
Therefore, over this open set we have the desired multiplicativity property.

Now, the fiber over $-1$ is
\[F_{-1}=\left\{ \left(\begin{array}{cc}
    -xy & x \\
    y & 1
\end{array}\right) : x,y\in \C^\times\right\}\]
that is isomorphic to $\C^{\times}\times \C^{\times}$ with action $(x,y)\mapsto (x^{-1},-y)$. Now, being the action on the second factor linear with trivial isotypic decomposition, we have
\[ [F_{-1}]^{\Gamma_{\kappa}} = [(\prod \GL_{b_j})/\C^\times]^{\Gamma_\kappa} [(\prod \GL_{a_i})/\C^\times]^{\Gamma_\kappa}\]
as desired.

\subsection{Rank three} In this rank, there are three cases that cover all necessary ones. All $a_i$ and $b_j$ equal to one and $\Gamma_\kappa$ of order six, acting by shifting the rows and permuting the first two columns. Or all $b_j$ equal to one, $a_1=2$, $a_2=1$, $\Gamma_\kappa$ of order three, acting by shifting the rows, or two, acting by permuting the first two rows.  

Let us start with $a_i=b_j=1$ and $\Gamma_\kappa$ of order six. In this case, $(\prod \GL_{b_j})/\C^\times = (\prod \GL_{a_i})/\C^\times \simeq (\C^\times)^{2}$ with corresponding actions $(\lambda_1,\lambda_2)\mapsto (\lambda_2^{-1},\lambda_1\lambda_2^{-1})$ and $(\mu_1,\mu_2)\mapsto (\mu_2,\mu_1)$. We need to check that the first action is quasi-nice. Note that $1+\lambda_1+\lambda_2=0$ is an invariant subspace. On its complement, after multiplying with $\C^\times$ with trivial action, we have the inverse isomorphisms 
\[(\lambda_1,\lambda_2,s)\mapsto \left(\dfrac{\lambda_1 s}{1+\lambda_1+\lambda_2},\dfrac{\lambda_2 s}{1+\lambda_1+\lambda_2},\dfrac{s}{1+\lambda_1+\lambda_2}\right)\]
and $(x,y,z)\mapsto (xz^{-1},yz^{-1},x+y+z)$ with $\{(x,y,z)\in \C^\times:x+y+z\neq 0\}$ with action $(x,y,z)\mapsto (z,x,y)$ which has class on $A_{\Gamma_\kappa}$. The invariant subspace is isomorphic to $\C^\times\setminus\{-1\}$ with action $\lambda\mapsto -\dfrac{1}{1+\lambda}$. We choose a root of unity $\xi$ of order three and take out the fixed point $\xi$. Consider the linear $\Gamma_\kappa$-actions $(x,y)\mapsto (y,-x-y)$ on $\C^2$ and $z\mapsto \xi z$ on $\C^\times$. We have the inverse equivariant isomorphisms 
\[(\lambda,z)\in (\C^\times\setminus\{-1,\xi\})\times \C^\times\mapsto \left( \frac{\lambda z}{\lambda-\xi},\frac{z}{\lambda-\xi}\right)\in \{(x,y)\in (\C^\times)^2: x\neq \xi y, -y\} \]
and $(x,y)\mapsto (xy^{-1}, x -\xi y)$. Now, the right hand side is $\C^2$ with a linear action less $\C\simeq \{x=\xi y\}$ with a linear action and $\Ind_{\{e\}}^{\Gamma_\kappa}(\C\setminus\{0\})\simeq\{x=0,y\neq 0\}\cup \{x+y=0,x,y\neq 0\}\cup \{y=0,x\neq 0\}$. Hence, its class belongs to $A_{\Gamma_\kappa}$. This shows that $(\C^\times)^2$ with action $(\lambda_1,\lambda_2)\mapsto(\lambda_2^{-1},\lambda_1\lambda_2^{-1})$ is quasi-nice. 

We will work over the open set $U\subset \PGL_3(\C)$ given by those matrices with no two zeros in the same column or row. If a matrix does not satisfy this and is induced by a representation $(A,B)$, either $A$ and $B$ share a common eigenvector or a two dimensional invariant subspace, any of these conditions implies that the representation is not irreducible. This means that the image of $R_\kappa^{irr}$ is inside of $U$. We stratify $U$ by the vanishing of the entries. Call $V$ the open stratum of matrices with non-zero entries.

Over $V$ we have a similar description of the quotient than before. It is 
\[\left(\begin{array}{ccc}
    a & b & c\\
    d & e & f\\
    g & h & i
\end{array}\right) \mapsto \left(\begin{array}{cc}
    \frac{ai}{gc} & \frac{bi}{hc} \\
    \frac{di}{gf} & \frac{ei}{hf} \\
\end{array}\right).\]
The induced $\Gamma_\kappa$ action is 
\[ \left(\begin{array}{cc}
    x & y \\
    z & w \\
\end{array}\right)\mapsto \left(\begin{array}{cc}
    w^{-1} & z^{-1} \\
    yw^{-1} & xz^{-1} \\
\end{array}\right).\]

We have a equivariant section 
\[\left(\begin{array}{cc}
    x & y \\
    z & w \\
\end{array}\right)\mapsto \left(\begin{array}{ccc}
    x(1+x^{-1}+z^{-1}) & y(1+y^{-1}+w^{-1}) & 1\\
    z(1+x^{-1}+z^{-1}) & w(1+y^{-1}+w^{-1}) & 1\\
    1+x^{-1}+z^{-1} & 1+y^{-1}+w^{-1} & 1
\end{array}\right)\]
defined when $1+x^{-1}+z^{-1}$ and $1+y^{-1}+w^{-1}$ do not vanish.

We have another section. Let $\xi$ a primitive root of unity or order three. Consider
\[\left(\begin{array}{cc}
    x & y \\
    z & w \\
\end{array}\right)\mapsto \left(\begin{array}{ccc}
    x(1+\xi x^{-1}+\xi^{2}z^{-1}) & y(1+\xi y^{-1}+\xi^2 w^{-1}) & 1\\
    z(1+\xi x^{-1}+\xi^2z^{-1}) & w(1+\xi y^{-1}+\xi^2 w^{-1}) & 1\\
    1+\xi x^{-1}+\xi^2z^{-1} & 1+\xi y^{-1}+\xi^2w^{-1} & 1
\end{array}\right)\]
This is a non-equivariant section defined when $1+\xi x^{-1}+\xi^2z^{-1}$ and $1+\xi y^{-1}+\xi^2w^{-1}$ do not vanish. Indeed,
\[\left(\begin{array}{cc}
    w^{-1} & z^{-1} \\
    yw^{-1} & xz^{-1} \\
\end{array}\right)\mapsto \left(\begin{array}{ccc}
    \xi(1+\xi y^{-1}+\xi^2w^{-1}) & \xi(1+\xi x^{-1}+\xi^2z^{-1}) & 1\\
    \xi y(1+\xi y^{-1}+\xi^2 w^{-1}) & \xi x(1+\xi x^{-1}+\xi^{2}z^{-1}) & 1\\
    \xi w(1+\xi y^{-1}+\xi^2 w^{-1}) & \xi z(1+\xi x^{-1}+\xi^2z^{-1}) & 1
\end{array}\right)\]
But, if we change the action on $(\mu_1,\mu_2)$ to $(\mu_1,\mu_2)\mapsto (\xi\mu_2,\xi\mu_1)$, then
\[\left((\lambda_1,\lambda_2),\left(\begin{array}{cc}
    x & y \\
    z & w \\
\end{array}\right),(\mu_1,\mu_2)\right)\mapsto \left(\begin{array}{ccc}
    \lambda_1\mu_1 x(1+\xi x^{-1}+\xi^{2}z^{-1}) & \lambda_1\mu_2 y(1+\xi y^{-1}+\xi^2 w^{-1}) & \lambda_1\\
    \lambda_2\mu_1 z(1+\xi x^{-1}+\xi^2z^{-1}) & \lambda_2\mu_2 w(1+\xi y^{-1}+\xi^2 w^{-1}) & \lambda_2\\
    \mu_1 (1+\xi x^{-1}+\xi^2z^{-1}) & \mu_2(1+\xi y^{-1}+\xi^2w^{-1}) & 1
\end{array}\right)\]
is an equivariant isomorphism. The fiber has not the correct action, but has class on $A_{\Gamma_\kappa}$ and the change we made does not change the motivic isotypic decomposition. Hence, this section works for our purpose. Note the same strategy can be carry over with $\xi^2$.

Note that for any pair $(x,z)$ at least one among $1+x^{-1}+z^{-1}$, $1+\xi x^{-1}+ \xi^2 z^{-1}$ and $1+\xi^2 x^{-1}+ \xi z^{-1}$ must be non-zero, as they add up to $3$. We stratify $V$ according to the vanishing or not of each of the functions $1+x^{-1}+z^{-1}$, $1+\xi x^{-1}+ \xi^2 z^{-1}$, $1+\xi^2 x^{-1}+ \xi z^{-1}$, $1+y^{-1}+w^{-1}$, $1+\xi y^{-1}+ \xi^2 w^{-1}$ and $1+\xi^2 y^{-1}+ \xi w^{-1}$. If one of the products $(1+x^{-1}+z^{-1})(1+y^{-1}+w^{-1})$, $(1+\xi x^{-1}+ \xi^2 z^{-1})(1+\xi y^{-1}+ \xi^2 w^{-1})$ and $(1+\xi^2 x^{-1}+ \xi z^{-1})(1+\xi^2 y^{-1}+ \xi w^{-1})$ does not vanish, we can use one of the previous sections. This holds for all $\Gamma_\kappa$-invariant stratum. Over the non-invariant ones, we need to work only with the permutation of rows actions and then induce. Hence, we can simply combine the previous sections. For example, if $1+x^{-1}+z^{-1}$ and $1+\xi y^{-1}+ \xi^2 w^{-1}$ are not zero, then
\[\left((\lambda_1,\lambda_2),\left(\begin{array}{cc}
    x & y \\
    z & w \\
\end{array}\right),(\mu_1,\mu_2)\right)\mapsto \left(\begin{array}{ccc}
    \lambda_1\mu_1 x(1+ x^{-1}+z^{-1}) & \lambda_1\mu_2 y(1+\xi y^{-1}+\xi^2 w^{-1}) & \lambda_1\\
    \lambda_2\mu_1 z(1+ x^{-1}+z^{-1}) & \lambda_2\mu_2 w(1+\xi y^{-1}+\xi^2 w^{-1}) & \lambda_2\\
    \mu_1 (1+ x^{-1}+z^{-1}) & \mu_2(1+\xi y^{-1}+\xi^2w^{-1}) & 1
\end{array}\right)\]
is a $\Z/3\Z$-equivariant isomorphism if $(\mu_1,\mu_2)$ is endow with the action $(\mu_1,\mu_2)\mapsto (\xi \mu_1,\mu_2)$.

We now need to analyze the complement of $V$. Any of its strata is stabilized only by the identity of $\Gamma_\kappa$. Hence, it suffices to show just a section over each one. This is straightforward and we will skip this computation.

Let us analyze $a_1=2$ and $a_2=1=b_1=b_2=b_3$ with $\Gamma_\kappa$ of order three. In this case, $(\prod \GL_{b_j})/\C^\times \simeq (\C^\times)^{2}$ with action $(\lambda_1,\lambda_2)\mapsto (\lambda_2^{-1},\lambda_1\lambda_2^{-1})$ and $(\prod \GL_{a_i})/\C^\times\simeq \GL_2(\C)$  with the trivial action. We will work over the open subset $U\subset \PGL_3(\C)$ defined as
\[\left\{\left(\begin{array}{ccc}
    a & b & c\\
    d & e & f\\
    g & h & i
\end{array}\right): ae-bd,dh-ge,ah-bg,c,f,i\neq 0 \right\}\]
Note that if a matrix with $ae-bd=0$ is associated to a representation $(A,B)$, the two dimensional eigenspace of $A$ contains the third eigenvector of $B$, and therefore the representation is reducible.

Over $U$ the quotient is
\[\left(\begin{array}{ccc}
    a & b & c\\
    d & e & f\\
    g & h & i
\end{array}\right) \mapsto \left(\begin{array}{cc}
    \frac{a}{c} & \frac{b}{c}
\end{array}\right)\left(\begin{array}{cc}
    \frac{d}{f} & \frac{e}{f} \\
    \frac{g}{i} & \frac{h}{i} \\
\end{array}\right)^{-1}\in (\C^\times)^2.\]
The $\Gamma_\kappa$-action on the quotient is
\[(x,y)\mapsto (y^{-1},-y^{-1}x)\]
We have the equivariant section
\[(x,y)\mapsto \left(\begin{array}{ccc}
    x & y & 1\\
    1 & 0 & 1\\
    0 & 1 & 1
\end{array}\right)\left(\begin{array}{ccc}
    1+x^{-1} & y-y^{-1} & 0\\
    1-xy^{-1} & x^{-1}-x & 0\\
    0 & 0 & 1
\end{array}\right)\]
defined over $(1+x^{-1})(x^{-1}-x)\neq (y-y^{-1})(1-xy^{-1})$.

If we choose a root of unity of order three $\xi$ and change the action on $\GL_2(\C)$ to be $A\mapsto \xi A$, we can use
\[(x,y)\mapsto \left(\begin{array}{ccc}
    x & y & 1\\
    1 & 0 & 1\\
    0 & 1 & 1
\end{array}\right)\left(\begin{array}{ccc}
    1+\xi^2 x^{-1} & y-\xi^2 y^{-1} & 0\\
    \xi- xy^{-1} & \xi x^{-1}- x & 0\\
    0 & 0 & 1
\end{array}\right)\]
over $(1+\xi^2 x^{-1})(\xi x^{-1}- x)\neq (y-\xi^2y^{-1})(\xi-xy^{-1})$. We have a similar map over $(1+\xi x^{-1})(\xi^2 x^{-1}- x)\neq (y-\xi y^{-1})(\xi^2- xy^{-1})$. We can also swap the first two columns of the most right hand side matrix in each of the previous maps to define three more. 

The complement of the previous six domains is $x=y=-1$ cup $(x,y)=(\omega,\omega^{-1})$ for a root of order six $\omega$. Over $(-1,-1)$ the fiber is $(\C^\times)^2\times \GL_2(\C)$ with action $(\lambda_1,\lambda_2,A)\mapsto (\lambda_2^{-1}, \lambda_1\lambda_2^{-1}, \left(\begin{smallmatrix}
  -1 & -1\\
  1 & 0
\end{smallmatrix}\right)A)$ which has the correct motivic isotypic decomposition. Over $(\omega,\omega^{-1})$, the fiber is $(\C^\times)^2\times \GL_2(\C)$ with action $(\lambda_1,\lambda_2,A)\mapsto (\lambda_2^{-1}, \lambda_1\lambda_2^{-1}, \left(\begin{smallmatrix}
  \omega & \omega^{-1}\\
  1 & 0
\end{smallmatrix}\right)A)$ which again has the correct motivic isotypic decomposition.

To finish with rank three, we have to analyze $a_1=2$ and $a_2=1=b_1=b_2=b_3$ with $\Gamma_\kappa$ of order two. The action on $(\prod \GL_{b_j})/\C^\times \simeq (\C^\times)^{2}$ is now $(\lambda_1,\lambda_2)\mapsto(\lambda_2,\lambda_1)$. The open set $U$ and the quotient are the same, but $\Gamma_\kappa$-action is now $(x,y)\mapsto (x^{-1},-x^{-1}y)$. 

We have the equivariant section
\[(x,y)\mapsto \left(\begin{array}{ccc}
    x & y & 1\\
    1 & 0 & 1\\
    0 & 1 & 1
\end{array}\right)\left(\begin{array}{ccc}
    -y & -y+1+x^{-1} & 0\\
    x+x^{-1} & x+x^{-1} & 0\\
    0 & 0 & 1
\end{array}\right) \]
over $(x+1)(x^2+1)\neq 0$. If we change the action on $\GL_2(\C)$ to $A\mapsto -A$, we can use
\[(x,y)\mapsto \left(\begin{array}{ccc}
    x & y & 1\\
    1 & 0 & 1\\
    0 & 1 & 1
\end{array}\right)\left(\begin{array}{ccc}
    -y & -y+1-x^{-1} & 0\\
    x-x^{-1} & x-x^{-1} & 0\\
    0 & 0 & 1
\end{array}\right) \]
over $(x-1)(x^2-1)\neq 0$. 

The remaining set is $x=-1$ with trivial action, i.e. $\{-1\}\times \C^\times$. The preimage of this set is $(\C^\times)\times \C^\times \times \GL_2(\C)$ with action $(\lambda_1,\lambda_2,y,A)\mapsto (\lambda_2,\lambda_1,y,\left(\begin{smallmatrix}
    -1 & y \\
    0 & 1
\end{smallmatrix}\right)A)$. Hence, it suffices to show that if $\C^\times \times \GL_2(\C)$ is endow with the $\Gamma_\kappa$-action $(y,A)\mapsto (y,\left(\begin{smallmatrix}
    -1 & y \\
    0 & 1
\end{smallmatrix}\right)A)$ then its class lives in $A_{\Gamma_\kappa}$ and it has a trivial motivic isotypic decomposition. Now, consider the automorphism 
\[\left(y,\left(\begin{array}{cc}
    a & b \\
    c & d
\end{array}\right)\right)\mapsto\left(y,\left(\begin{array}{cc}
    a & b \\
    yc & yd
\end{array}\right)\right)\]
of $\C^\times \times \GL_2(\C)$. It changes the action to $(y,A)\mapsto (y, \left(\begin{smallmatrix}
    -1 & 1 \\
    0 & 1
\end{smallmatrix}\right)A)$. We have a product of $\C^\times$ with a trivial action and $\GL_2(\C)$ with the multiplication of a finite order matrix. Both of their classes belong to $A_{\Gamma_\kappa}$ and have trivial isotypic decomposition.

\section{Computations for low ranks}

Whenever $r$ is coprime with $n$ and $m$, we can simply use formulas from \cite{PM} thanks to Theorem \ref{coprime-case}. Hence, our focus will be on the non-coprime case. Moreover, we only need to focus on configurations of eigenvalues and types with non-trivial stabilizer by Lemmas \ref{correction-lemma} and \ref{big-correction-lemma}.

\subsection{Rank one} Here $R_\kappa^{irr}=R_\kappa$ is a point for any $\kappa$. In addition, there are no stabilizers as $\gcd(1,nm)=1$. Hence,
\[[R_1^{irr}]= \sum_\kappa\frac{1}{|\mu_{nm}|}\Ind_{\{e\}}^{\mu_{nm}}(*) \]
Now $\kappa$ belongs to $\mu_{n}\times \mu_{m}$. Therefore,
\[R_1^{irr}= \mu_{nm} \]
with the translation action. This implies
\[[\mcR_1^{irr}] = \langle T_{\mu_{nm}},((q-1)\otimes T_{\mu_{nm}})\Q^{\mu_{nm}}\rangle = q-1.\]

\subsection{Rank two} By remark 3.5 of \cite{PM}, the only relevant $\kappa$'s are $((\epsilon_1,\epsilon_2),(\varepsilon_1,\varepsilon_2))$ with $\epsilon_1\neq \epsilon_2$ and $\varepsilon_1\neq \varepsilon_2$. 

If $\gcd(2,nm)=1$, there are no stabilizers and we can simply use González-Prieto and Muñoz's formula(see subsection 7.1 in loc.cit.):
\[R_\kappa^{irr} = q^4 - 2q^3 - q^2 + 2q\]
Now there are 
\[\frac{1}{4}nm(n-1)(m-1)\]
possible $\kappa$'s. Hence,
\[[R_2^{irr}]^{\mu_{nm}} = \frac{1}{4}(n-1)(m-1)(q^4 - 2q^3 - q^2 + 2q) \otimes \Q^{nm} \]
in this case. Therefore,
\begin{align*}
    [\mcR_2^{irr}] &= \langle T_{\mu_{nm}},((q-1)\otimes T_{\mu_{nm}})(\frac{1}{4}(n-1)(m-1)(q^4 - 2q^3 - q^2 + 2q) \otimes \Q^{nm})\rangle\\
    &= \frac{1}{4}(n-1)(m-1)(q-1)(q^4 - 2q^3 - q^2 + 2q).
\end{align*}

Now assume $n$ is even and $m$ is odd. There is only one case with non-trivial stabilizer namely $\kappa=((\epsilon_1,-\epsilon_1),(\varepsilon_1,\varepsilon_2))$. There are two kind of types $\tau\in\mcT_\kappa^*$:
\begin{itemize}
    \item the semisimple ones: we have $s=1$, $r_1=2$, $d_{i,j}=m_{i,j}=M_{i,j}=1$. As $\varepsilon_1\neq \varepsilon_2$, $\Gamma_\tau$ is trivial. Hence, we can use again \cite{PM}:
    \[ [R(\tau)] = q^2+q \]
    There are two possible types in this case. Namely $((\epsilon_1,\varepsilon_1),(-\epsilon_1,\varepsilon_2))$ and $((\epsilon_1,\varepsilon_2),(-\epsilon_1,\varepsilon_1))$.
    
    \item the non-semisimple ones: $\Gamma_\tau$ is also trivial in this case. By loc.cit.
    \[ R(\tau) = q^3-q\]
    Note this case encloses four types. 
\end{itemize}

Hence,
\[ [R_\kappa^{red}]^{\Gamma_\kappa} = (2q^3+q^2-q)\otimes \Q^2\]
On the other hand,
\begin{align*}
    [ R_\kappa]^{\Gamma_\kappa} &= ((q^2-q)\otimes \Q^1+q\otimes \Q^2) (q(q+1)\otimes \Q^1)\\
    &= (q^2-q)(q^2+q)\otimes\Q^1+q^2(q+1)\otimes \Q^2
\end{align*}
by example \ref{ex:reg-ss}. Hence,
\begin{align*}
    [R_\kappa^{irr}]^{\Gamma_\kappa} &=(q^2-q)(q^2+q)\otimes\Q^1+q^2(q+1)\otimes \Q^2  -(2q^3+q^2-q)\otimes \Q^2\\
    &= (q^2-q)(q^2+q)\otimes\Q^1+(-q^3+q)\otimes \Q^2
\end{align*}
for $\kappa=((\epsilon_1,-\epsilon_1),(\varepsilon_1,\varepsilon_2))$. There are $\frac{1}{4}nm(m-1)$ possibilities for this kind of $\kappa$. For the others $\kappa$ we can use the same formula as before. Summing up all together we get

\begin{align*}
    [R^{irr}_2]^{\mu_{nm}} &= \frac{1}{2}(m-1)((q^2-q)(q^2+q)\otimes\Q^{\frac{nm}{2}}+(q-q^3)\otimes \Q^{nm})\\ &+\frac{1}{4}(n-2)(m-1) (q^4-2q^3-q^2+2q)\otimes \Q^{nm}
\end{align*}

and

\begin{align*}
    [\mcR_2^{irr}] &= (q-1)(\frac{1}{2}(m-1)((q^2-q)(q^2+q)+(-q^3+q))))+\frac{1}{4}(n-2)(m-1) (q^4-2q^3-q^2+2q))\\
    &= \frac{1}{4}(n-1)(m-1)(q-1)(q^4-2q^3-q^2+2q)+ \frac{1}{2}(m-1)(q-1)(q^4-q^3-q^2+q)
\end{align*}

if $n$ is even. In conclusion, for $r=2$, we have \[[\mcR_2^{irr}(\GL)]=\frac{q-1}{2}[\mcR_2^{irr}(\SL)]^*+\delta_{2|n}\frac{1}{4}(m-1)(q-1)(q^2-q)(q^2-1)+\delta_{2|m}\frac{1}{4}(n-1)(q-1)(q^2-q)(q^2-1)\]

Quotient by the $\PGL_2(\C)$ action we get
\[[\mcM^{irr}_2(\GL)]=\frac{q-1}{2}[\mcM^{irr}_2(\SL)]^*+\delta_{2|n}\frac{1}{2}(m-1)(q-1)^2+\delta_{2|m}\frac{1}{2}(n-1)(q-1)^2\]
which agrees with \cite[Proposition 7.3]{sl3}.

\subsection{Rank three}

As in the previous case, when $\gcd(3,nm)=1$ we can use González-Prieto and Muñoz's formulas:
\[[R_\kappa^{irr}]=\left\{\begin{array}{cc}
    \begin{array}{c}q^{12} + 4q^{11} - 10q^{10} - 8q^9 + 17q^8 + 13q^7 \\
    - 5q^6 - 21q^5 - 3q^4 + 12q^3\end{array} & \text{if there are no repeated eigenvalues}  \\
    q^{10} - 3q^9 + 2q^8 + 2q^7 - 2q^5 - 3q^4 + 3q^3 & \text{if there is exactly one repeated eigenvalue} \\
    0 & \text{in any other case}
\end{array}\right.\]
Henceforth,
\begin{align*}
    [R_3^{irr}]^{\mu_{nm}} &= \frac{1}{nm}\left(\binom{n}{3}\binom{m}{3}(q^{12} + 4q^{11} - 10q^{10} - 8q^9 + 17q^8 + 13q^7 - 5q^6 - 21q^5 - 3q^4 + 12q^3)\right.\\
    &+\left. \left(2\binom{n}{2}\binom{m}{3}+2\binom{n}{3}\binom{m}{2}\right)(q^{10} - 3q^9 + 2q^8 + 2q^7 - 2q^5 - 3q^4 + 3q^3)\right)\otimes \Q^{nm}
\end{align*}
and
\begin{align*}
    [\mcR_3^{irr}(\GL)]=\frac{q-1}{3}[\mcR_3^{irr}(\SL)]
\end{align*}
if $\gcd(3,nm)=1$.

Now assume that $3|n$ and fix a root of unity $\omega$ of order $3$. For $\Gamma_\kappa$ to be non-trivial, it must be $\kappa=((\epsilon_1,\omega\epsilon_1,\omega^2\epsilon_1),(\kappa_1,\kappa_2,\kappa_3))$ or $((\epsilon_1,\omega\epsilon_1,\omega^2\epsilon_1),(\kappa_1,\kappa_1,\kappa_2))$ with $\kappa_i\neq \kappa_j$ if $i\neq j$. In any case, no $\tau\in\mcT_\kappa^*$ has non trivial stabilizer. Hence,
\[[R^{red}_\kappa]^{\Gamma_\kappa}=\frac{1}{3}[R^{red}_\kappa]\otimes \Q^3\]
and we can use \cite{PM}. 

For $\kappa=((\epsilon_1,\omega\epsilon_1,\omega^2\epsilon_1),(\kappa_1,\kappa_2,\kappa_3))$ we find
\begin{align*}
    [R_\kappa]^{\Gamma_\kappa} &= ((q^3-q^2)(q^3-q)\otimes \Q^1 + q^4(q+1)\otimes \Q^3)(q^2(q^2+q)(q^2+q+1)\otimes \Q^1)\\
    &= (q^3-q^2)(q^3-q)q^2(q^2+q)(q^2+q+1)\otimes \Q^1 + q^4(q+1)(q^2(q^2+q)(q^2+q+1)\otimes \Q^3
\end{align*}
by example \ref{ex:reg-ss}. Hence
\begin{align*}
    [R_\kappa^{irr}]^{\Gamma_\kappa} &= (q^3-q^2)(q^3-q)q^2(q^2+q)(q^2+q+1)\otimes \Q^1 + q^4(q+1)q^2(q^2+q)(q^2+q+1)\otimes \Q^3\\
    &- (6q^{10} + 3q^9 - 3q^8 - 3q^7 + 2q^6 + 7q^5 + q^4 - 4q^3)\otimes \Q^3\\
    &=(q^3-q^2)(q^3-q)q^2(q^2+q)(q^2+q+1)\otimes \Q^1 \\&+q^3  (q^8 - 3q^7 + q^6 + 6q^5 + 4q^4 - 2q^3 - 7q^2 - q + 4)\otimes \Q^3
\end{align*}
 
For $\kappa=((\epsilon_1,\omega\epsilon_1,\omega^2\epsilon_1,(\kappa_1,\kappa_1,\kappa_2))$ we have
\begin{align*}
    [R_\kappa]^{\Gamma_\kappa} &= ((q^3-q^2)(q^3-q)\otimes \Q^1 + q^4(q+1)\otimes \Q^3)(q^2(q^2+q+1)\otimes \Q^1)\\
    &= (q^3-q^2)(q^3-q)q^2(q^2+q+1)\otimes \Q^1 + q^4(q+1)q^2(q^2+q+1)\otimes \Q^3
\end{align*}
by example \ref{ex:reg-ss}. Hence
\begin{align*}
    [R_\kappa^{irr}]^{\Gamma_\kappa} &= (q^3-q^2)(q^3-q)q^2(q^2+q+1)\otimes \Q^1 + q^4(q+1)q^2(q^2+q+1)\otimes \Q^3\\
    &- (2q^9 + q^8 + q^7 + q^6 + q^5 + q^4 - q^3)\otimes \Q^3\\
    &=(q^3-q^2)(q^3-q)q^2(q^2+q+1)\otimes \Q^1-(q + 1) (q - 1)^3 q^3 (q^2 + q + 1)\otimes \Q^3
\end{align*}

Hence, for $3|n$,

\begin{align*}
    [\mcR_3^{irr}(\GL)]- &\frac{q-1}{3}[\mcR_3^{irr}(\SL)]= \frac{2}{3}(m-1)(\frac{1}{2}(m-2)q(q+1)+1)(q-1)^3q^5(q+1)(q^2+q+1)
\end{align*}

and

\begin{align*}
    [\mcM^{irr}_3(\GL)]= \frac{q-1}{3}[\mcM^{irr}_3(\SL)] + \frac{2}{3}(\frac{1}{6}(m-1)(m-2)(q^2+q)+(m-1)))(q-1)q^2 
\end{align*}

which also matches the formulas of \cite[Proposition 10.1]{sl3}.

\subsection{Rank four}

The computation gets more subtle as there are non trivial $\Gamma_\tau$ now. There are several kinds of $\kappa$ with non-trivial stabilizer. Some of them can be rule out because $R_\kappa^{irr}=\emptyset$. There are given by Proposition 8.1 and Remark 3.5 of \cite{PM}. Remaining ones are in the following table:
\begin{table}[H]
    \centering
    \begin{tabular}{c|c|c|c}
        $n$ & $\kappa$ & $\#\kappa$ & $\Gamma_\kappa$  \\
        \hline
        $4|n$ & $\kappa_1=((\epsilon,i\epsilon,-\epsilon,-i\epsilon),(\varepsilon_1,\varepsilon_2,\varepsilon_3,\varepsilon_4))$ & $\frac{n}{4}\binom{m}{4}$ & $\mu_4$ \\
        $4|n$ & $\kappa_2=((\epsilon,i\epsilon,-\epsilon,-i\epsilon),(\varepsilon_1,\varepsilon_1,\varepsilon_2,\varepsilon_3))$ &$\frac{n}{4}m\binom{m-1}{2}$ & $\mu_4$ \\
        $4|n$ & $\kappa_3=((\epsilon,i\epsilon,-\epsilon,-i\epsilon),(\varepsilon_1,\varepsilon_1,\varepsilon_2,\varepsilon_2))$ & $\frac{n}{4}\binom{m}{2}$ & $\mu_4$ \\
        $4|n$ & $\kappa_4=((\epsilon,i\epsilon,-\epsilon,-i\epsilon),(\varepsilon_1,\varepsilon_1,\varepsilon_1,\varepsilon_2))$ & $\frac{n}{4}m(m-1)$ & $\mu_4$ \\
        $2|n$ & $\kappa_5=((\epsilon_1,-\epsilon_1,\epsilon_2,-\epsilon_2),(\varepsilon_1,\varepsilon_2,\varepsilon_3,\varepsilon_4))$ & $\frac{n(n-4+2\delta_{4\not|n})}{8}\binom{m}{4}$ & $\mu_2$\\ 
        $2|n$ & $\kappa_6=((\epsilon_1,-\epsilon_1,\epsilon_2,-\epsilon_2),(\varepsilon_1,\varepsilon_1,\varepsilon_2,\varepsilon_3))$ & $\frac{n(n-4+2\delta_{4\not|n})}{8}m\binom{m-1}{2}$ & $\mu_2$\\            
        $2|n$ & $\kappa_7=((\epsilon_1,-\epsilon_1,\epsilon_2,-\epsilon_2),(\varepsilon_1,\varepsilon_1,\varepsilon_2,\varepsilon_2))$ & $\frac{n(n-4+2\delta_{4\not|n})}{8}\binom{m}{2}$ & $\mu_2$\\            
        $2|n$ & $\kappa_8=((\epsilon_1,-\epsilon_1,\epsilon_2,-\epsilon_2),(\varepsilon_1,\varepsilon_1,\varepsilon_1,\varepsilon_2))$ & $\frac{n(n-4+2\delta_{4\not|n})}{8}m(m-1)$ & $\mu_2$\\
        $2|n$ & $\kappa_9=((\epsilon,-\epsilon,\epsilon,-\epsilon),(\varepsilon_1,\varepsilon_2,\varepsilon_3,\varepsilon_4))$ & $\frac{n}{2}\binom{m}{4}$ & $\mu_2$\\
        $2|n$ & $\kappa_{10}=((\epsilon,-\epsilon,\epsilon,-\epsilon),(\varepsilon_1,\varepsilon_1,\varepsilon_2,\varepsilon_3))$ & $\frac{n}{2}m\binom{m-1}{2}$ & $\mu_2$
    \end{tabular}
    \caption{Configurations of eigenvalues with non-trivial stabilizer on rank four.}
\end{table}
where $\epsilon_i\neq\pm\epsilon_j,\pm i\epsilon_j$ and $\varepsilon_i\neq\varepsilon_j$ for any $i\neq j$. 

Let us start our analysis with $\kappa_7$. Note that for all types we must have $m_{i,j}=M_{i,j}=1$. Relevant types are:
\begin{table}[H]
    \centering
    \begin{tabular}{c|c|c|c|c|c|c}
        $\tau$ & $\#\tau$ & $s$ & $r_i$ & $d_{i,j}$ & $\kappa_{1,j}$ & $\kappa_{2,j}$   \\
        \hline
        $\tau_1$ & 2 & 1 & 4 & 1 & $((\epsilon_1),(\varepsilon_i)),((-\epsilon_1),(\varepsilon_i)),((\epsilon_2),(\varepsilon_j)),((-\epsilon_2),(\varepsilon_j))$ & -\\
        $\tau_2$ & 1 & 1 & 2 & 2 & $((\epsilon_1,-\epsilon_1),(\varepsilon_1,\varepsilon_2)),((\epsilon_2,-\epsilon_2),(\varepsilon_1,\varepsilon_2))$& - \\
        $\tau_3$ & 2 & 1 & 2 & 2 & $((\epsilon_1,\pm\epsilon_2),(\varepsilon_1,\varepsilon_2)),((-\epsilon_1,\mp\epsilon_2),(\varepsilon_1,\varepsilon_2))$&-\\
        $\tau_4$ & 4 & 2 & 2 & 1 & $((\epsilon_k),(\varepsilon_i)),((-\epsilon_k),(\varepsilon_i))$ & $((\epsilon_l),(\varepsilon_j)),((-\epsilon_l),(\varepsilon_j))$\\
        $\tau_5$ & 2 & 2 & 1 & 2 & $((\epsilon_k,-\epsilon_k),(\varepsilon_1,\varepsilon_2))$ & $((\epsilon_l,-\epsilon_l),(\varepsilon_1,\varepsilon_2))$ 
    \end{tabular}
    \caption{Types with non-trivial $R(\tau)$ for $\kappa_7$.}
\end{table}
We apply the core fibration \ref{cor:corefib}. In all cases $[F_0']^{\mu_2}$ have been computed in by examples \ref{ex:reg-ss}, \ref{ex:2-2} and corollary \ref{gl}. We have:
\begin{table}[H]
    \centering
    \begin{tabular}{c|c|c}
        $\tau$ & $F_0'(\tau)$ & $[F_0'(\tau)]^{\mu_2}$ \\
        \hline
        $\tau_1$ & $\Per_{\{e\}}^{\mu_2}(\C^\times)\times \Per_{\{e\}}^{\mu_2}(\C^\times)$ & $(q^2-1)^2\otimes \Q^1 - 2(q^2-q)(q-1)\otimes \Q^2$\\
        $\tau_2$ & $\GL_2(\C)\times \GL_2(\C)$ & $(q^2-q)^2(q^2-1)^2\otimes \Q^1$\\
        $\tau_3$ & $\Per_{\{e\}}^{\mu_2}(\GL_2(\C))$ & $(q^4-q^2)(q^4-1)\otimes \Q^1 - q^3(q-1)^2(q+1)^2\otimes \Q^2$ \\
        $\tau_4$ & $ \Per_{\{e\}}^{\mu_2}(\C^\times)\times \Per_{\{e\}}^{\mu_2}(\C^\times)$ & $(q^2-1)^2\otimes \Q^1 - 2(q^2-q)(q-1)\otimes \Q^2$\\
        $\tau_5$ & $ \GL_2(\C)\times \GL_2(\C)$ & $(q^2-q)^2(q^2-1)^2\otimes \Q^1$
    \end{tabular}
    \caption{Isotypic decompositions of $[F_0'(\tau)]$.}
\end{table}

For the base $\mcI(\tau)$ we have, by previously done lower rank computations and Corollary \ref{pfinzq},

\begin{table}[H]
    \centering
    \begin{tabular}{c|c|c}
        $\tau$ & $\mcI(\tau)$ & $[\mcI(\tau)]^{\mu_2}$ \\
        \hline
        $\tau_1,\tau_4$ & $(*)^{4}$ & $1\otimes \Q^1$ \\
        $\tau_2,\tau_5$ & $(R^{irr}_{(\epsilon,-\epsilon)(\varepsilon_1,\varepsilon_2)})^{2}$ & $(q - 1)^2 (q + 1)^2 q^4\otimes\Q^1-2q^2(q + 1)^2(q - 1)^3\otimes\Q^2$ \\
        $\tau_3$ & $\Per_{\{e\}}^{\mu_2}(R^{irr}_{(\epsilon_1,\epsilon_2)(\varepsilon_1,\varepsilon_2)})$ & $ (q^8 - 2q^6 - q^4 + 2q^2)\otimes\Q^1+ (-2q^7 + 2q^6 + 4q^5 - 3q^4 - 2q^3 + q^2)\otimes\Q^2$
    \end{tabular}
    \caption{Isotypic decompositions for $[\mcI(\tau)]$.}
\end{table}

Types $\tau_1$, $\tau_2$, and $\tau_3$ have trivial $\mcM_\tau$. On the other hand, for $\tau_4$,
\begin{align*}
    \mcM_{0}=\Per_{\{e\}}^{\mu_2}(\C^2),&& \ell_{\alpha\beta}=0, && L_\beta=0, && \mcM_\rho=\Per_{\{e\}}^{\mu_2}(\C^2\setminus\{0\})
\end{align*}
and, therefore,
\[[\mcM_\rho]^{\mu_2}=(q^4-1)\otimes\Q^1-(q^2-1) \otimes \Q^2\]
And for $\tau_5$,
\begin{align*}
    \mcM_{0}=M_2(\C),&& \ell_{\alpha\beta}=0, && L_\beta=0, && \mcM_\rho=M_2(\C)\setminus\{0\}
\end{align*}
and, therefore,
\[[\mcM_\rho]^{\mu_2}=(q^4-1)\otimes\Q^1\]

All have trivial $\mcD$ as there is no repeated eigenvalues on $A$. After some straightforward computations,

\begin{table}[H]
    \centering
    \begin{tabular}{c|c}
        $\tau$ & $[R(\tau)]^{\mu_2}$  \\
        \hline
        $\tau_1$ & $\begin{array}{c}
              (q^2+1)(q^4-q)q^2(q^4-q^3)\otimes\Q^1\\+ 2q^7(q^2+1)(q^2+q+1)\otimes\Q^2
        \end{array}$\\
        &\\
        $\tau_2$ & $\begin{array}{c}
            (q - 1)^2 (q + 1)^2 q^8 (q^2 + 1) (q^2 + q + 1)\otimes\Q^1\\-2 (q + 1)^2 (q - 1)^3 q^6 (q^2 + 1) (q^2 + q + 1)\otimes\Q^2
        \end{array}$\\
        &\\
        $\tau_3$ & $\begin{array}{c}
              (q + 1)(q - 1)^3 q^6 (q^2 - 2) (q^2 + 1) (q^2 + q + 1)\otimes\Q^1\\- (q + 1) (q - 1)^2 q^6 (q^2 - q - 1) (q^2 + 1) (q^2 + q + 1)\otimes\Q^2
        \end{array}$ \\
        &\\
        $\tau_4$ & $\begin{array}{c}
              (q + 1) q^6 (q - 1)^3 (q^2 + q + 1) (q^2 + 1)^2\otimes\Q^1\\+ (q + 1) (q - 1)^2 q^6 (q^2 + 1) (2q^2 + q + 1) (q^2 + q + 1)\otimes\Q^2
        \end{array}$\\
        &\\
        $\tau_5$ & $\begin{array}{c}
              (q - 1)^3 (q + 1)^3 q^8 (q^2 + q + 1) (q^2 + 1)^2\otimes\Q^1\\-2 q^6 (q + 1)^3 (q - 1)^4 (q^2 + q + 1) (q^2 + 1)^2\otimes\Q^2
        \end{array}$
    \end{tabular}
    \caption{Isotypic decompositions of $[R(\tau)]$.}
    \label{kappa7}
\end{table}

On the other hand,
\begin{align*}
    [R_{\kappa_7}]^{\mu_2}&=[\GL_4(\C)/(\C^\times)^4]^{\mu_2}([\GL_4(\C)/\GL_2(\C)^2]\otimes\Q^1)\\
    &= ([\GL_4(\C)]^2/[\GL_2(\C)]^2\otimes\Q^1)/([\Per_{\{e\}}^{\mu_2}(\C^\times)]^{\mu_2})^2\\
    &=(q - 1)^2 q^{10} (q^2 + 1)^2 (q^2 + q + 1)^2\otimes\Q^1+ 2 q^{11} (q^2 + 1)^2 (q^2 + q + 1)^2\otimes\Q^2
\end{align*}
by example \ref{ex:reg-ss}. In conclusion,

\begin{align*}
    [R_{\kappa_7}^{irr}]^{\mu_2}&=\frac{1}{2}\Ind_{\{e\}}^{\mu_2}([R_{\kappa_7}^{irr}])\\&-(q + 1)^2 (q - 1)^3 q^6 (q^2 + 1) (q^2 + q + 1) (q^5 + 2q^4 - q^3 + 3q^2 + 2q - 2)\otimes(\Q^1-\frac{1}{2}\Q^2)
\end{align*}

Similar considerations can be made for the other configurations of eigenvalues with $\Gamma_\kappa=\mu_2$. We summarize the results in Tables \ref{kappa5}, \ref{kappa6}, \ref{kappa8}, \ref{kappa9}, \ref{kappa10}, and \ref{resultados}. Let us analyze now the four configurations of eigenvalues with $\mu_4$ stabilizer. In all cases $\mcD$ is trivial. Types with stabilizer $\mu_4$ only appear for $\kappa_3$. Henceforth, for $\kappa_1$, $\kappa_2$ and $\kappa_4$, $R^{red}_{\kappa}$ can be computed by means of the previous tables for $(\epsilon_1,-\epsilon_1,\epsilon_2,-\epsilon_2)$ up to a correction by $\frac{1}{2}\Ind_{\mu_2}^{\mu_4}$, see Lemma \ref{correction-lemma}. The results are shown in Table \ref{kappa-with-mu4}.

\begin{table}[H]
    \centering
    \begin{tabular}{c|c}
    $\kappa$ & $[R_\kappa^{red}]^{\mu_4}-\frac{1}{4}[R_\kappa^{red}]\otimes\Q^{\mu_4}$\\
    \hline
    $\kappa_1$ & $3(q - 1)^2 (q + 1)^2 q^8(q^2 + 1) (q^2 + q + 1)(2q^4 - 1)\otimes(\Q^2-\frac{1}{4}\Q^4)$\\
    $\kappa_2$ & $(q + 1) (q - 1)^2 q^7 (q^2 + 1) (q^2 + q + 1)^2 (2q^4 - 1)\otimes(\Q^2-\frac{1}{4}\Q^4)$\\
    $\kappa_4$ & $(q + 1) (q - 1)^2 q^7 (q^2 + 1) (q^2 + q + 1) (2q^4 - 1)\otimes(\Q^2-\frac{1}{4}\Q^4)$
    \end{tabular}
    \caption{Isotypic decomposition of $[R_\kappa^{red}]$ for $\kappa_1$, $\kappa_2$ and $\kappa_4$.}
    \label{kappa-with-mu4}
\end{table}

To finish let us analyze $\kappa_3$. There is only one type with stabilizer $\mu_4$. Namely $\tau_2'$ with $s=1$, $r_i=2$, $d_{i,j}=2$ and  
\[\kappa_{i,j}= ((\epsilon,-\epsilon),(\varepsilon_1,\varepsilon_2)),((i\epsilon,-i\epsilon),(\varepsilon_1,\varepsilon_2))\]
It has $M_\rho=M_2(\C)\setminus\{0\}$, $F'_0(\tau)=\Per_{\mu_2}^{\mu_4}(\GL_2(\C))$ and $\mcI(\tau)=\Per_{\mu_2}^{\mu_4}(R^{irr}_{((\epsilon,-\epsilon),(\varepsilon_1,\varepsilon_2))})$. We use the formula of Example \ref{binom2to4}:

\begin{align*}
    [F_0'(\tau)]^{\mu_2}&=(q^4-q^2)(q^4-1)\otimes\Q^1+\frac{1}{2}((q^2-q)^2(q^2-1)^2-(q^4-q^2)(q^4-1))\otimes\Q^2\\
    &=(q^4-q^2)(q^4-1)\otimes\Q^1-q^3(q^2-1)^2\otimes\Q^2
\end{align*}

and

\begin{align*}
    [\mcI(\tau)]^{\mu_4}&=\Omega_2((q^2-q)(q^2+q)\otimes\Q^1+(-q^3+q)\otimes \Q^2)\\
    &=(q^4-q^2)(q^4+q^2)\otimes\Q^1 - q^4(q-1)(q+1) \otimes\Q^2 + (-q^3+q)(q^2-q)(q^2-1)\otimes\Q^4
\end{align*}

Hence,

\begin{align*}
    [R(\tau_2')]^{\mu_4}&= (q + 1)^2 (q - 1)^4 q^8 (q^2 + q + 1) (q^2 + 1)^2\otimes\Q^1\\
    &+(q + 1)^2 (q - 1)^3 q^8 (q^2 + q + 1) (q^2 + 1)^2\otimes\Q^2\\
    &- (q + 1)^3 (q - 1)^4 q^6 (q^2 + q + 1) (q^2 + 1)^2)\otimes\Q^4
\end{align*}

The other types for $\kappa_3$ are $\tau_1$, $\tau_3$, $\tau_4$ and $\tau_5$. Their associated computations are in Table \ref{kappa7}. From this point, this configuration of eigenvalues can be treat as before. The final result is in Table \ref{resultados}.

Summing up together all contributions from Table \ref{resultados} we get for odd $m$:

\begin{align*}
    &\frac{[\mcR^{irr}_4(\GL)]}{[\PGL_4(\C)]}- \frac{q-1}{4}[\mcM^{irr}_4(\SL)]=
    \delta_{2|n}\left(-
    \frac{(n-2)(m-1)}{16}(q - 1)   (q^5 + 2q^4 - q^3 + 3q^2 + 2q - 2)\right.\\
    &\left.+\frac{n-2}{8}\binom{m-1}{2}(q - 1) q  (q^4 - 2q^2 - 2)-\frac{(n-2)(m-1)}{8}
     (q - 1) q (3q^2 + 2)\right.\\
    &\left.+\frac{n-2}{32}\binom{m-1}{3} (q - 1) q^2 (q^7 + 2q^6 + 4q^5 + 5q^4 - 6q^3 - 6q^2 - 6q - 6)\right.\\
    &\left.-\frac{1}{2}\binom{m-1}{2}  (q-1)q  (q^2 + q + 3)+\frac{1}{8}\binom{m-1}{3}q^4 (q^2 + q - 5)
    \right)\\ 
    &+\delta_{4|n}\left( \frac{1}{128}\binom{m-1}{3} (q - 1) q^2 (q^7 + 2q^6 - 3q^4 + 6q^3 + 6q^2 + 6q + 6)\right.\\
    &\left.+\frac{1}{16}\binom{m-1}{2} q (3q^7 + 3q^6 + q^4 + 3q^3 - 4q^2 + 4q - 4) \right.\\
    &\left.+\frac{m-1}{32}(q^5 - 3q^4 + 4q^3 + q^2 - 4q + 2) +\frac{m-1}{16}(q - 1)q(3  q^2  +2q^2 + 2) \right).
\end{align*}

\begin{landscape}
\begin{table}[H]
    \centering
    \begin{tabular}{c|c|c|c|c|c|c|c|c}
        $\tau$ & $\#\tau$ & $s$ & $r_i$ & $d_{i,j}$ & $\kappa_{1,j}$ & $\kappa_{2,j}$ & $F_0'(\tau)$ & $[F_0'(\tau)]^{\mu_2}$ \\
        \hline
        $\tau_{11}$ & 6 & 1 & 2 & 2 & $((\epsilon_1,-\epsilon_1),(\varepsilon_i,\varepsilon_j)),((\epsilon_2,-\epsilon_2),(\varepsilon_k,\varepsilon_l))$& - & $\GL_2(\C)\times \GL_2(\C)$ & $(q^2-q)^2(q^2-1)^2\otimes \Q^1$\\
        $\tau_{12}$ & 12 & 2 & 1 & 2 & $((\epsilon_k,-\epsilon_k),(\varepsilon_i,\varepsilon_j))$ & $((\epsilon_l,-\epsilon_l),(\varepsilon_p,\varepsilon_q))$ & $\GL_2(\C)\times \GL_2(\C)$ & $(q^2-q)^2(q^2-1)^2\otimes \Q^1$
    \end{tabular}
    
    \vspace{1cm}

    \begin{tabular}{c|c|c|c|c|c}
        $\tau$ & $\mcI(\tau)$ & $[\mcI(\tau)]^{\mu_2}$ & $M_\tau$ & $[M_\tau]^{\mu_2}$ & $[R(\tau)]^{\mu_2}$\\
        \hline
        $\tau_{11}$ & $(R^{irr}_{(\epsilon,-\epsilon)(\varepsilon_1,\varepsilon_2)})^{2}$ & $\begin{array}{c}(q - 1)^2 (q + 1)^2 q^4\otimes\Q^1\\-2 q^2 (q + 1)^2 (q - 1)^3\otimes\Q^2\end{array}$ & $*$ & $1\otimes\Q^1 $ & $\begin{array}{c} (q - 1)^2 (q + 1)^2 q^8 (q^2 + 1) (q^2 + q + 1)\otimes\Q^1\\-2 (q + 1)^2 (q - 1)^3 q^6 (q^2 + 1) (q^2 + q + 1)\otimes\Q^2\end{array}$\\
        $\tau_{12}$ & $(R^{irr}_{(\epsilon,-\epsilon)(\varepsilon_1,\varepsilon_2)})^{2}$ & $\begin{array}{c}(q - 1)^2 (q + 1)^2 q^4\otimes\Q^1\\-2 q^2 (q + 1)^2 (q - 1)^3\otimes\Q^2\end{array}$ & $M_2(\C)\setminus \{0\}$ & $(q^4-1)\otimes\Q^1$ & $\begin{array}{c}(q - 1)^3 (q + 1)^3 q^8 (q^2 + q + 1) (q^2 + 1)^2\otimes\Q^1\\-2 q^6 (q + 1)^3 (q - 1)^4 (q^2 + q + 1) (q^2 + 1)^2\otimes\Q^2\end{array}$
    \end{tabular}
    \caption{Computations for $\kappa_5$.}
    \label{kappa5}
\end{table}

\begin{table}[H]
    \centering
    \begin{tabular}{c|c|c|c|c|c|c|c|c}
        $\tau$ & $\#\tau$ & $s$ & $r_i$ & $d_{i,j}$ & $\kappa_{1,j}$ & $\kappa_{2,j}$ & $F_0'(\tau)$ & $[F_0'(\tau)]^{\mu_2}$   \\
        \hline
        $\tau_{21}$ & 6 & 1 & 2 & 2 & $((\epsilon,-\epsilon),(\varepsilon_i,\varepsilon_j)),((\epsilon,-\epsilon),(\varepsilon_k,\varepsilon_l))$&-& $\GL_2(\C)\times \GL_2(\C)$ & $(q^2-q)^2(q^2-1)^2\otimes \Q^1$\\
        $\tau_{22}$ & 6 & 2 & 1 & 2 & $((\epsilon,-\epsilon),(\varepsilon_i,\varepsilon_j))$ & $((\epsilon,-\epsilon),(\varepsilon_k,\varepsilon_l))$ & $\GL_2(\C)\times \GL_2(\C)$ & $(q^2-q)^2(q^2-1)^2\otimes \Q^1$
    \end{tabular}

    \vspace{1cm}

    \begin{tabular}{c|c|c|c|c|c}
        $\tau$ & $\mcI(\tau)$ & $[\mcI(\tau)]^{\mu_2}$  & $M_\tau$ & $[M_\tau]^{\mu_2}$ & $[R(\tau)]^{\mu_2}$\\
        \hline
        $\tau_{21}$ & $(R^{irr}_{(\epsilon,-\epsilon)(\varepsilon_1,\varepsilon_2)})^{2}$ & $\begin{array}{c}(q - 1)^2 (q + 1)^2 q^4\otimes\Q^1\\-2 q^2 (q + 1)^2 (q - 1)^3\otimes\Q^2\end{array}$ & $*$ & $1\otimes\Q^1 $ & $\begin{array}{c}(q - 1)^2 (q + 1)^2 q^8 (q^2 + 1) (q^2 + q + 1)\otimes\Q^1\\-2 (q + 1)^2 (q - 1)^3 q^6 (q^2 + 1) (q^2 + q + 1)\otimes \Q^2\end{array}$\\
        &&&&&\\
        $\tau_{22}$ & $(R^{irr}_{(\epsilon,-\epsilon)(\varepsilon_1,\varepsilon_2)})^{2}$ & $\begin{array}{c}(q - 1)^2 (q + 1)^2 q^4\otimes\Q^1\\-2 q^2 (q + 1)^2 (q - 1)^3\otimes\Q^2\end{array}$ & $M_2(\C)\setminus \C^2$ & $(q^4-q^2)\otimes \Q^1$ & $\begin{array}{c}(q - 1)^3 (q + 1)^3 q^{8} (q^2 + 1) (q^2 + q + 1)\otimes\Q^1\\-2 (q + 1)^3 (q - 1)^4 q^6 (q^2 + 1) (q^2 + q + 1)\otimes\Q^2\end{array}$
    \end{tabular}
    \caption{Computations for $\kappa_9$.}
    \label{kappa9}
\end{table}

\begin{table}[H]
    \centering
    \begin{tabular}{c|c|c|c|c|c|c|c|c}
        $\tau$ & $\#\tau$ & $s$ & $r_i$ & $d_{i,j}$ & $\kappa_{1,j}$ & $\kappa_{2,j}$  & $F_0'(\tau)$ & $[F_0'(\tau)]^{\mu_2}$ \\
        \hline
        $\tau_6$ & 2 & 1 & 3 & 1,1,2 & $\begin{array}{c}((\epsilon_i),(\varepsilon_1)),((-\epsilon_i),(\varepsilon_1)),\\((\epsilon_j,-\epsilon_j),(\varepsilon_2,\varepsilon_3))\end{array}$& - &$\Per_{\{e\}}^{\mu_2}(\C^\times)\times \GL_2(\C)$ & $\begin{array}{c}(q^2-1)(q^2-q)(q^2-1)\otimes \Q^1 \\- (q-1)(q^2-q)(q^2-1)\otimes \Q^2\end{array}$\\
        &&&&&&&&\\
        $\tau_7$ & 2 & 1 & 2 & 2 & $\begin{array}{c}((\epsilon_1,-\epsilon_1),(\varepsilon_1,\varepsilon_i)),\\((\epsilon_2,-\epsilon_2),(\varepsilon_1,\varepsilon_j))
        \end{array}$&-& $\GL_2(\C)\times \GL_2(\C)$ & $(q^2-q)^2(q^2-1)^2\otimes \Q^1$\\
        &&&&&&&&\\
        $\tau_8$ & 2 & 2 & 2,1 & 1,2 & $((\epsilon_i),(\varepsilon_1)),((-\epsilon_i),(\varepsilon_1))$ & $((\epsilon_j,-\epsilon_j),(\varepsilon_2,\varepsilon_3))$ & $\Per_{\{e\}}^{\mu_2}(\C^\times)\times \GL_2(\C)$ & $\begin{array}{c}(q^2-1)(q^2-q)(q^2-1)\otimes \Q^1\\ - (q-1)(q^2-q)(q^2-1)\otimes \Q^2\end{array}$ \\
        &&&&&&&&\\
        $\tau_9$ & 2 & 2 & 1,2 & 2,1 & $((\epsilon_j,-\epsilon_j),(\varepsilon_2,\varepsilon_3))$ & $((\epsilon_i),(\varepsilon_1)),((-\epsilon_i),(\varepsilon_1))$& $\GL_2(\C)\times\Per_{\{e\}}^{\mu_2}(\C^\times)$ & $\begin{array}{c}(q^2-1)(q^2-q)(q^2-1)\otimes \Q^1\\ - (q-1)(q^2-q)(q^2-1)\otimes \Q^2\end{array}$ \\
        &&&&&&&&\\
        $\tau_{10}$ & 4 & 2 & 1 & 2 & $((\epsilon_k,-\epsilon_k),(\varepsilon_1,\varepsilon_i))$ & $((\epsilon_l,-\epsilon_l),(\varepsilon_1,\varepsilon_j))$& $\GL_2(\C)\times \GL_2(\C)$ & $(q^2-q)^2(q^2-1)^2\otimes \Q^1$
    \end{tabular}

    \vspace{1cm}
    
    \begin{tabular}{c|c|c|c|c|c}
        $\tau$ & $\mcI(\tau)$ & $[\mcI(\tau)]^{\mu_2}$ & $M_\tau$ & $[M_\tau]^{\mu_2}$ & $[R(\tau)]^{\mu_2}$  \\
        \hline
        $\tau_6$ & $R^{irr}_{(\epsilon,-\epsilon)(\varepsilon_1,\varepsilon_2)}$ & $ \begin{array}{c}(q^2-q)(q^2+q)\otimes\Q^1\\+(-q^3+q)\otimes \Q^2\end{array}$ & $*$ & $1\otimes\Q^1 $ & $\begin{array}{c}(q + 1) (q - 1) q^7 (q^2 + 1) (q^3-1)\otimes\Q^1\\- (q + 1) q^6 (q^2 + 1) (q^3-1)\otimes\Q^2 \end{array}$\\
        &&&&&\\
        $\tau_7$ & $(R^{irr}_{(\epsilon,-\epsilon)(\varepsilon_1,\varepsilon_2)})^{2}$ & $\begin{array}{c}(q - 1)^2 (q + 1)^2 q^4\otimes\Q^1\\-2 q^2 (q + 1)^2 (q - 1)^3\otimes\Q^2\end{array}$ & $*$ & $1\otimes\Q^1 $ & $\begin{array}{c} (q - 1) (q + 1)^2 q^8 (q^2 + 1) (q^3-1)\otimes\Q^1\\-2 (q + 1)^2 (q - 1)^2 q^6 (q^2 + 1) (q^3-1)\otimes\Q^2 \end{array}$\\
        &&&&&\\
        $\tau_8$ & $R^{irr}_{(\epsilon,-\epsilon)(\varepsilon_1,\varepsilon_2)}$ & $\begin{array}{c} (q^2-q)(q^2+q)\otimes\Q^1\\+(-q^3+q)\otimes \Q^2\end{array}$ & $\C^4\setminus \{0\}$ & $(q^4-1)\otimes \Q^1$ & $\begin{array}{c} (q + 1)^2 (q - 1)^2 q^7 (q^3- 1) (q^2 + 1)^2\otimes\Q^1\\- (q - 1) q^6 (q + 1)^2 (q^3- 1) (q^2 + 1)^2\otimes\Q^2 \end{array}$ \\
        &&&&&\\
        $\tau_9$ & $R^{irr}_{(\epsilon,-\epsilon)(\varepsilon_1,\varepsilon_2)}$ & $\begin{array}{c} (q^2-q)(q^2+q)\otimes\Q^1\\+(-q^3+q)\otimes \Q^2\end{array}$ & $\Per_{\{e\}}^{\mu_2}(\C^2\setminus\{0\})$ & $\begin{array}{c}(q^4-1)\otimes \Q^1\\ - (q^2-1)\otimes \Q^2\end{array}$ & $\begin{array}{c} (q + 1)^2 (q - 1)^2 q^7 (q^3- 1) (q^2 + 1)^2\otimes\Q^1\\- (2q + 1) q^6 (q^2 - 1)^2 (q^2 + 1) (q^3-1)\otimes\Q^2 \end{array}$\\
        &&&&&\\
        $\tau_{10}$ & $(R^{irr}_{(\epsilon,-\epsilon)(\varepsilon_1,\varepsilon_2)})^{2}$ & $\begin{array}{c}(q - 1)^2 (q + 1)^2 q^4\otimes\Q^1\\-2 q^2 (q + 1)^2 (q - 1)^3\otimes\Q^2\end{array}$ & $M_2(\C)\setminus \{0\}$ & $(q^4-1)\otimes\Q^1$ & $\begin{array}{c} (q - 1)^2 (q + 1)^3 q^8 (q^3 - 1) (q^2 + 1)^2\otimes\Q^1\\-2 q^6 (q + 1)^3 (q - 1)^3 (q^3- 1) (q^2 + 1)^2\otimes\Q^2 \end{array}$
    \end{tabular}

    \caption{Computations for $\kappa_6$.}
    \label{kappa6}
\end{table}    

\begin{table}[H]
    \centering
    \begin{tabular}{c|c|c|c|c|c|c|c|c}
        $\tau$ & $\#\tau$ & $s$ & $r_i$ & $d_{i,j}$ & $\kappa_{1,j}$ & $\kappa_{2,j}$ & $F_0'(\tau)$ & $[F_0'(\tau)]^{\mu_2}$  \\
        \hline
        $\tau_{13}$ & 2 & 1 & 3 & 1,1,2 & $\begin{array}{c}((\epsilon_i),(\varepsilon_1)),((-\epsilon_i),(\varepsilon_1)),\\((\epsilon_j,-\epsilon_j),(\varepsilon_1,\varepsilon_2))\end{array}$& - & $\Per_{\{e\}}^{\mu_2}(\C^\times)\times \GL_2(\C)$ & $\begin{array}{c} (q^2-1)(q^2-q)(q^2-1)\otimes \Q^1 \\- (q-1)(q^2-q)(q^2-1)\otimes \Q^2\end{array}$\\
        &&&&&&&&\\
        $\tau_{14}$ & 2 & 2 & 2,1 & 1,2 & $((\epsilon_i),(\varepsilon_1)),((-\epsilon_i),(\varepsilon_1))$ & $((\epsilon_j,-\epsilon_j),(\varepsilon_1,\varepsilon_2))$ & $\Per_{\{e\}}^{\mu_2}(\C^\times)\times \GL_2(\C)$ & $\begin{array}{c}(q^2-1)(q^2-q)(q^2-1)\otimes \Q^1 \\- (q-1)(q^2-q)(q^2-1)\otimes \Q^2\end{array}$ \\
        &&&&&&&&\\
        $\tau_{15}$ & 2 & 2 & 1,2 & 2,1 & $((\epsilon_j,-\epsilon_j),(\varepsilon_1,\varepsilon_2))$ & $((\epsilon_i),(\varepsilon_1)),((-\epsilon_i),(\varepsilon_1))$ & $\GL_2(\C)\times\Per_{\{e\}}^{\mu_2}((\C^\times)$ & $\begin{array}{c}(q^2-1)(q^2-q)(q^2-1)\otimes \Q^1 \\- (q-1)(q^2-q)(q^2-1)\otimes \Q^2\end{array}$
    \end{tabular}

    \vspace{1cm}

    \begin{tabular}{c|c|c|c|c|c}
        $\tau$ & $\mcI(\tau)$ & $[\mcI(\tau)]^{\mu_2}$  & $M_\tau$ & $[M_\tau]^{\mu_2}$ & $[R(\tau)]^{\mu_2}$\\
        \hline
        $\tau_{13}$ & $R^{irr}_{(\epsilon,-\epsilon)(\varepsilon_1,\varepsilon_2)}$ & $\begin{array}{c} (q^2-q)(q^2+q)\otimes\Q^1\\+(-q^3+q)\otimes \Q^2\end{array}$ & $*$ & $1\otimes\Q^1 $  & $\begin{array}{c}(q + 1) (q - 1)^2 q^7 (q^2 + 1) (q^2 + q + 1)\otimes\Q^1\\- (q - 1) (q + 1) q^6 (q^2 + 1) (q^2 + q + 1)\otimes\Q^2\end{array}$\\
        &&&&&\\
        $\tau_{14}$ & $R^{irr}_{(\epsilon,-\epsilon)(\varepsilon_1,\varepsilon_2)}$ & $\begin{array}{c} (q^2-q)(q^2+q)\otimes\Q^1\\+(-q^3+q)\otimes \Q^2\end{array}$ & $\C^4\setminus \{0\}$ & $(q^4-1)\otimes \Q^1$ & $\begin{array}{c}(q + 1)^2 (q - 1)^3 q^7 (q^2 + q + 1) (q^2 + 1)^2\otimes\Q^1\\- (q - 1)^2 q^6 (q + 1)^2 (q^2 + q + 1) (q^2 + 1)^2\otimes\Q^2\end{array}$ \\
        &&&&&\\
        $\tau_{15}$ & $R^{irr}_{(\epsilon,-\epsilon)(\varepsilon_1,\varepsilon_2)}$ & $\begin{array}{c} (q^2-q)(q^2+q)\otimes\Q^1\\+(-q^3+q)\otimes \Q^2\end{array}$& $\Per_{\{e\}}^{\mu_2}(\C^2\setminus\{0\})$ & $\begin{array}{c}(q^4-1)\otimes \Q^1\\ - (q^2-1)\otimes \Q^2\end{array}$ & $\begin{array}{c}(q + 1)^2 (q - 1)^3 q^7 (q^2 + q + 1) (q^2 + 1)^2\otimes\Q^1\\- (2q + 1) q^6 (q + 1)^2 (q - 1)^3 (q^2 + 1) (q^2 + q + 1)\otimes\Q^2\end{array}$
    \end{tabular}
    
    \caption{Computations for $\kappa_8$.}
    \label{kappa8}
\end{table}

\begin{table}[H]
    \centering
    \begin{tabular}{c|c|c|c|c|c|c|c|c}
        $\tau$ & $\#\tau$ & $s$ & $r_i$ & $d_{i,j}$ & $\kappa_{1,j}$ & $\kappa_{2,j}$   & $F_0'(\tau)$ & $[F_0'(\tau)]^{\mu_2}$ \\
        \hline
        $\tau_{16}$ & 1 & 1 & 3 & 1,1,2 & $\begin{array}{c}((\epsilon),(\varepsilon_1)),((-\epsilon),(\varepsilon_1)),\\((\epsilon,-\epsilon),(\varepsilon_2,\varepsilon_3))\end{array}$& - & $\Per_{\{e\}}^{\mu_2}(\C^\times)\times \GL_2(\C)$ & $\begin{array}{c}(q^2-1)(q^2-q)(q^2-1)\otimes \Q^1 \\- (q-1)(q^2-q)(q^2-1)\otimes \Q^2\end{array}$\\
        &&&&&&&&\\
        $\tau_{17}$ & 1 & 1 & 2 & 2 & $\begin{array}{c}((\epsilon,-\epsilon),(\varepsilon_1,\varepsilon_2)),\\((\epsilon,-\epsilon),(\varepsilon_1,\varepsilon_3))\end{array}$&-& $\GL_2(\C)\times \GL_2(\C)$ & $(q^2-q)^2(q^2-1)^2\otimes \Q^1$\\
        &&&&&&&&\\
        $\tau_{18}$ & 1 & 2 & 2,1 & 1,2 & $((\epsilon),(\varepsilon_1)),((-\epsilon),(\varepsilon_1))$ & $((\epsilon,-\epsilon),(\varepsilon_2,\varepsilon_3))$ & $\Per_{\{e\}}^{\mu_2}(\C^\times)\times \GL_2(\C)$ & $\begin{array}{c}(q^2-1)(q^2-q)(q^2-1)\otimes \Q^1 \\- (q-1)(q^2-q)(q^2-1)\otimes \Q^2\end{array}$\\
        &&&&&&&&\\
        $\tau_{19}$ & 1 & 2 & 1,2 & 2,1 & $((\epsilon,-\epsilon),(\varepsilon_2,\varepsilon_3))$ & $((\epsilon),(\varepsilon_1)),((-\epsilon),(\varepsilon_1))$ & $\GL_2(\C)\times\Per_{\{e\}}^{\mu_2}(\C^\times)$ & $\begin{array}{c}(q^2-1)(q^2-q)(q^2-1)\otimes \Q^1 \\- (q-1)(q^2-q)(q^2-1)\otimes \Q^2\end{array}$\\
        &&&&&&&&\\
        $\tau_{20}$ & 2 & 2 & 1 & 2 & $((\epsilon,-\epsilon),(\varepsilon_1,\varepsilon_i))$ & $((\epsilon,-\epsilon),(\varepsilon_1,\varepsilon_j))$& $\GL_2(\C)\times \GL_2(\C)$ & $(q^2-q)^2(q^2-1)^2\otimes \Q^1$
    \end{tabular}

    \vspace{1cm}

    \begin{tabular}{c|c|c|c|c|c}
        $\tau$ & $\mcI(\tau)$ & $[\mcI(\tau)]^{\mu_2}$ & $M_\tau$ & $[M_\tau]^{\mu_2}$ & $[R(\tau)]^{\mu_2}$\\
        \hline
        $\tau_{16}$ & $R^{irr}_{(\epsilon,-\epsilon)(\varepsilon_1,\varepsilon_2)}$ & $\begin{array}{c} (q^2-q)(q^2+q)\otimes\Q^1\\+(-q^3+q)\otimes \Q^2\end{array}$ & $*$ & $1\otimes\Q^1 $ & $\begin{array}{c} (q + 1) (q - 1) q^7 (q^2 + 1) (q^3 - 1)\otimes\Q^1\\-(q + 1)q^6 (q^2 + 1) (q^3 - 1)\otimes \Q^2\end{array}$\\
        &&&&&\\
        $\tau_{17}$ & $(R^{irr}_{(\epsilon,-\epsilon)(\varepsilon_1,\varepsilon_2)})^{2}$ & $\begin{array}{c}(q - 1)^2 (q + 1)^2 q^4\otimes\Q^1\\-2 q^2 (q + 1)^2 (q - 1)^3\otimes\Q^2\end{array}$ & $*$ & $1\otimes\Q^1 $ & $\begin{array}{c}(q - 1)^2 (q + 1) q^8 (q^2 + 1) (q^3 - 1)\otimes\Q^1\\-2(q + 1)^2 (q - 1)^2 q^6 (q^2 + 1) (q^3 - 1)\otimes\Q^2 \end{array}$\\
        &&&&&\\
        $\tau_{18}$ & $R^{irr}_{(\epsilon,-\epsilon)(\varepsilon_1,\varepsilon_2)}$ & $\begin{array}{c} (q^2-q)(q^2+q)\otimes\Q^1\\+(-q^3+q)\otimes \Q^2\end{array}$ & $\Per_{\{e\}}^{\mu_2}(\C^2\setminus \C)$ & $\begin{array}{c}(q^4-q^2)\otimes\Q^1\\+(q^2-q^3)\otimes\Q^2\end{array}$ & $\begin{array}{c} (q + 1)^2 (q - 1)^2 q^7 (q^2 + 1) (q^3 - 1)\otimes\Q^1\\-(q + 1)^2 (q - 1)^2 q^6 (q^2 + 1) (q^3 - 1)\otimes\Q^2 \end{array}$ \\
        &&&&&\\
        $\tau_{19}$ & $R^{irr}_{(\epsilon,-\epsilon)(\varepsilon_1,\varepsilon_2)}$ & $\begin{array}{c} (q^2-q)(q^2+q)\otimes\Q^1\\+(-q^3+q)\otimes \Q^2\end{array}$ & $\Per_{\{e\}}^{\mu_2}(\C^2\setminus\C)$ & $\begin{array}{c}(q^4-q^2)\otimes\Q^1\\+(q^2-q^3)\otimes\Q^2\end{array}$ & $\begin{array}{c} (q + 1)^2 (q - 1)^2 q^7 (q^2 + 1) (q^3 - 1)\otimes\Q^1\\-(q + 1)^2 (q - 1)^2 q^6 (q^2 + 1) (q^3 - 1)\otimes\Q^2 \end{array}$\\
        &&&&&\\
        $\tau_{20}$ & $(R^{irr}_{(\epsilon,-\epsilon)(\varepsilon_1,\varepsilon_2)})^{2}$ & $\begin{array}{c}(q - 1)^2 (q + 1)^2 q^4\otimes\Q^1\\-2 q^2 (q + 1)^2 (q - 1)^3\otimes\Q^2\end{array}$ & $M_2(\C)\setminus \C$  & $(q^4-q)\otimes \Q^1$  & $\begin{array}{c} (q + 1)^2 (q - 1) q^7 (q^2 + 1) (q^3 - 1)^2\otimes\Q^1\\-2 (q + 1)^2 q^5 (q - 1)^2 (q^2 + 1) (q^3 - 1)^2\otimes\Q^2\end{array}$
    \end{tabular}    

    \caption{Computations for $\kappa_{10}$.}
    \label{kappa10}
\end{table}

    \begin{table}[H]
        \centering
        \begin{tabular}{c|c|c}
             $\kappa$ & $[R_\kappa]^{\Gamma_\kappa}$ & $[R_\kappa^{irr}]^{\Gamma_\kappa}-\frac{1}{|\Gamma_\kappa|}\Ind_{\{e\}}^{\Gamma_\kappa}([R_\kappa^{irr}])$  \\
             \hline
             $\kappa_1$ & $\begin{array}{c}(q - 1)^3(q + 1)^3 q^{12} (q^2 + 1) (q^2 + q + 1)^2\otimes\Q^1\\+(q - 1)^2(q + 1)^2q^{12} (q^2 + 1) (q^2 + q + 1)^2\otimes\Q^2\\+ (q + 1)^2 q^{13} (q^2 + 1)^2 (q^2 + q + 1)^2\otimes\Q^4 \end{array}$ & $\begin{array}{c}(q - 1)^3(q + 1)^3 q^{12} (q^2 + 1) (q^2 + q + 1)^2\otimes(\Q^1-\frac{1}{4}\Q^4)\\+(q^2 - 1)^2 q^8 (q^2 + 1) (q^3 - 1) (q^5 + 2q^4 - 3q^3 - 3q^2 - 3q - 3)\otimes(\Q^2-\frac{1}{4}\Q^4) \end{array}$ \\
             &&\\
             $\kappa_2$ &  $\begin{array}{c} (q + 1)^2 (q - 1)^3 q^{11} (q^2 + 1) (q^2 + q + 1)^2\otimes\Q^1\\+(q + 1) (q - 1)^2 q^{11} (q^2 + 1)(q^2 + q + 1)^2\otimes\Q^2\\+ (q + 1) q^{12} (q^2 + 1)^2 (q^2 + q + 1)^2\otimes\Q^4 \end{array}$  &  $\begin{array}{c} (q + 1)^2 (q - 1)^3 q^{11} (q^2 + 1) (q^2 + q + 1)^2\otimes(\Q^1-\frac{1}{4}\Q^4)\\- (q + 1)^2 (q - 1)^3 q^7 (q^2 + 1)^2 (q^2 + q + 1)^2\otimes(\Q^2-\frac{1}{4}\Q^4) \end{array}$\\
             &&\\
             $\kappa_3$ & $\begin{array}{c}(q + 1) (q - 1)^3 q^{10}  (q^2 + 1) (q^2 + q + 1)^2\otimes\Q^1\\+ (q - 1)^2 q^{10} (q^2 + 1) (q^2 + q + 1)^2\otimes\Q^2\\+ q^{11} (q^2 + 1)^2 (q^2 + q + 1)^2\otimes\Q^4\end{array}$ & $\begin{array}{c} - (q^2 - 1) q^6 (q^2 + 1) (q^3 - 1) (q^7 + 2q^6 - q^5 + q^3 - 4q^2 - q + 1)\otimes (\Q^2-\frac{1}{2}\Q^4)\\+( q^{10}  (q^4 - 1) (q^3 - 1)^2- (q + 1)^2 (q - 1)^3 q^8 (q^3 - 1) (q^2 + 1)^2)\otimes(\Q^1-\frac{1}{4}\Q^4)\end{array}$\\
             &&\\
             $\kappa_4$ & $\begin{array}{c} (q + 1)^2 (q - 1)^3 q^9 (q^2 + 1) (q^2 + q + 1)\otimes\Q^1\\+ (q + 1) (q - 1)^2 q^9 (q^2 + 1) (q^2 + q + 1)\otimes\Q^2\\+ (q + 1) q^{10} (q^2 + q + 1) (q^2 + 1)^2\otimes\Q^4 \end{array}$  & $\begin{array}{c} (q + 1)^2 (q - 1)^3 q^9 (q^2 + 1) (q^2 + q + 1)\otimes(\Q^1-\frac{1}{4}\Q^4)\\-(q + 1)^2 (q - 1)^3 q^7 (2q^2 + 1) (q^2 + 1) (q^2 + q + 1)\otimes(\Q^2-\frac{1}{4}\Q^4) \end{array}$\\
             &&\\
             $\kappa_5$ & $\begin{array}{c} (q - 1)^2 (q + 1)^2 q^{12} (q^2 + 1)^2 (q^2 + q + 1)^2\otimes\Q^1\\+ 2 (q + 1)^2 q^{13} (q^2 + 1)^2 (q^2 + q + 1)^2\otimes\Q^2\end{array}$ & $\begin{array}{c}(q + 1)^2 (q - 1)^3 q^8 (q^2 + 1)(q^2 + q + 1)\\(q^7 + 2q^6 + 4q^5 + 5q^4 - 6q^3 - 6q^2 - 6q - 6)\otimes(\Q^1-\frac{1}{2}\Q^2)\end{array}$\\
             &&\\
             $\kappa_6$ & $\begin{array}{c}(q + 1) (q - 1)^2 q^11 (q^2 + 1)^2 (q^2 + q + 1)^2\otimes\Q^1\\+ 2 (q + 1) q^12 (q^2 + 1)^2 (q^2 + q + 1)^2\otimes\Q^2 \end{array}$ & $(q + 1)^2 (q - 1)^3 q^7 (q^2 + 1) (q^2 + q + 1)^2 (q^4 - 2q^2 - 2)\otimes(\Q^1-\frac{1}{2}\Q^2)$\\
             &&\\
             $\kappa_8$ & $\begin{array}{c}(q + 1) (q - 1)^2 q^9 (q^2 + q + 1) (q^2 + 1)^2\otimes\Q^1\\+ 2 (q + 1) q^{10} (q^2 + q + 1) (q^2 + 1)^2\otimes\Q^2\end{array}$ & $- (q + 1)^2  (q - 1)^3 q^7 (3q^2 + 2) (q^2 + 1) (q^2 + q + 1)\otimes(\Q^1-\frac{1}{2}\Q^2)$\\
             &&\\
             $\kappa_9$ & $\begin{array}{c}(q - 1)^2 (q + 1)^2 q^{10} (q^2 + 1) (q^2 + q + 1)^2\otimes\Q^1\\+ (q + 1)^2 q^{11} (q^2 + 1) (q^2 + q + 1)^2\otimes\Q^2\end{array}$ & $(q - 1)^2 (q + 1)^2 q^{10} (q^2 + 1) (q^2 + q - 5) (q^2 + q + 1)\otimes(\Q^1-\frac{1}{2}\Q^2)$\\
             &&\\
             $\kappa_{10}$ & $\begin{array}{c}(q + 1) (q - 1)^2 q^9 (q^2 + 1) (q^2 + q + 1)^2\otimes\Q^1\\+(q + 1) q^{10}(q^2 + 1) (q^2 + q + 1)^2\otimes\Q^2\end{array}$ & $- (q + 1)^2 (q - 1)^3 q^7 (q^2 + 1) (q^2 + q + 1) (q^2 + q + 3)\otimes(\Q^1-\frac{1}{2}\Q^2)$
        \end{tabular}
        \caption{Isotypic decompositions of $[R_\kappa]$ and $[R_\kappa^{irr}]$.}
        \label{resultados}
    \end{table}
\end{landscape}

\printbibliography

@article {CS,
    AUTHOR = {Culler, Marc and Shalen, Peter B.},
     TITLE = {Varieties of group representations and splittings of
              {$3$}-manifolds},
   JOURNAL = {Ann. of Math. (2)},
  FJOURNAL = {Annals of Mathematics. Second Series},
    VOLUME = {117},
      YEAR = {1983},
    NUMBER = {1},
     PAGES = {109--146},
      ISSN = {0003-486X},
   MRCLASS = {57N10},
  MRNUMBER = {683804},
MRREVIEWER = {G.\ Peter\ Scott},
       DOI = {10.2307/2006973},
       URL = {https://doi.org/10.2307/2006973},
}

@article {FS21,
    AUTHOR = {Florentino, Carlos and Silva, Jaime},
     TITLE = {Hodge-{D}eligne polynomials of character varieties of free
              abelian groups},
   JOURNAL = {Open Math.},
  FJOURNAL = {Open Mathematics},
    VOLUME = {19},
      YEAR = {2021},
    NUMBER = {1},
     PAGES = {338--362},
      ISSN = {2391-5455},
   MRCLASS = {32S35 (14L30 20C30)},
  MRNUMBER = {4261777},
MRREVIEWER = {Mohammad\ Reza\ Rahmati},
       DOI = {10.1515/math-2021-0038},
       URL = {https://doi.org/10.1515/math-2021-0038},
}

@article {free1,
    AUTHOR = {Florentino, Carlos and Lawton, Sean},
     TITLE = {Singularities of free group character varieties},
   JOURNAL = {Pacific J. Math.},
  FJOURNAL = {Pacific Journal of Mathematics},
    VOLUME = {260},
      YEAR = {2012},
    NUMBER = {1},
     PAGES = {149--179},
      ISSN = {0030-8730,1945-5844},
   MRCLASS = {14B05 (14L30 14P10 32S05)},
  MRNUMBER = {3001789},
MRREVIEWER = {Alexandru\ Dimca},
       DOI = {10.2140/pjm.2012.260.149},
       URL = {https://doi.org/10.2140/pjm.2012.260.149},
}

@article {free2,
    AUTHOR = {Florentino, Carlos and Nozad, Azizeh and Zamora, Alfonso},
     TITLE = {Serre polynomials of {$\mathrm{SL}_n$}- and {$\mathrm{PGL}_n$}-character
              varieties of free groups},
   JOURNAL = {J. Geom. Phys.},
  FJOURNAL = {Journal of Geometry and Physics},
    VOLUME = {161},
      YEAR = {2021},
     PAGES = {Paper No. 104008, 21},
      ISSN = {0393-0440,1879-1662},
   MRCLASS = {14M35 (14D20 14L30 32S35)},
  MRNUMBER = {4180103},
MRREVIEWER = {Artem\ A.\ Lopatin},
       DOI = {10.1016/j.geomphys.2020.104008},
       URL = {https://doi.org/10.1016/j.geomphys.2020.104008},
}

@article {free3,
    AUTHOR = {Lawton, Sean},
     TITLE = {Minimal affine coordinates for {$\mathrm{SL}(3,\mathbb C)$}
              character varieties of free groups},
   JOURNAL = {J. Algebra},
  FJOURNAL = {Journal of Algebra},
    VOLUME = {320},
      YEAR = {2008},
    NUMBER = {10},
     PAGES = {3773--3810},
      ISSN = {0021-8693,1090-266X},
   MRCLASS = {20G05 (14L24 16R30)},
  MRNUMBER = {2457722},
MRREVIEWER = {Vesselin\ Drensky},
       DOI = {10.1016/j.jalgebra.2008.06.031},
       URL = {https://doi.org/10.1016/j.jalgebra.2008.06.031},
}

@article {free4,
    AUTHOR = {Lawton, Sean and Mu\~noz, Vicente},
     TITLE = {{$E$}-polynomial of the {$\mathrm{SL}(3,\mathbb{C})$}-character variety of free groups},
   JOURNAL = {Pacific J. Math.},
  FJOURNAL = {Pacific Journal of Mathematics},
    VOLUME = {282},
      YEAR = {2016},
    NUMBER = {1},
     PAGES = {173--202},
      ISSN = {0030-8730,1945-5844},
   MRCLASS = {14L30 (20C15 20E05)},
  MRNUMBER = {3463428},
MRREVIEWER = {Anthony\ Henderson},
       DOI = {10.2140/pjm.2016.282.173},
       URL = {https://doi.org/10.2140/pjm.2016.282.173},
}

@article {PM,
    AUTHOR = {Gonz\'alez-Prieto, \'Angel and Mu\~noz, Vicente},
     TITLE = {Motive of the {$\rm SL_4$}-character variety of torus knots},
   JOURNAL = {J. Algebra},
  FJOURNAL = {Journal of Algebra},
    VOLUME = {610},
      YEAR = {2022},
     PAGES = {852--895},
      ISSN = {0021-8693,1090-266X},
   MRCLASS = {14M35 (14C15 14D20 57K31)},
  MRNUMBER = {4475981},
MRREVIEWER = {Haimiao\ Chen},
       DOI = {10.1016/j.jalgebra.2022.06.008},
       URL = {https://doi.org/10.1016/j.jalgebra.2022.06.008},
}

@unpublished{TQFT,
author = {Gonz\'alez-Prieto, \'Angel and Hablicsek, M\'arton and Vogel, Jesse},
year = {2023},
month = {09},
pages = {},
title = {Arithmetic-{G}eometric Correspondence of Character Stacks
via Topological Quantum Field Theory},
note = {Preprint},
doi = {10.48550/arXiv.2309.15331},
}

@article {GP,
    AUTHOR = {Gonz\'alez-Prieto, \'Angel},
     TITLE = {Virtual classes of parabolic {${\rm SL}_2(\mathbb C)$}-character
              varieties},
   JOURNAL = {Adv. Math.},
  FJOURNAL = {Advances in Mathematics},
    VOLUME = {368},
      YEAR = {2020},
     PAGES = {107148, 41},
      ISSN = {0001-8708,1090-2082},
   MRCLASS = {14C35 (14C30 14D21 14L24 57R56)},
  MRNUMBER = {4082993},
MRREVIEWER = {Andr\'e\ Oliveira},
       DOI = {10.1016/j.aim.2020.107148},
       URL = {https://doi.org/10.1016/j.aim.2020.107148},
}

@article {GP2,
    AUTHOR = {Gonz\'alez-Prieto, \'Angel},
     TITLE = {Pseudo-quotients of algebraic actions and their application to
              character varieties},
   JOURNAL = {Commun. Contemp. Math.},
  FJOURNAL = {Communications in Contemporary Mathematics},
    VOLUME = {26},
      YEAR = {2024},
    NUMBER = {4},
     PAGES = {Paper No. 2350009, 48},
      ISSN = {0219-1997,1793-6683},
   MRCLASS = {14L24 (14C25 20G05)},
  MRNUMBER = {4730612},
       DOI = {10.1142/S0219199723500098},
       URL = {https://doi.org/10.1142/S0219199723500098},
}

@article {GWZ,
    AUTHOR = {Groechenig, Michael and Wyss, Dimitri and Ziegler, Paul},
     TITLE = {Mirror symmetry for moduli spaces of {H}iggs bundles via
              p-adic integration},
   JOURNAL = {Invent. Math.},
  FJOURNAL = {Inventiones Mathematicae},
    VOLUME = {221},
      YEAR = {2020},
    NUMBER = {2},
     PAGES = {505--596},
      ISSN = {0020-9910,1432-1297},
   MRCLASS = {14J33 (11S80 14H60)},
  MRNUMBER = {4121158},
MRREVIEWER = {Helge\ Ruddat},
       DOI = {10.1007/s00222-020-00957-8},
       URL = {https://doi.org/10.1007/s00222-020-00957-8},
}

@book {Hall,
    AUTHOR = {Hall, Jr., Marshall},
     TITLE = {Combinatorial theory},
 PUBLISHER = {Blaisdell Publishing Co. Ginn and Co., Waltham,
              Mass.-Toronto, Ont.-London},
      YEAR = {1967},
     PAGES = {x+310},
   MRCLASS = {05.00},
  MRNUMBER = {224481},
MRREVIEWER = {H.\ J.\ Ryser},
}

@article {HT-TMS,
    AUTHOR = {Hausel, Tam\'{a}s and Thaddeus, Michael},
     TITLE = {Examples of mirror partners arising from integrable systems},
   JOURNAL = {C. R. Acad. Sci. Paris S\'{e}r. I Math.},
  FJOURNAL = {Comptes Rendus de l'Acad\'{e}mie des Sciences. S\'{e}rie I.
              Math\'{e}matique},
    VOLUME = {333},
      YEAR = {2001},
    NUMBER = {4},
     PAGES = {313--318},
      ISSN = {0764-4442},
   MRCLASS = {14D21 (14J32 32G13 53C26)},
  MRNUMBER = {1854771},
MRREVIEWER = {Bal\'{a}zs\ Szendr\H{o}i},
       DOI = {10.1016/S0764-4442(01)02057-2},
       URL = {https://doi.org/10.1016/S0764-4442(01)02057-2},
}

@unpublished{HMMS-PW,
author = {Hausel, Tamas and Mellit, Anton and Minets, Alexandre and Schiffmann, Oliver},
year = {2022},
month = {09},
pages = {},
title = {{$P = W$} via {$\mathcal{H}_2$}},
note = {Preprint},
doi = {10.48550/arXiv.2209.05429},
}

@article {HRV,
    AUTHOR = {Hausel, Tam\'{a}s and Rodriguez-Villegas, Fernando},
     TITLE = {Mixed {H}odge polynomials of character varieties},
      NOTE = {With an appendix by Nicholas M. Katz},
   JOURNAL = {Invent. Math.},
  FJOURNAL = {Inventiones Mathematicae},
    VOLUME = {174},
      YEAR = {2008},
    NUMBER = {3},
     PAGES = {555--624},
      ISSN = {0020-9910},
   MRCLASS = {14L30 (14D20 14F43)},
  MRNUMBER = {2453601},
MRREVIEWER = {Sean Lawton},
       DOI = {10.1007/s00222-008-0142-x},
       URL = {https://doi.org/10.1007/s00222-008-0142-x},
}

@article {HLRV1,
    AUTHOR = {Hausel, Tam\'{a}s and Letellier, Emmanuel and
              Rodriguez-Villegas, Fernando},
     TITLE = {Arithmetic harmonic analysis on character and quiver
              varieties},
   JOURNAL = {Duke Math. J.},
  FJOURNAL = {Duke Mathematical Journal},
    VOLUME = {160},
      YEAR = {2011},
    NUMBER = {2},
     PAGES = {323--400},
      ISSN = {0012-7094,1547-7398},
   MRCLASS = {14D20 (05E05 14F35 16G10 20C15 20G20)},
  MRNUMBER = {2852119},
MRREVIEWER = {Bogdan\ Ion},
       DOI = {10.1215/00127094-1444258},
       URL = {https://doi.org/10.1215/00127094-1444258},
}

@article {HLRV2,
    AUTHOR = {Hausel, Tam\'{a}s and Letellier, Emmanuel and
              Rodriguez-Villegas, Fernando},
     TITLE = {Arithmetic harmonic analysis on character and quiver varieties
              {II}},
   JOURNAL = {Adv. Math.},
  FJOURNAL = {Advances in Mathematics},
    VOLUME = {234},
      YEAR = {2013},
     PAGES = {85--128},
      ISSN = {0001-8708,1090-2082},
   MRCLASS = {14L30 (05E05 14C05 20G05)},
  MRNUMBER = {3003926},
MRREVIEWER = {Bogdan\ Ion},
       DOI = {10.1016/j.aim.2012.10.009},
       URL = {https://doi.org/10.1016/j.aim.2012.10.009},
}

@article {Hitchin,
    AUTHOR = {Hitchin, N. J.},
     TITLE = {The self-duality equations on a {R}iemann surface},
   JOURNAL = {Proc. London Math. Soc. (3)},
  FJOURNAL = {Proceedings of the London Mathematical Society. Third Series},
    VOLUME = {55},
      YEAR = {1987},
    NUMBER = {1},
     PAGES = {59--126},
      ISSN = {0024-6115,1460-244X},
   MRCLASS = {32G13 (14F05 14H15 32L10 53C05 58E99 81E13)},
  MRNUMBER = {887284},
MRREVIEWER = {Mitsuhiro\ Itoh},
       DOI = {10.1112/plms/s3-55.1.59},
       URL = {https://doi.org/10.1112/plms/s3-55.1.59},
}

@article {su2,
    AUTHOR = {Kitano, Teruaki and Morifuji, Takayuki},
     TITLE = {Twisted {A}lexander polynomials for irreducible {$SL(2,\mathbb
              C)$}-representations of torus knots},
   JOURNAL = {Ann. Sc. Norm. Super. Pisa Cl. Sci. (5)},
  FJOURNAL = {Annali della Scuola Normale Superiore di Pisa. Classe di
              Scienze. Serie V},
    VOLUME = {11},
      YEAR = {2012},
    NUMBER = {2},
     PAGES = {395--406},
      ISSN = {0391-173X,2036-2145},
   MRCLASS = {57M27},
  MRNUMBER = {3011996},
MRREVIEWER = {Heather\ A.\ Dye},
}

@article {LMN,
    AUTHOR = {Logares, Marina and Mu\~noz, Vicente and Newstead, P. E.},
     TITLE = {Hodge polynomials of {$\mathrm{SL}(2,\mathbb{C})$}-{c}haracter varieties for curves of small genus},
   JOURNAL = {Rev. Mat. Complut.},
  FJOURNAL = {Revista Matem\'atica Complutense},
    VOLUME = {26},
      YEAR = {2013},
    NUMBER = {2},
     PAGES = {635--703},
      ISSN = {1139-1138,1988-2807},
   MRCLASS = {14C30 (14D20 14L24 32J25)},
  MRNUMBER = {3068615},
MRREVIEWER = {Sean\ Lawton},
       DOI = {10.1007/s13163-013-0115-5},
       URL = {https://doi.org/10.1007/s13163-013-0115-5},
}

@unpublished{MS-PW,
author = {Maulik, Davesh and Shen, Junliang},
year = {2022},
month = {09},
pages = {},
title = {The {$P=W$} conjecture for {$\mathrm{GL}_n$}},
note = {Preprint},
doi = {10.48550/arXiv.2209.02568},
}

@article {Mellit2,
    AUTHOR = {Mellit, Anton},
     TITLE = {Poincar\'e{} polynomials of character varieties, {M}acdonald
              polynomials and affine {S}pringer fibers},
   JOURNAL = {Ann. of Math. (2)},
  FJOURNAL = {Annals of Mathematics. Second Series},
    VOLUME = {192},
      YEAR = {2020},
    NUMBER = {1},
     PAGES = {165--228},
      ISSN = {0003-486X,1939-8980},
   MRCLASS = {14H60 (14D20 14M35)},
  MRNUMBER = {4125451},
MRREVIEWER = {Francesco\ Bottacin},
       DOI = {10.4007/annals.2020.192.1.3},
       URL = {https://doi.org/10.4007/annals.2020.192.1.3},
}

@article {MM,
    AUTHOR = {Mereb, Martin},
     TITLE = {On the {$E$}-polynomials of a family of
              {$\rm{Sl}_n$}-character varieties},
   JOURNAL = {Math. Ann.},
  FJOURNAL = {Mathematische Annalen},
    VOLUME = {363},
      YEAR = {2015},
    NUMBER = {3-4},
     PAGES = {857--892},
      ISSN = {0025-5831},
   MRCLASS = {14L24 (20G40)},
  MRNUMBER = {3412345},
MRREVIEWER = {Gergely B\'{e}rczi},
       DOI = {10.1007/s00208-015-1183-2},
       URL = {https://doi.org/10.1007/s00208-015-1183-2},
}

@article {sl3,
    AUTHOR = {Mu\~noz, Vicente and Porti, Joan},
     TITLE = {Geometry of the {${\rm SL}(3,\mathbb{C})$}-character variety of
              torus knots},
   JOURNAL = {Algebr. Geom. Topol.},
  FJOURNAL = {Algebraic \& Geometric Topology},
    VOLUME = {16},
      YEAR = {2016},
    NUMBER = {1},
     PAGES = {397--426},
      ISSN = {1472-2747,1472-2739},
   MRCLASS = {14D20 (20F34 57M25 57M27)},
  MRNUMBER = {3470704},
MRREVIEWER = {Haimiao\ Chen},
       DOI = {10.2140/agt.2016.16.397},
       URL = {https://doi.org/10.2140/agt.2016.16.397},
}

@article {DWW,
    AUTHOR = {Daskalopoulos, Georgios D. and Wentworth, Richard A. and
              Wilkin, Graeme},
     TITLE = {Cohomology of {${\rm SL}(2,\mathbb{C})$} character varieties of
              surface groups and the action of the {T}orelli group},
   JOURNAL = {Asian J. Math.},
  FJOURNAL = {Asian Journal of Mathematics},
    VOLUME = {14},
      YEAR = {2010},
    NUMBER = {3},
     PAGES = {359--383},
      ISSN = {1093-6106,1945-0036},
   MRCLASS = {53D30 (53C24 57M50)},
  MRNUMBER = {2755722},
       DOI = {10.4310/AJM.2010.v14.n3.a5},
       URL = {https://doi.org/10.4310/AJM.2010.v14.n3.a5},
}

@article {SimpsonI,
    AUTHOR = {Simpson, Carlos T.},
     TITLE = {Moduli of representations of the fundamental group of a smooth
              projective variety. {I}},
   JOURNAL = {Inst. Hautes \'{E}tudes Sci. Publ. Math.},
  FJOURNAL = {Institut des Hautes \'{E}tudes Scientifiques. Publications
              Math\'{e}matiques},
    NUMBER = {80},
      YEAR = {1994},
     PAGES = {5--79 (1995)},
      ISSN = {0073-8301},
   MRCLASS = {14D20 (14D22 14F05 14F10)},
  MRNUMBER = {1320603},
MRREVIEWER = {Nitin Nitsure},
       URL = {http://www.numdam.org/item?id=PMIHES_1994__80__5_0},
}

@article {SimpsonII,
    AUTHOR = {Simpson, Carlos T.},
     TITLE = {Moduli of representations of the fundamental group of a smooth
              projective variety. {II}},
   JOURNAL = {Inst. Hautes \'{E}tudes Sci. Publ. Math.},
  FJOURNAL = {Institut des Hautes \'{E}tudes Scientifiques. Publications
              Math\'{e}matiques},
    NUMBER = {80},
      YEAR = {1994},
     PAGES = {5--79 (1995)},
      ISSN = {0073-8301},
   MRCLASS = {14D20 (14D22 14F05 14F10)},
  MRNUMBER = {1320603},
MRREVIEWER = {Nitin Nitsure},
       URL = {http://www.numdam.org/item?id=PMIHES_1994__80__5_0},
}

@article {QuiverGrasmaniann,
    AUTHOR = {Schofield, Aidan},
     TITLE = {General representations of quivers},
   JOURNAL = {Proc. London Math. Soc. (3)},
  FJOURNAL = {Proceedings of the London Mathematical Society. Third Series},
    VOLUME = {65},
      YEAR = {1992},
    NUMBER = {1},
     PAGES = {46--64},
      ISSN = {0024-6115,1460-244X},
   MRCLASS = {16G20 (16D70)},
  MRNUMBER = {1162487},
MRREVIEWER = {Zygmunt\ Pogorza\l y},
       DOI = {10.1112/plms/s3-65.1.46},
       URL = {https://doi.org/10.1112/plms/s3-65.1.46},
}

@book {Serre,
    AUTHOR = {Serre, Jean-Pierre},
     TITLE = {Repr\'esentations lin\'eaires des groupes finis},
   EDITION = {revised},
 PUBLISHER = {Hermann, Paris},
      YEAR = {1978},
     PAGES = {182},
      ISBN = {2-7056-5630-8},
   MRCLASS = {20-01 (20C99)},
  MRNUMBER = {543841},
}

@phdthesis{JesseVogel,
    author = {Jesse Vogel},
    title = {Motivic invariants of character stacks},
    school = {Universiteit Leiden},
    year = {2024},
    url = {https://hdl.handle.net/1887/3762962},
}

\end{document}